\documentclass[14pt]{article}
\usepackage[utf8]{inputenc} 
\usepackage{amsfonts} 
\usepackage{amssymb,amsmath} 
\usepackage{amsthm} 
\usepackage{hyperref}
\usepackage[russian,english]{babel}
\usepackage{euscript} 
\usepackage{bm} 
\usepackage{float}

\usepackage{indentfirst}
\usepackage{pb-diagram}
\newtheorem{lemm}{Lemma}
\newtheorem{theorem}{Theorem}
\newtheorem*{theorem*}{Theorem}
\newtheorem{example}{Example}

\newtheorem{definition}{Definition}
\newtheorem{remark}{Remark}
\newtheorem{Relation}{Corollary}
\newtheorem{PREDLOZHENIE}{Proposition}
\usepackage{graphicx}
\graphicspath{{Lagunap/}}
\usepackage{epstopdf}
\usepackage{wrapfig}
\usepackage{amsmath,amscd}
\usepackage[T1,T2A]{fontenc}
\begin{document}
	
	\date{} 
	
	\title{\textbf{RAMIFICATION OF THE VOLUME FUNCTION CLOSE TO ASYMPTOTIC HYPERPLANES. MONODROMY OF BOUNDARY SINGULARITIES}}
	\author{N. M. ARTEMOV}	
	
	\maketitle
	
	\textbf{\abstractname.} It is shown that the discriminants of boundary singularities introduced by V.~I.~Arnold naturally arise in integral geometry. They represent the irregularity sets of the volume function. This function maps an affine hyperplane to the volume of the part of a domain in the Euclidean space  cut off by this hyperplane.
	For isolated boundary singularities, the role of the boundary is shown to be related to the hyperplane at infinity in the projective space. Versions of the Milnor fiber bundle adapted to the present problem are defined and studied. Analogs of the local even Petrovsky class are introduced and computed. It is proved that the volume function is locally finite-sheeted only if these classes are trivial. Such components of the complements to the discriminants (quasilacunas) are listed for simple boundary singularities --- those having no moduli in the classification. The obtained results establish new obstructions to the algebraic integrability of domains in a Euclidean space.

	\tableofcontents

	\section{Introduction}
	
	The {\em volume function} is a function defined on the space of affine hyperplanes. It is equal to the volume of the part cut off by a hyperplane from a domain with a smooth boundary in Euclidean space $\mathbb{R}^{N+1}$. If the domain is bounded, this function is two-valued. Otherwise, it is defined (and single-valued) only if one of the cut-off volumes is finite. The study of this function dates back to the works of Archimedes and Newton (\cite{Arch}, \cite{NEWTON}~Lemma~XXVIII, see also~\cite{300}, \cite{VVA}~Ch.~III). It is closely related to solutions of hyperbolic PDEs (\cite{VVA}~Ch.~V,~VI).
	
	We assume that the boundary of the domain is semi-algebraic. This condition is essential for the problem of local algebraicity of the volume function and is necessary for algebraic integrability in the large (see~\cite{V}). On the other hand, it allows us to ``go into the complex domain'' and apply the methods of Picard--Lefschetz theory. Namely, we study the finite-sheetedness in terms of the finiteness of the orbits of local integration contours under the action of the monodromy group associated with the germ of the boundary.
	
	In suitable local affine coordinates $c, x_1, \dots, x_{N}$ centered at a {\em non-degenerate} (see Sec.~\ref{Pro}) point, the (complexified) boundary hypersurface $\Xi$ is given by the germ $A$ of the graph $c=f(x)$ of a holomorphic and even algebraic function $f$. This function, called the {\em generating function} (\cite{VVA},~Ch.~III,~\S4), has a real isolated singularity (see~\cite{AVG}, Sec.~\ref{kr}) at the origin. The set of hyperplanes in the space $\mathbb{C}^{N+1}$ tangent to $A$ and close to the plane $c=0$ defines the germ of the projectively dual surface $A^*$. It is identified with the local {\em discriminant} (see Sec.~\ref{kr}) $\Sigma_{\mathbb{C}}(\Lambda_f)$ of the linear deformation of the function $f$. The real discriminant $\Sigma_{\mathbb{R}}(\Lambda_f)$ is also known as the {\em wave front} (see~\cite{AVG}, \cite{VVA}).
	
	Let $D\subset\mathbb{C}^{N+1}$ be an $\varepsilon$-ball around the point of tangency  (for sufficiently small $\varepsilon$, see~\cite{VVA},~Ch.~I,~\S3). Any real hyperplane $X$ close to the hyperplane $c=0$ is given by the equation $c=\xi_0+\sum\xi_ix_i$ for a small $\xi=(\xi_0, \xi_1, \dots, \xi_N)\in(\mathbb R^{N+1})^*$. Fix a general plane $X_0$ and equip it with a positively oriented coordinate system $x=(x_1, \dots, x_N)$. The local {\em even Petrovsky class} is the class of the cycle $\Pi=\operatorname{Re}A\cap X_0\cap D$ in the homology group $H_{N-1}(A\cap X_0,\partial D)$. This cycle is canonically cooriented by the gradient of the function $f-\xi_0-\sum\xi_ix_i$. The class $\Pi$ is defined invariantly on a local connected component of the complement to $\Sigma_{\mathbb R}(\Lambda_f)$.
	
	Under the action of the {\em reduction homomorphism} (see Sec.~\ref{kasal}) $H_{N+1}(D,A\cup X_0\cup\partial D)\to H_{N-1}(A\cap X_0,\partial D)$, the class $\Omega$ of any integration contour localized in the ball $D$ (cut off by the hyperplane $X_0$ on one side from an arbitrary connected component of the set $\mathbb{R}^{N+1}\setminus\operatorname{Re}\Xi$), taken with an appropriate orientation, maps to the class $\Pi$. Let $B$ be a $\delta$-ball $(\delta\ll\varepsilon)$ around zero in the parameter space of the deformation $\Lambda_f$. Define the germ of the volume function $V(X)=\int_{\Omega(X)}\omega$ near $X_0$, where $\omega=dc\wedge dx_1\wedge\dots\wedge dx_{N}$. The ramification of its analytic continuation is called {\em locally finite-sheeted} if the integral over all possible variations of the cycle $\Omega$ takes only a finite number of values under the action of the fundamental group $\pi_1(B\setminus\Sigma_{\mathbb C}(\Lambda_f);X_0)$. Since the reduction homomorphism commutes with monodromy, the study of the variations of the cycle $\Omega$ reduces to the identical problem for $\Pi$.
	
	\begin{theorem*}\label{osno}(see Sec.~\ref{kasal}) Let either $N$ be even, or $N$ be odd and both inertia indices of the quadratic intersection form on the group $H_{N-1}(A\cap X)$ be non-zero. Then $\Pi\neq0\in H_{N-1}(A\cap X,\partial D)$ if and only if the function $V$ is infinitely ramified in $B$.
	\end{theorem*}
	
	Any small perturbation (including a linear one) of the germ $f$ by a diffeomorphism of the argument space (depending on a parameter) can be reduced to a function germ contained in a fixed {\em versal} deformation (\cite{AVG},~\S8) of the germ $f$. The computation of the Petrovsky classes and the enumeration of their triviality components (called \textit{formal lacunas}) for the versal deformations of many isolated singularities were carried out by V.~A.~Vassiliev in~\cite{VVA}, Ch.~VI,~\S6. The absence of formal lacunas in the base of a versal deformation implies their absence in the base of any arbitrary deformation, which yields an obstruction (in the specified cases) to the local finite-sheetedness (and all the more so to the algebraicity) of the volume function. On the other hand, if the linear deformation is versal and there is a formal lacuna, then the analytic continuation from the side of this lacuna is locally regular.
	
	The present work is devoted to a similar study close to hyperplanes that are asymptotic with respect to the boundary of the domain. It dates back to the {\em boundary} singularities introduced by V.~I.~Arnold in\(~\cite{ARN},\) where the role of the boundary is related to the ``hyperplane at infinity'' $S$ in the projective closure $\mathbb{CP}^{N+1}$ of the space $\mathbb{C}^{N+1}$. A hyperplane $Y\subset\mathbb C^{N+1}$ is called {\em asymptotic} to an algebraic hypersurface $\Xi\subset\mathbb{C}^{N+1}$ if its projective closure is non-transversal to the part at infinity \(\overline\Xi\cap S\). First and foremost, we are interested in real asymptotic hyperplanes arising in the case of unbounded domains in the space \(\mathbb{R}^{N+1}\).
	
	As in the case of tangency, near a {\em simple} (see Sec.~\ref{Pro}) point $a \in \overline{\Xi} \cap S$, the pair $(\overline{\Xi}, \overline{\Xi} \cap S)$ is given by the germ $(A, A \cap S)$ of the graphs of the generating function $f$ and its restriction to the boundary $\hat{S} = Y \cap S$; moreover, $f$ has an isolated real boundary singularity at the point $a$. In adapted affine coordinates, the equations $c=f(x)$, $c=0$, and $x_1=0$ define $A$, $Y$, and $S$, respectively. Here, the set $\Sigma(\Lambda_f)$ has two components: the hyperplane section $A \cap X$ (for $X$ close to $Y$) can be either non-smooth or non-transversal to the boundary $\hat{S} = X \cap S$.
	
	For generic $X=X_0$ an analogous homomorphism $H_{N+1}(D\setminus S, A\cup X_0\cup\partial D\setminus S)\to H_{N-1}(A\cap X_0\setminus S, \partial D\setminus S)$ is an isomorphism (see Sec.~\ref{j}). We introduce a set of classes $\{\Omega_i\}\subset H_{N+1}(D\setminus S, A\cup X_0\cup\partial D\setminus S)$, see Fig.~\ref{F17}. Each class is represented by a relative cycle given by a connected domain in the space $\mathbb{R}^{N+1}\cap D$ that is not adjacent to $S$. This domain is bounded by the sphere $\partial D$, the hyperplane $X_0$, and the hypersurface $\operatorname{Re}A$, and it is equipped with the orientation induced by a fixed positive orientation of $\mathbb{R}^{N+1}$. Let us denote by $\Pi_i$ the image of $\Omega_i$ under the specified isomorphism. The cycle $\Pi_i$ is realized by a set of connected components of the set $\operatorname{Re}A\cap X_0$ that do not intersect the boundary $S$.
	
	Consider the set of germs of the volume functions near $X_0$ defined by the integrals $V_i(X)=\int_{\Omega_i(X)}\omega$ of a rational differential form $\omega=dc\wedge\dfrac{dx_1}{x_1^{N+2}}\wedge dx_2\wedge\dots\wedge dx_{N}$ on $\mathbb{CP}^{N+1}.$ This form is proportional to the standard holomorphic volume form in the affine chart $\mathbb{C}^{N+1}\supset\Xi$, for which the hyperplane $S:x_1=0$ is at infinity. Let $B\subset(\mathbb{CP}^{N+1})^*$ be as before a $\delta$-ball around zero in the parameter space of the deformation $\Lambda_f$.
	
	\begin{theorem*}\label{1!}(see Sec.~\ref{osn}) a) Let $N$ be odd and the germ $f$ be stable (see Sec.~\ref{Pro}). Then each function $V_i$ is infinitely ramified in $B$. b) Let $N$ be even. If $\Pi_{i}\neq 0\in H_{N-1}(A\cap X, \partial D)$, then the function $V_i$ is infinitely ramified in $B$. If the singularity $f|_{\hat S}$ is simple, the converse is also true; otherwise, the converse, generally speaking, is not true.
	\end{theorem*}
	
	If a linear deformation of a boundary singularity germ $f$ is $\mathcal{K}_{\hat S}$-versal (see Sec.~\ref{kr}), such a germ (together with the wave front germ) is called {\em stable}. Geometric stability criteria for simple real boundary singularities are obtained in Sec.~\ref{simplust}. Stability guarantees that the integral of the volume form over any vanishing integration contour (which form a distinguished basis of the group $H_{N+1}(D\setminus S, A\cup X\setminus S)$, see Sec.~\ref{j}) is not identically equal to zero, see Lemma~\ref{Vol1}.
	
	A connected component of the set $\operatorname{Re}B\setminus\Sigma_{\mathbb R}(\Lambda_f)$ is called a {\em quasilacuna} if all functions $V_i$ defined for it have finite ramification in $B$. We enumerate (Sec.~\ref{quasisimpl}) the quasilacunas for versal deformations of simple boundary singularities.
	
	In the first half of the work, we study the monodromy of isolated boundary singularities with a {\em non-proper boundary}. The role of the Milnor fiber is played by the set $A\cap X\setminus S$. We construct a distinguished basis of the group $H_{N-1}(A\cap X\setminus S)$, which consists of vanishing cycles of the group $H_{N-1}(A\cap X)$ and Leray tubes around vanishing cycles of the group $H_{N-2}(A\cap X\cap S)$, and compute the homotopy type of the space $A\cap X\setminus S$, see Lemma~\ref{baza}. A {\em form of generalized intersections} $[\cdot,\cdot]$ is introduced on the group $H_{N-1}(A\cap X\setminus S)$, and a variant of the Picard--Lefschetz formula \eqref{PL} is obtained:
	$$
	\gamma_i(\alpha_j)=\alpha_j+(-1)^{\frac{N(N+1)}{2}}[\alpha_j,\alpha_i]\alpha_i,
	$$
	where $\gamma_i$ is the monodromy operator corresponding to the basis element $\alpha_i\in H_{N-1}(A\cap X\setminus S)$. The monodromy group, defined by the image of the representation $\pi_1(B\setminus\Sigma)\to O(H_{N-1}(A\cap X\setminus S),[\cdot,\cdot])$, is finite and non-trivial only for even $N$ in two cases: 1) for the singularity $C_k$ (see~\cite{ARN}), and 2) for a simple purely boundary singularity (Lemma~\ref{monodromia}). It is shown (Sec.~\ref{MP}, Lemma~\ref{hard}) that the matrix of generalized intersections of the distinguished basis of the group $H_{N-1}(A\cap X\setminus S)$ is reconstructed from the intersection matrix of the {\em Gabrielov antisymmetric basis} (\cite{Ga}, Theorem 3) defined by the same system of paths. For arbitrary distinguished bases, this statement holds modulo two (Prop.~\ref{pepe}). The introduced matrices are computed for simple singularities (Lemma~\ref{BCF}). In Sec.~\ref{Dnk}, Lemma~\ref{Lyas} on the indecomposability of the monodromy representation is proved, which plays a key role in the proof of the main theorem in Sec.~\ref{osn}.
	
	In Sec.~\ref{Petrov}, we study the generalized even Petrovsky cycle $\Pi$, defined under the condition $\operatorname{Re}A\cap X\cap S=\varnothing$. Its class in the homology group $H_{N-1}(A\cap X\setminus S, \partial D\setminus S)$ is computed in terms of the matrix of generalized intersections. This cycle is directly related to the problem of the ramification of the total cut-off volume (lying on one side of the cutting hyperplane) from an unbounded domain in $\mathbb{R}^{N+1}$. The work concludes with Theorems~\ref{Vasvas} and \ref{Vas} (an analog of Theorem 1 on p. 188 of~\cite{VVA} for simple points) on topological obstructions to the algebraic integrability of domains, which follow from the main results.
	
	{\em Remarks.} The symbol $\Box$ denotes the end of a proof. If we refer to a ``topological pair'' of spaces $(A,B)$ where $B\subset A$ does not hold, it means the pair $(A, A\cap B)$.
	
	\section{Boundary singularities and their deformations}\label{kr}
	
	A holomorphic function $f:\mathbb{C}^N\rightarrow\mathbb C$ is called \textit{real} if it takes only real values on the subspace $\mathbb{R}^N\subset\mathbb{C}^N.$ Let $f: (\mathbb{C}^N, \mathbb{R}^N,0) \to (\mathbb{C}, \mathbb{R}, 0)$ be a real holomorphic function, and $\hat{S}\subset\mathbb{C}^N$ be a hyperplane defined by a real homogeneous linear form, called the \textit{boundary}. The germ $f$ is called a \textit{real boundary singularity} if $df|_{\hat{S}}(0)=0$, see~\cite{ARN}. If, in addition, $df(0)\neq0$ holds, then the singularity is called {\em purely boundary}. Two real boundary singularities are called \textit{($\mathcal{R}_{\hat S}$-)equivalent} if they belong to the same orbit of the action of the group of real holomorphic diffeomorphisms that preserve the boundary and the origin. In the absence of the requirement of reality for the function $f$, the boundary $\hat S$, and the corresponding diffeomorphisms, we obtain the definitions for the complex case. In what follows, by a boundary singularity we mean either a complex germ or a real one together with its complexification. A boundary singularity $f$ is called \textit{simple} if a sufficiently small neighborhood of the jet of $f$ in the space of jets of an arbitrarily high order intersects only a finite number of orbits of the specified action. The classification of simple boundary singularities was obtained in the works of V.~I.~Arnold~\cite{ARN}, \cite{Norm}, see also~\cite{Lyash}.
	
	Let $\Phi:(\mathbb{C}^N\times\mathbb{C}^{l},\mathbb{R}^N\times\mathbb{R}^{l},0)\to(\mathbb{C},\mathbb{R},0)$ be an arbitrary real deformation of a boundary singularity $f$. The set of deformation parameters ${\lambda}\in\mathbb{C}^l$ such that the corresponding function $f_{\lambda}\equiv \Phi(\cdot,\lambda)$ or its restriction to the boundary $f_{\lambda}|_{\hat{S}}$ has a critical point with value $0$ is called the \textit{discriminant} $\Sigma(\Phi)$ of the deformation $\Phi$. The set $\Sigma$ consists of two hypersurfaces $\Sigma^{f}$ and $\Sigma^{f|_{\hat{S}}}$ in $\mathbb{C}^{l}$: the zero-level set can be either non-smooth or non-transversal to the boundary. In the case of purely boundary singularities, the set $\Sigma^{f}$ is empty. The set of those $\lambda\in\mathbb{R}^l$ for which the function $f_{\lambda}$ or $f_{\lambda}|_{\hat{S}}$ has real critical points on the zero-level set is called the real discriminant and is denoted by $\Sigma_{\mathbb{R}}$. The set $\Sigma_{\mathbb{R}}$ does not coincide with $\Sigma\cap\mathbb{R}^l$ if, for some $\lambda\in\mathbb{R}^l$, the function $f_{\lambda}$ or $f_{\lambda}|_{\hat{S}}$ has only imaginary critical points with zero value. 
	
	A holomorphic deformation $F:(\mathbb{C}^N\times\mathbb{C}^m,0)\rightarrow(\mathbb{C},0)$ of a boundary singularity germ $f$ is called \textit{($\mathcal{R}_{\hat S}$-)versal} if any holomorphic deformation $\Phi$ of the germ $f$ is equivalent to a deformation induced from $F$ (see~\cite{AVG},~\S8). This means that for any deformation $\Phi:(\mathbb{C}^N\times\mathbb{C}^{l},0)\to(\mathbb{C},0)$ there exist
	
	1) a map $\varphi:(\mathbb{C}^l,0)\rightarrow(\mathbb{C}^m,0)$ holomorphic near zero,
	
	2) a family $\mathcal{G}(x,\lambda):(\mathbb{C}^N\times\mathbb{C}^l,0)\rightarrow(\mathbb{C}^N,0)$, holomorphically depending on the parameter $\lambda\in\mathbb{C}^l$, of boundary-preserving holomorphic diffeomorphisms of the space $\mathbb{C}^N$ defined near the point $0\in\mathbb{C}^N$ that perturb the identity ($\mathcal{G}(\cdot,0)\equiv\operatorname{id}$),
	
	such that
	\begin{equation}\label{VerS}
		\Phi(x,\lambda)=F(\mathcal{G}(x,\lambda),\varphi(\lambda)).
	\end{equation}
	For $\lambda\neq0$, the value $\mathcal{G}(0,\lambda)$ is not necessarily equal to zero, meaning that translations of the origin along the boundary are allowed. A versal deformation with the minimal number of parameters is called \textit{miniversal}.
	
	For real singularities and their deformations, the definition of versality is the same, except that all the maps involved are real.
	
	Two boundary singularities are called {\em $\mathcal{K}_{\hat S}$-equivalent} if they can be obtained from each other by a boundary-preserving holomorphic diffeomorphism of the argument space and multiplication by a holomorphic function that is non-zero at the origin. The definition of a {\em $\mathcal{K}_{\hat S}$-versal} deformation is analogous. It is clear that a versal deformation is also $\mathcal{K}_{\hat S}$-versal.
	
	Let the boundary hyperplane $\hat{S}$ be given by the condition $x_{1}=0$ in the space $\mathbb{C}^N$ with coordinates $x=(x_{1},\dots,x_{N}).$ By virtue of the versality theorem (\cite{AVG},~Sec.~8.3), for any isolated boundary singularity $f$, the miniversal deformation can be chosen in the form
	\begin{equation}\label{miniversal}
		f+\sum_{i=0}^{\mu-1}\lambda_i\varphi_i,
	\end{equation}
	where the classes of functions $\varphi_0\equiv1,\varphi_1,\dots,\varphi_{\mu-1}$ form a basis of the \textit{local algebra} $\mathbb{C}\{x\}_0/\nabla_{f|_{x_1}}$ of the boundary singularity $f$. Here, $\mathbb{C}\{x\}_0$ is the algebra of germs of holomorphic functions at the origin, $$\nabla_{f|_{x_1}}=\left(x_1\frac{\partial f}{\partial x_1},\frac{\partial f}{\partial x_2},\dots,\frac{\partial f}{\partial x_N}\right)$$ is the boundary gradient ideal, and $\mu=\mu(f)+\mu(f|_{\hat S})$ is the \textit{Milnor number} of the boundary singularity $f$ (see~\cite{ARN}; \cite{AVG},~\S 27). The minimal $\mathcal{K}_{\hat S}$-versal deformation can be chosen in the same form, by requiring that the classes of functions $\varphi_i$ (the number of which may be less than $\mu$) form a basis of the vector space $$\mathbb{C}\{x\}_0/(\nabla_{f|_{x_1}},f).$$
	
	In the present work, we will frequently consider deformations $F:(\mathbb{C}^N\times\mathbb{C}^l,\mathbb{R}^N\times\mathbb{R}^l,0)\rightarrow(\mathbb{C},\mathbb{R},0)$ such that
	\begin{equation}\label{obsch}
		\begin{split}
			& 1)\;F(x,\lambda)= F'(x,\lambda_1,\dots,\lambda_{l-1})+\lambda_0 \quad \text{for some subdeformation} \quad F',\\
			& 2)\;\text{for almost all}\; \lambda\in\mathbb{R}^l,\;\text{the functions}\; f_{\lambda}\;\text{and}\; f_{\lambda}|_{\hat{S}}\;\text{are Morse functions}.
		\end{split}
	\end{equation}
	
	Along with each function $f_{\tilde{\lambda}}$, a deformation of the form \eqref{obsch} contains a family of functions $f_{\tilde{\lambda}}-c$, $c\in\mathbb{C}$. By virtue of \eqref{VerS} and \eqref{miniversal}, any deformation is equivalent to a deformation induced from a versal deformation of the form \eqref{obsch}. The linear deformations $f+\sum_{i=1}^{N}\xi_i x_i+\xi_0$ arising in the study of the volume function are automatically deformations of the form \eqref{obsch}.
	
	\section{Topology of a local non-singular level set}
	
	In this section, we define an analog of the Milnor fiber (a local non-singular level set) for an isolated boundary singularity with a non-proper boundary and construct a distinguished basis for its middle-dimensional homology group (Sec.~\ref{otmb}). This is preceded by an introductory part. The Milnor fiber bundle is represented as a bundle of hyperplane sections of a hypersurface germ (Sec.~\ref{Rassl}). Definitions of a distinguished basis for an ordinary singularity and an antisymmetric basis for a boundary singularity are given (Sec.~\ref{Arncyc}), along with related concepts (intersection matrices, Picard--Lefschetz formulas, monodromy groups). 
	
	\subsection{Milnor fiber bundle of a boundary singularity}\label{Rassl}
	Consider a deformation $F$ of the form \eqref{obsch} of an isolated boundary singularity $f$. For a generic small value $\tilde{\lambda}\in\mathbb{C}^l$, the line $L_{\tilde{\lambda}}=\{f_{\tilde{\lambda}}-c \mid c\in\mathbb{C}\}$ transversally intersects the discriminant $\Sigma(F)$ at $\mu=\mu(f)+\mu(f|_{\hat{S}})$ points near the origin of the space $\mathbb{C}^l$.
	
	Let us introduce a surface $A_{\tilde{\lambda}}$ in $\mathbb{C}^{N+1}$ with coordinates $x,c$, given by the equation $c=f_{\tilde{\lambda}}(x)$. Let $X_{c'}$ be the hyperplane $c=c'$ for any $c'\in\mathbb C$, and let $S$ be the hyperplane $x_1=0$. Under the natural projection $p:\mathbb{C}^{N+1}\rightarrow L_{\tilde{\lambda}}$, the sections $A_{\tilde{\lambda}}\cap X_{c'}$ are identified with the level sets $f_{\tilde{\lambda}}=c'$. Moreover, the critical points of the restrictions $p|_{A_{\tilde{\lambda}}}$ and $p|_{A_{\tilde{\lambda}}\cap S}$ correspond to the critical points of the functions $f_{\tilde{\lambda}}$ and $f_{\tilde{\lambda}}|_{\hat{S}}$.
	
	\begin{figure}[h]
		\center{\includegraphics[width=1\linewidth]{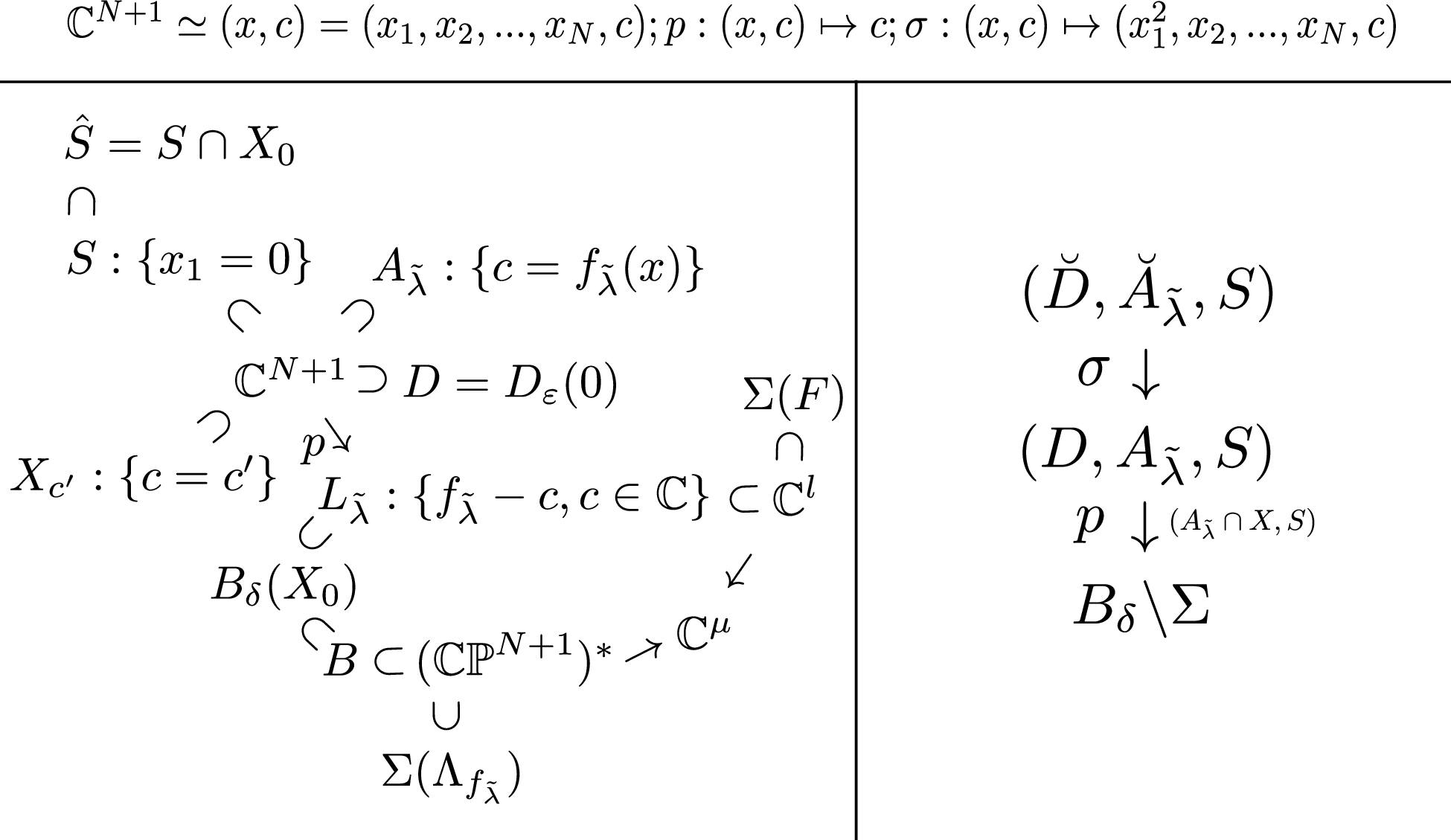}}
		\caption{Diagram of the sets involved}
		\label{F167}	
	\end{figure}
	
	By the curve selection lemma (\cite{MilnOs},~$\S3$), there exists an $\varepsilon>0$ such that the spheres $S^{2N-1}_{r}$ and $S^{2N-1}_{r}\cap\hat{S}$ of radii $r\leqslant\varepsilon$ centered at $0\in\mathbb{C}^N\simeq X_0$ are transversal to the sets $f=0$ and $f|_{\hat{S}}=0$, respectively. Let us denote by $D$ the $\varepsilon$-ball centered at the origin of $\mathbb{C}^{N+1}$. We choose $\delta>0$ such that the spheres $\partial D\cap X_{c'}$ and $\partial D\cap X_{c'}\cap S$ are transversal to the level sets $f=c'$ and $f|_{\hat{S}}=c'$ for $|c'|\leqslant\delta$ (and hence, to the sets $f_{\tilde{\lambda}}=c'$ and $f_{\tilde{\lambda}}|_{\hat{S}}=c'$ for sufficiently small $\tilde{\lambda}$). Let us take $\tilde\lambda$ small enough so that exactly $\mu$ critical values of the restrictions $p|_{A_{\tilde{\lambda}}}$ and $p|_{A_{\tilde{\lambda}}\cap S}$ are contained in the disk $B_{\delta}\subset L_{\tilde{\lambda}}$ defined by the condition $|c|\leqslant\delta$. Then all the critical points of the restrictions $p|_{A_{\tilde{\lambda}}}$ and $p|_{A_{\tilde{\lambda}}\cap S}$ (into which the boundary singularity $f$ has split) are contained inside the ball $D$, and the spheres $\partial D\cap X_{c'}$ and $\partial D\cap X_{c'}\cap S$ are transversal to the sets $A_{\tilde\lambda}\cap X_{c'}$ and $A_{\tilde\lambda}\cap X_{c'}\cap S$ for all $c'\in B_\delta$.
	
	Let us set $A\equiv A_{\tilde{\lambda}}\cap D$. By Thom's isotopy theorem (see~\cite{PHam}), the projection $p$ defines a fiber bundle of pairs $(A\cap X_{c'}, A\cap X_{c'}\cap S)$ of local non-singular level sets (Milnor fibers) of the singularities $f$ and $f|_{\hat{S}}$ over the set of regular (non-critical) values in $B_{\delta}$ (Fig.~\ref{F167}). Other fiber bundles considered in Sec.~\ref{j} are also associated with the projection $p$: the bundles of the sets $A\cap X_{c'}\setminus S$, $A\cap X_{c'}\cap\partial D$, $A\cap X_{c'}\cap\partial D\setminus S$, the pairs $(A\cap X_{c'}\cup\partial D, \partial D)$, $((A\cap X_{c'}\setminus S)\cup\partial D, \partial D)$, $(D\setminus S, A\cup X_{c'}\cup\partial D\setminus S)$, and others. The fiber bundles of the boundaries of the Milnor fibers --- the sets $A\cap X_{c'}\cap\partial D$ and $A\cap X_{c'}\cap S\cap\partial D$ --- are trivializable: due to the transversality of all fibers to the boundary of the ball, they can be extended to fiber bundles over the contractible base $B_{\delta}.$ 
	
	Let us choose a system of smooth paths $\{\gamma_s\}_{s=1}^{\mu(f)}\cup\{\gamma_t\}_{t=1}^{\mu(f|_{\hat S})}$ in $B_{\delta}$ connecting the critical values of the restrictions $p|_{A}$ and $p|_{A\cap S}$ to the point $c=0$ such that any two paths intersect only at the origin. Such a system of paths will be called \textit{distinguished}, see Fig.~\ref{systpath} on p.~\pageref{systpath} (see also~\cite{AVG},~Sec.~23.2; Sec.~1.6, p.~287 of~\cite{VVA}). Being distinguished implies that the paths enter the point $0$ with different tangents and are numbered in decreasing order of the arguments of entry into the origin --- this numbering will be ignored where it is not required. 
	
	Let us set $X\equiv X_{0}$. A distinguished system of paths determines (up to a choice of orientations) the distinguished basic \textit{vanishing} cycles $\{\Delta_s\}_{s=1}^{\mu(f)}$ and $\{\Delta_t'\}_{t=1}^{\mu(f|_{\hat S})}$ of the homology groups $H_{N-1}(A\cap X)$ and $H_{N-2}(A\cap X\cap S)$ of the Milnor fibers of the ordinary singularities $f$ and $f|_{\hat{S}}$, respectively (see~\cite{Br}; \cite{AVG},~\S 24.1). These fibers have the homotopy type of bouquets of $\mu(f)$ and $\mu(f|_{\hat{S}})$ spheres of middle dimensions (\cite{MilnOs},~Theorem~6.5).
	
	\begin{figure}[h]
		\center{\includegraphics[width=1\linewidth]{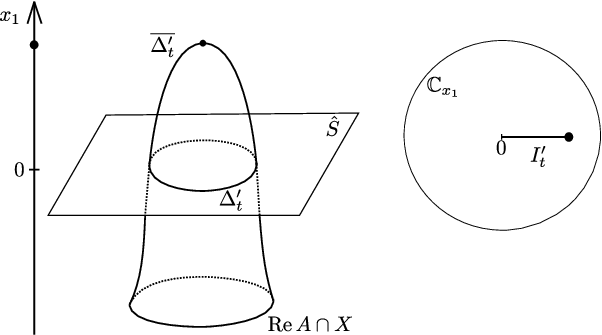}}
		\caption{Definition of the half-cycle $\overline{\Delta_t'}$ by the path $I_t'$. Here, the function $x_1$ has a real critical point with a zero positive Morse index}
		\label{it}	
	\end{figure}
	
	The paths $\{\gamma_s,\gamma_t\}$ also determine a distinguished basis of the homology group $H_{N-1}(A\cap X, A\cap X\cap S)\simeq\mathbb Z^{\mu}$. It is realized by the same vanishing cycles $\{\Delta_s\}$ and by the vanishing relative cycles (``half-cycle'') $\{\overline{\Delta_t'}\}$ with boundaries $\{\Delta_t'\}$. The quotient space $A\cap X / A\cap X\cap S$ is homotopically equivalent to a bouquet of $\mu$ spheres of dimension $N-1$, see~\cite{AVG},~\S 27.2. The following construction of the half-cycles $\{\overline{\Delta_t'}\}$ will be useful for us. The image of the {\em polar curve} of the germ $f_{\tilde\lambda}$ with respect to $x_1$ (defined in $\mathbb C^N$ by the equations $\frac{\partial f_{\tilde\lambda}}{\partial x_2}=\dots=\frac{\partial f_{\tilde\lambda}}{\partial x_N}=0$, see~\cite{Ga}) under the mapping $(x_1, f_{\tilde\lambda})$ has codimension (and hence, dimension) 1 in the space $\mathbb C^2$ with coordinates $x_1, c$. By a small perturbation of the germ $f_{\tilde\lambda}$ and of the paths, one can always ensure that the function $x_1$ (more precisely, its restriction) on each manifold $A\cap X_{c'}\subset\mathbb C^N$ has only Morse critical points with distinct values, where $c'$ is an arbitrary point of the set $\cup_t\gamma_t$. The half-cycle $\overline{\Delta_t'}\subset A\cap X$ can be defined by a smooth path $I_t'\subset\mathbb C_{x_1}$ (on the complex line with coordinate $x_1$) vanishing along $\gamma_t$ and connecting the critical value of the function $x_1$ on the manifold $A\cap X$ with zero. This path is the result of deforming the trajectory, originating at zero, of the critical value of the function $x_1$ on the fibers $f_{\tilde\lambda}=\nu$, where $\nu$ traces the distinguished path $\gamma_t$. Thus, near the boundary, the relative cycle $\overline{\Delta_t'}$ is identified with $I_t'^0\times\Delta_t'$, where $I_t'^0$ is a finite segment of the path $I_t'$. It is worth bearing in mind that for different $t$, the paths $I_t'$ may have non-zero non-removable by homotopy intersection points (see Example~\ref{x9}).
	
	Let us fix a distinguished basis $\{\Delta_s\}$ of the homology group $H_{N-1}(A\cap X)$. The {\em intersection matrix} $\langle\cdot,\cdot\rangle$ of this basis is a $\mu(f)\times\mu(f)$ matrix consisting of the intersection numbers (determined by the complex orientation of the fiber) of the vanishing cycles. The corresponding bilinear form is symmetric for odd $N$ and skew-symmetric for even $N$.
	
	The generators of the group $\pi_{1}(B_\delta\setminus\Sigma^{f})$ act on the group $H_{N-1}(A\cap X)$ according to the {\em Picard--Lefschetz formula} (\cite{AVG},~\S 23.3):
	\begin{equation}\label{PicL}
		\gamma_i(\Delta_j)=\Delta_j+(-1)^{\frac{N(N+1)}{2}}\langle\Delta_j,\Delta_i\rangle\Delta_i,
	\end{equation} 
	where $\gamma_i$ is the {\em monodromy operator} defined by a simple loop around the critical value for the vanishing cycle $\Delta_i$. The image of the representation $$\pi_{1}(B_\delta\setminus\Sigma^{f})\to O(H_{N-1}(A\cap X),\langle\cdot,\cdot\rangle)$$ is called the {\em monodromy group} of the isolated singularity $f$. The transition from one distinguished basis to another is carried out by the Picard--Lefschetz transformations (\cite{AVG},~\S 24.6), and the corresponding monodromy groups are conjugate in the group $O(H_{N-1}(A\cap X),\langle\cdot,\cdot\rangle).$ The monodromy group is finite only in the case of odd $N$ for simple singularities (it is isomorphic to one of the Weyl groups $A_k, D_k, E_6, E_7, E_8$, see~\cite{Norm}) and in the case of a Morse singularity for even $N$ (it is trivial).
	
	Let $N\equiv3\pmod4.$ In this case, the intersection matrix is symmetric, and all its diagonal elements are equal to $-2$. A singularity is called {\em elliptic} if the corresponding quadratic form is negative definite, {\em parabolic} if it is semi-definite, and {\em hyperbolic} if the positive inertia index is equal to one. Elliptic singularities are precisely the simple singularities (see~\cite{Norm}).
	
	The {\em Dynkin diagram} of a basis $\{\Delta_s\}_{s=1}^{\mu(f)}$ is a graph whose vertices are numbered from $1$ to $\mu(f)$, and whose edges are constructed according to the rule:
	the vertices $i$ and $j$ are connected by $r$ solid edges if $\langle\Delta_{i},\Delta_{j}\rangle=r>0$ and by $-r$ dashed edges if $r<0$.
	
	The Dynkin diagram is connected, as was shown by A.~M.~Gabrielov in~\cite{Gabr}, see also~\cite{AVG}, Theorem 4 of Sec.~25.2.
	
	For an arbitrary $N$, a graph on $\mu(f)$ vertices, where two vertices with indices $i$ and $j$ are connected by a single solid edge if and only if $\langle\Delta_i,\Delta_j\rangle\neq 0$, will be called the {\em support} of the Dynkin diagram.
	
	Thus, the distinguished bases $\{\Delta_{s}\}_{s=1}^{\mu(f)}\subset H_{N-1}(A\cap X)$ and $\{\Delta_{t}'\}_{t=1}^{\mu(f|_{\hat{S}})}\subset H_{N-2}(A\cap X\cap S)$ correspond to the intersection matrices and monodromy groups computed in these bases. By means of the following constructions (Sec.~\ref{Arncyc}, \ref{otmb}), integer lattices of rank $\mu=\mu(f)+\mu(f|_{\hat{S}})$ and their distinguished bases are introduced, for which one can define invariant intersection forms and analogs of formula \eqref{PicL}.
	
	\subsection{The group of antisymmetric cycles (after V.~I.~Arnold)}\label{Arncyc}
	Let $\sigma:\mathbb{C}^N\rightarrow\mathbb{C}^N$ be the map given by the formula $(x_{1},x_2,\dots,x_{N})\mapsto(x_{1}^2,x_2,\dots,x_{N}).$
	V.~I.~Arnold in~\cite{ARN} considered the local non-singular level set $\breve{A}\cap X=\sigma^{-1}(A\cap X)$ of the singularity $\breve f(x)=f(\sigma x)$, which is a two-sheeted covering of the manifold $A\cap X$ branched along the ``boundary'' $A\cap X\cap S$, and studied the action of the group $\pi_{1}(\Theta\setminus\Sigma)$ on the group $H^-$. The latter group consists of anti-invariant cycles with respect to the action on the homology group $H_{N-1}(\breve{A}\cap X)$ of the natural involution of the manifold $\breve{A}\cap X.$ The involution is induced by the transformation of $\mathbb{C}^N$ given by the formula $\rho:(x_{1},x_2,\dots,x_{N})\mapsto(-x_{1},x_2,\dots,x_{N})$. The distinguished basis of the group $H^-\simeq\mathbb Z^{\mu}$ is realized by \textit{long} and \textit{short} cycles $\beta_s\equiv\Delta_s^{\operatorname{I}}-\Delta_s^{\operatorname{II}}$ and $\beta_t'\equiv\overline{\Delta_t'}^{\operatorname{I}}-\overline{\Delta_t'}^{\operatorname{II}}$, which are appropriately oriented full preimages of the $\mu(f)$ vanishing spheres $\Delta_s$ and $\mu(f|_{\hat{S}})$ vanishing half-spheres $\overline{\Delta_t'}$, realizing the distinguished basis of the homology group $H_{N-1}(A\cap X,S)$. The intersection form on $H^-$ is defined by the restriction to $H^-$ of the standard intersection form on $H_{N-1}(\breve{A}\cap X)$.
	
	Let $\beta\in H^-$. The generators of the group $\pi_{1}(\Theta\setminus\Sigma)$ act on $H^-$ according to the generalized Picard--Lefschetz formulas (\cite{ARN},~Theorem~5):
	\begin{equation}
		\gamma_t(\beta)=\beta+(-1)^{\frac{N(N+1)}{2}}\langle\beta,\beta_t'\rangle\beta_t',
	\end{equation} 
	where $\gamma_t$ is the monodromy operator defined by a simple loop around the critical value for the short cycle $\beta_t'$, and
	\begin{equation}
		\gamma_s(\beta)=\beta+(-1)^{\frac{N(N+1)}{2}}\frac{\langle\beta,\beta_s\rangle}{2}\beta_s
	\end{equation} 
	in the case of the long cycle $\beta_s.$
	
	Let $N\equiv3\pmod4.$ The {\em Dynkin diagram} of a distinguished antisymmetric basis $\{\beta_s,\beta_t'\}$ is a graph constructed in the following way:
	
	1) the number of vertices of the graph is equal to $\mu$;
	
	2) the vertices $i$ and $j$ are connected by $|r|$ solid (if $r>0$) or dashed (if $r<0$) edges, where $r=\langle\beta_i,\beta_j'\rangle$ if one vertex corresponds to a long cycle and the other to a short cycle, and $r=\langle\beta_i',\beta_j'\rangle$ if both vertices correspond to short cycles. Vertices corresponding to two long cycles are similarly connected by $|r|/2$ edges, where $r=\langle\beta_i,\beta_j\rangle$;
	
	3) the edges connecting vertices of different types are equipped with an orientation from long vertices to short vertices.
	
	Arnold's construction is remarkable because the classification of simple boundary singularities with $df(0)=0$ is related to the classification of simple Lie algebras with double bonds in their Dynkin diagrams (see~\cite{ARN}). In particular, for odd $N$, the monodromy group is isomorphic to one of the Weyl groups $B_k, C_k, F_4$ (short and long cycles correspond to short and long roots).
	
	\subsection{Distinguished basis of homology for a Milnor fiber with a non-proper boundary}\label{otmb}
	We will show that the set of vanishing cycles of the group $H_{N-1}(A\cap X)$ and \textit{Leray tubes} (\cite{VVA},~Ch.~I,~Sec.~2.3,~\cite{Lere}) around the vanishing cycles of the group $H_{N-2}(A\cap X\cap S)$ forms a basis of the homology group $H_{N-1}(A\cap X\setminus S; \mathbb Z)$. This basis will be called \textit{distinguished}. The homotopy type of the space $A\cap X\setminus S$ will also be computed. 
	
	\begin{definition}\label{Miln}
		The set $A\cap X\setminus S$ is called the \textbf{Milnor fiber} (a local non-singular level set) \textbf{with a non-proper boundary} of the boundary singularity $f$. 
	\end{definition}
	
	\begin{example}\label{B2} 
		In the case of the singularity $f=x_{1}^2+x_{2}^2+\dots+x_{N}^2$ of type $B_2$ (see~\cite{ARN}) with the boundary $\hat S:x_1=0$, the triple of manifolds $(A_0, S, Y)$, where the hyperplane $Y\subset\mathbb{C}^{N+1}$ is transversal to $A_0$ at the point $0\in\mathbb{C}^{N+1}$ but is tangent to $A_0\cap S$, realizes the standard Pham degeneration $P_3$ (see~\cite{PH}, \cite{VVA},~Ch.~I,~\S9).
	\end{example} 
	
	We will need the following construction. The projection of the Milnor fiber of an isolated singularity $f$ along the boundary $\hat S$ defines a locally trivial fibration over the set of regular points (provided that the singularity $f|_{\hat S}$ is isolated) in the following way. Let us consider the set $\mathfrak{D}\equiv(D\cap S)\times \mathfrak{B}_{\delta'}$ in $\mathbb C^{N+1}$, where $\mathfrak{B}_{\delta'}$ is a $\delta'$-disk around zero on the line $\mathbb C_{x_1}$, and $\delta'$ is chosen such that the manifold $A_0\cap X\cap\{x_1=x_1'\}$ is transversal to $\partial D$ for $x_1'\in \mathfrak{B}_{\delta'}$. We take the parameter $\tilde\lambda$ small enough so that: 1) the manifold $A_{\tilde\lambda}\cap X\cap\{x_1=x_1'\}$ is transversal to $\partial D$ for $x_1'\in \mathfrak{B}_{\delta'}$; 2) the $\mu$ (distinct) critical points of the function $x_1$ on $A_{\tilde\lambda}\cap X$ that tend to zero as $\lambda\to0$ are contained in $\mathfrak{D}$; 3) the manifold $A_{\tilde\lambda}\cap X$ is transversal to all spheres $S^{2N-1}_r$ of radii $\delta'\leqslant r\leqslant\varepsilon$, where $S^{2N-1}_\varepsilon=\partial D\cap X$. Then: 1) the pair $(A_{\tilde\lambda}\cap X, A_{\tilde\lambda}\cap X\cap\partial D)$ is diffeomorphic to the pair $(A_{\tilde\lambda}\cap X\cap \mathfrak{D}, \partial(A_{\tilde\lambda}\cap X\cap \mathfrak{D}))$; the diffeomorphism is realized by the phase flow of a vector field supported in $\delta'\leqslant|x_1|\leqslant\varepsilon$, tangent to $A_{\tilde\lambda}\cap X$ and transversal to the spheres $S^{2N-1}_r$ (see the proof of Theorem 2.10 in~\cite{MilnOs}); 2) the function $x_1$ defines a locally trivial fibration of the set $A_{\tilde\lambda}\cap X\cap\mathfrak{D}$ (with singular fibers removed) over the complement in $\mathfrak{B}_{\delta'}$ to the set of critical values of the function $x_1$. A system of non-intersecting paths connecting these values to zero determines a certain basis of the group $H_{N-1}(A\cap X, A\cap X\cap S)$ which is not distinguished.
	
	The following commutative diagram holds:
	\begin{equation}\label{c'd}
		\begin{CD}
			H_{N-1}(A\cap X\setminus S) @>>> H_{N-1}(A\cap X,S)\\
			@VVV @VVV \\
			H_{N-1}(A\cap X\setminus S,\partial D) @>>>  H_{N-1}(A\cap X,S\cup\partial D)\\
		\end{CD}
	\end{equation}
	The horizontal arrows are given by the inclusion of finite chains into locally finite ones, and the vertical arrows are given by reduction modulo $\partial D$. The groups on the diagonals of the diagram are related by the Poincaré--Lefschetz duality.
	
	\begin{lemm}\label{baza}
		For any isolated boundary singularity $f$ with $N>2$, the set of classes of vanishing cycles $\{\Delta_{s}\}_{s=1}^{\mu(f)}\subset H_{N-1}(A\cap X)$ (which is empty if $df(0)\neq0$) and tubes $\{\tau(\Delta_{t}')\}_{t=1}^{\mu(f|_{\hat{S}})}$ around the vanishing cycles $\{\Delta_{t}'\}_{t=1}^{\mu(f|_{\hat{S}})}\subset H_{N-2}(A\cap X\cap S)$ defined by a distinguished system of paths forms a basis of the group $H_{N-1}(A\cap X\setminus S)\simeq\mathbb{Z}^{\mu}.$ The space $A\cap X\setminus S$ is homotopically equivalent to a bouquet of spheres $S^1 \vee \bigvee_{i=1}^{\mu} S_i^{N-1}$, where $\mu=\mu(f)+\mu(f|_{\hat S})$.
	\end{lemm}
	
	\textbf{Proof.} Consider a fragment of the homological sequence of the pair $(A\cap X, A\cap X\setminus S):$
	\begin{equation}\label{Lere}
		\dots\rightarrow H_{N}(A\cap X,A\cap X\setminus S)\xrightarrow{h_{1}} H_{N-1}(A\cap X\setminus S)\xrightarrow{h_{2}} H_{N-1}(A\cap X)\rightarrow\dots
	\end{equation}
	The group $H_{N}(A\cap X,A\cap X\setminus S)$ is naturally isomorphic to the group $H_{N-2}(A\cap X\cap S).$ Indeed, let us denote by $T_{\varepsilon'}$ the total space of a tubular $\varepsilon'$-neighborhood of the submanifold $A\cap X\cap S$ in $A\cap X$. This fibration is trivializable due to the transversality of $A\cap X$ to $\hat S$. We have $H_{N}(A\cap X,A\cap X\setminus S)\simeq H_{N}(T_{\varepsilon},T_{\varepsilon}\setminus S)\simeq H_{N-2}(A\cap X\cap S),$ where the first isomorphism is excision, and the second is the Thom isomorphism. From this and Milnor's theorem (\cite{MilnOs}, Theorem 6.5), it follows that the extreme homology groups are trivial in all cases except for the fragment under consideration (and two more identical fragments containing the groups $H_1(A\cap X\setminus S)$ and $H_0(A\cap X\setminus S)$). The homomorphism $H_{N-2}(A\cap X\cap S)\rightarrow H_{N-1}(A\cap X\setminus S)$ is given by the Leray tube operator $\tau$, and the $\mu$ classes $\{\tau(\Delta_{t}')\}_{t=1}^{\mu(f|_{\hat{S}})}$ and $\{h_{2}^{-1}(\Delta_{s})\}_{s=1}^{\mu(f)}$ form a basis of the group $H_{N-1}(A\cap X\setminus S)$ with coefficients in a field. 
	
	Let us prove that this set forms an integral basis. First, the group $H_{N-1}(A\cap X\setminus S)$ is torsion-free since the manifold $A\cap X\setminus S$ is homotopic to an $(N-1)$-dimensional CW-complex (\cite{MilnMors},~\S7). Second, for each of the $\mu$ classes $\{\Delta_s,\tau(\Delta_t')\}$, there exists an element of the group $H_{N-1}(A\cap X, S\cup\partial D;\mathbb Z)$ giving an intersection number equal to $1$ with it. For the vanishing cycles \(\Delta_s\), these are the classes of the dual cycles \(\nabla_s\) of the group $H_{N-1}(A\cap X,\partial D)$; for the tubes \(\tau(\Delta'_t)\), we define the elements \(\nabla_t\) as follows: we take the cycle $\nabla_t'\in H_{N-2}(A\cap X\cap S,\partial D)$ dual to $\Delta_t'$ and transport it along an arbitrary path on $\mathbb C_{x_1}$ that bypasses the critical values of the function $x_1$ on the manifold $A\cap X\subset\mathbb C^N$ and connects the zero point ($x_1=0$ defines $\hat S$) with ``infinity''. Since the cycles \(\nabla_s\) can be realized outside $T_{\varepsilon'}$ (see the upper row in \eqref{c'd}), the intersection matrix of the constructed sets has a determinant equal to $\pm1,$ which implies the required result. $\Box$
	
	Now we compute the homotopy type. First, $\pi_1(A\cap X\setminus S)\simeq\mathbb Z$. Indeed, since $\pi_1(A\cap X)=0$ and for $N>2$ the space $A\cap X\cap S$ is path-connected, the fundamental group is generated by a tube around any point of the set $A\cap X\cap S$. From \eqref{Lere}, we have $H_1(A\cap X\setminus S)\simeq\mathbb Z,$ and there are no relations on this generator. Next, we trivialize $T_{\varepsilon'}$ using the coordinate $x_1$, that is, $T_{\varepsilon'}\simeq\{|x_1|\leqslant\varepsilon'\}\times A\cap X\cap S$. A distinguished system of paths determines the mappings of the spheres $\Delta_s: S_s^{N-1}\to A\cap X\setminus S$, the tubes $\tau(\Delta_t'): S^1\times S_t^{N-2}\to A\cap X\setminus S$, where $S^1\equiv\{|x_1|=\varepsilon'\}$, and the disks $\overline{\Delta_t'}:D_t^{N-1}\to A\cap X$, where $\partial D_t^{N-1}=S_t^{N-2}\subset A\cap X\cap S$. Since the relative cycle $D_t^{N-1}$ in the manifold $A\cap X$ can be defined by a path $I_t'\subset\mathbb C_{x_1}$ terminating at zero (see Sec.~\ref{Rassl}), for each tube $\tau(\Delta_t')$, the sphere $\operatorname{pt}(t)\times S_t^{N-2}$, where $\operatorname{pt}(t)=I_t'\cap\{|x_1|=\varepsilon'\}$, is capped by a disk $\mathring{D}_t^{N-1}$ (defined by a path $\mathring{I}_t'$ with a truncated end). This defines the mappings $S_t^{N-1}\to A\cap X\setminus S$ (the disks around the poles of the sphere $S_t^{N-1}$ map to the image of $\mathring{D}_t^{N-1}$, and the sphere itself is stretched along the tube). The union of all these mappings yields the mapping $\kappa:S^1 \vee \bigvee_{i=1}^{\mu} S_i^{N-1}\to A\cap X\setminus S$.
	
	Let us denote by $\widetilde{A\cap X\setminus S}$ the universal covering space of $A\cap X\setminus S.$ Let us emit a ray $l$ from zero to ``infinity'' in $\mathbb C_{x_1}$ that does not pass through the critical values of the function $x_1$ on the manifold $A\cap X.$ Using the construction indicated before the lemma, we identify the pair $(A\cap X, \partial D)$ with its image in $\mathfrak{D}$. We cut the manifold $A\cap X$ along the set $x_1^{-1}(l)\equiv W\simeq A\cap X\cap S\times l\sim A\cap X\cap S$. The resulting manifold with corners $V\sim A\cap X$ has a boundary containing the subset $W^{+}\cup W^-$. Then the space $\widetilde{A\cap X\setminus S}$ can be obtained from the union of $\mathbb Z$ copies $V_k$ of the space $V$ by the identifications of $W_k^+$ and $W_{k+1}^-.$ The corresponding Mayer--Vietoris sequence of reduced homology groups contains a single non-trivial fragment:
	$$0\to \tilde H_{N-1}(A\cap X)\otimes\mathbb Z[\mathbb Z]\to \tilde H_{N-1}(\widetilde{A\cap X\setminus S})\to \tilde H_{N-2}(A\cap X\cap S)\otimes\mathbb Z[\mathbb Z]\to 0.$$
	The map $\kappa$ induces an isomorphism of the fundamental groups and lifts to a map $\tilde\kappa$ of the universal coverings. By virtue of the proof of Lemma~\ref{hard} in Sec.~\ref{MP}, we can assume that the realizations of the cycles $\Delta_s$ do not intersect the set ${x_1}^{-1}(l)$, and the images of $S^{N-1}_t$ intersect it transversally along the spheres $\operatorname{pt}'\times S^{N-2}_t$, where $\operatorname{pt}'=S^1\cap l$. We construct the universal covering of the bouquet $S^1 \vee \bigvee_{i=1}^{\mu} S_i^{N-1}$ by cutting the $\mu(f|_{\hat S})$ spheres along the bouquet of their equators, and consider an analogous exact sequence. By the 5-lemma, the map $\tilde\kappa$ induces an isomorphism of all homology groups, and then, by Whitehead's theorem, $\kappa$ is a homotopy equivalence. $\Box$ 
	
	\begin{remark}\label{N2}
		Let $N=2$. Then ${H_{0}}(A\cap X\cap S)\simeq\mathbb{Z}^{\mu(f|_{\hat{S}})+1}$, and the true basis of the group $H_{1}(A\cap X\setminus S)\simeq\mathbb{Z}^{\mu+1}$ is formed by the classes of the cycles $\{h_{2}^{-1}(\Delta_{s})\}_{s=1}^{\mu(f)}$ and of the tubes around all points of the set $A\cap X\cap S$. By a distinguished basis we mean the set of cycles $\{h_{2}^{-1}(\Delta_{s})\}_{s=1}^{\mu(f)}$ and $\mu(f|_{\hat{S}})$ tubes around the vanishing cycles of the group $\tilde{H}_{0}(A\cap X\cap S)$ (reduced modulo the fundamental cycle). The space $A\cap X\setminus S$ is homotopically equivalent to a bouquet of $\mu+1$ circles (each punctured point from the fiber $A\cap X\sim\bigvee_{i=1}^{\mu(f)}S^{1}_i$ adds a circle to the bouquet).
	\end{remark}
	
	\begin{remark}
		In the case where $df(0)\neq0$ holds, we have $H_{N-1}(A\cap X)=0$, and $h_{1}$ (see \eqref{Lere}) is an isomorphism.
	\end{remark}
	
	\section{The monodromy group of boundary singularities with a non-proper boundary}
	In Sec.~\ref{PLF}, we compute the action of the group $\pi_{1}(B_{\delta}\setminus\Sigma)$ on $H_{N-1}(A\cap X\setminus S)$. Its generators act according to formulas \eqref{pl}--\eqref{yyi}. A {\em form of generalized intersections} $[\cdot,\cdot]$ is introduced, and an analog of the Picard--Lefschetz formula \eqref{PL} is obtained, which unifies these formulas into a single one. It is shown (Sec.~\ref{MP}) that the matrix of generalized intersections is reconstructed from the intersection matrix of the antisymmetric {\em Gabrielov basis} (\cite{Ga}). These matrices are computed for simple singularities (Lemma~\ref{BCF}). In Sec.~\ref{Dnk}, another proof of the connectedness of the Dynkin diagram for a boundary singularity is presented (\cite{Lyash}). 
	
	\subsection{Picard--Lefschetz--Pham formulas}\label{PLF}
	We will derive the formulas by which the generators of the group $\pi_{1}(B_{\delta}\setminus\Sigma)$ act on the distinguished basis of the homology group $H_{N-1}(A\cap X\setminus S)$ introduced in Prop.~\ref{baza}.
	
	Let us fix a distinguished system of paths in $L_{\tilde{\lambda}}$ and define the bases $\{\Delta_{s}\}_{s=1}^{\mu(f)}\subset H_{N-1}(A\cap X)$, $\{\Delta_{t}'\}_{t=1}^{\mu(f|_{\hat{S}})}\subset H_{N-2}(A\cap X\cap S)$, and $\{\Delta_{s}, \tau(\Delta_{t}')\}\subset H_{N-1}(A\cap X\setminus S)$. 
	
	For an ordered pair $(\Delta_{s},\tau(\Delta_{t}'))$, consider the linking number $l(\Delta_{s},\Delta_{t}')$ defined via the {\em localization homomorphism} \eqref{23} in the following way (see~\cite{VVA},~Ch.~II,~Sec.~1.2). 
	
	Let a cycle $\Delta_{t}'\in H_{N-2}(A\cap X\cap S)$, determined by a distinguished path $\gamma_{t}$, vanish at a critical point $x_{\theta'}\in\mathbb{C}^{N+1}$ of the function $p|_{A\cap S}$ with value $\theta'$. Let us denote by $O(x_{\theta'})\subset\mathbb{C}^{N+1}$ a small ball around the point $x_{\theta'}$. Transport an arbitrary basic cycle $\Delta_{s}\in H_{N-1}(A\cap X)$ along the path $\gamma_{t}$ from the distinguished fiber $A\cap X$ to a generic fiber $A\cap X_{\theta'-\varepsilon'}$ sufficiently close to the critical one. The pair $$(O(x_{\theta'})\cap A\cap X_{\theta'-\varepsilon'}, O(x_{\theta'})\cap A\cap X_{\theta'-\varepsilon'}\cap S)$$ is homeomorphic to the pair $$(\mathbb{C}^{N-1}, \text{the Milnor fiber of a Morse function on}\; \mathbb{C}^{N-1})$$ (see Fig.~\ref{u}).
	
	\begin{figure}[H]
		\center{\includegraphics[width=0.75\linewidth]{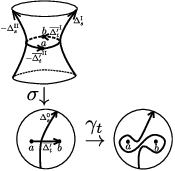}}
		\caption{Example (the case $N=2$): Topological scheme of the branched two-sheeted covering $\sigma$ of the set $O(x_{\theta'})\cap A\cap X_{\theta'-\varepsilon'}\simeq D^2$, given in local Morse coordinates $(x,y)$ on $X_{\theta'-\varepsilon'}=\mathbb{C}^2$ by the equation $x=\varepsilon-y^2$, $\varepsilon>0$. The covering $\sigma$ branches along the boundary $\hat S$ defined by the condition $x=0$. $O(x_{\theta'})\cap A\cap X_{\theta'-\varepsilon'}\cap S=\{a,b\},$ where $b=\sqrt{\varepsilon}$ and $a=-\sqrt{\varepsilon}.$ $\Delta_{t}'=\{b\}-\{a\}$. The localization $\Delta_{s}^0$ of the cycle $\Delta_{s}$ realizes a generator of the homology group $H_1(O(x_{\theta'})\cap A\cap X_{\theta'-\varepsilon'}\setminus S, \partial O(x_{\theta'}))$, and $l(\Delta_s,\Delta_t')=-1$ holds. The covering space is locally given by the equation $x^2=\varepsilon-y^2.$ The monodromy along the loop $\gamma_t$ is realized by a one-parameter family of diffeomorphisms, along whose trajectories the points $a$ and $b$ make a counterclockwise half-turn. Hence, $\operatorname{Var}_{\gamma_{t}}(\Delta_{s}^0)=\tau(b)-\tau(a)=\tau(\Delta_{t}').$}
		\label{u}	
	\end{figure}
	
	Let us denote by $\Delta_{s}^0$ the image of the cycle $\Delta_{s}$ under the homomorphism 
	\begin{equation}\label{23}
		H_{N-1}(A\cap X_{\theta'-\varepsilon'}\setminus S)\rightarrow H_{N-1}(O(x_{\theta'})\cap A\cap X_{\theta'-\varepsilon'}\setminus S, \partial O(x_{\theta'})),	
	\end{equation} 
	induced by the reduction of $D\subset\mathbb{C}^{N+1}$ modulo the complement of $O(x_{\theta'}).$ 
	
	The relative cycle $\overline{\Delta_{t}'}$ of the group $$H_{N-1}(A\cap X_{\theta'-\varepsilon'}, A\cap X_{\theta'-\varepsilon'}\cap S)$$ that vanishes at the point $x_{\theta'}$ is automatically localized in $O(x_{\theta'}).$ Between the groups $$H_{N-1}(O(x_{\theta'})\cap A\cap X_{\theta'-\varepsilon'}\setminus S, \partial O(x_{\theta'}))$$ and $$H_{N-1}(O(x_{\theta'})\cap A\cap X_{\theta'-\varepsilon'}, O(x_{\theta'})\cap A\cap X_{\theta'-\varepsilon'}\cap S)$$ there is a non-degenerate pairing given by the intersection numbers (see~\cite{VVA}, p.~29). We define the orientation of the relative cycle $\overline{\Delta_{t}'}$ from the condition $\partial\overline{\Delta_{t}'}=\Delta_{t}'.$
	
	\begin{definition}\label{chislol}
		$l(\Delta_{s},\Delta_{t}')\equiv\langle\Delta_{s}^0,\overline{\Delta_{t}'}\rangle_{A\cap X\setminus S}$.
	\end{definition} 
	
	Let us compute the action on the cycle $\Delta_s$ of the monodromy operator $\gamma_{t}:H_{N-1}(A\cap X\setminus S)\rightarrow H_{N-1}(A\cap X\setminus S)$, which is defined by a {\em simple loop} around the point $\theta'.$
	It follows from Proposition 6 on p. 29 of~\cite{VVA} that for $N>2$, the group $$H_{N-1}(O(x_{\theta'})\cap A\cap X_{\theta'-\varepsilon'}\setminus S, \partial O(x_{\theta'}))$$ is generated by the Leray tube $\tau(\nabla_{t}')$ around the relative cycle $$\nabla_{t}'\in H_{N-2}(O(x_{\theta'})\cap A\cap X_{\theta'-\varepsilon'}\cap S, \partial O(x_{\theta'})),$$ which is the Poincaré dual to the vanishing cycle $\Delta_{t}'.$ Let $\Delta_{s}^0=n\tau(\nabla_{t}'), n\in\mathbb{Z}.$ Then $$l(\Delta_{s},\Delta_{t}')=\langle\Delta_{s}^0,\overline{\Delta_{t}'}\rangle_{A\cap X\setminus S}=n\langle\tau(\nabla_{t}'),\overline{\Delta_{t}'}\rangle_{A\cap X\setminus S}=n(-1)^{N}\langle\nabla_{t}',\Delta_{t}'\pmb{\rangle}_{A\cap X\cap S}=n(-1)^{N}.$$ 
	The sign $(-1)^N$ arises from the commutativity of diagram I.7 on p. 29 of~\cite{VVA}, see also diagram \eqref{c'd} and the proof of Proposition~\ref{bilin}.
	
	Then, since $\operatorname{Var}_{\gamma_t}\nabla_t'=(-1)^{\frac{(N-1)N}{2}}\Delta_t'$ and the {\em variation} (see \cite{VVA},~Ch.~I,~\S 1.3) commutes with the tube operator, the operator $\gamma_{t}$ acts on the cycle $\Delta_s$ according to the Picard--Lefschetz--Pham formula:
	\begin{equation}\label{pl}
		\gamma_{t}(\Delta_{s})=\Delta_{s}+(-1)^{\frac{N(N-1)}{2}}(-1)^{N} l(\Delta_{s},\Delta_{t}')\tau(\Delta_{t}')=\Delta_{s}+(-1)^{\frac{N(N+1)}{2}} l(\Delta_{s},\Delta_{t}')\tau(\Delta_{t}').
	\end{equation} 
	For $N=2$, formula \eqref{pl} is easily verified by direct computation (see Fig.~\ref{u}).
	
	\begin{remark}\label{rem}
		Since the monodromy can be realized by a diffeomorphism of the pair $(A\cap X, S)$ that is fixed near $\partial D$, the variation operator $\operatorname{Var}_{\gamma_{t}}:H_{N-1}(A\cap X\setminus S, \partial D)\to H_{N-1}(A\cap X\setminus S)$ is defined. The coefficients $l(\nabla, \Delta_{t}')$ for elements $\nabla\in H_{N-1}(A\cap X\setminus S, \partial D)$ are defined analogously to the absolute case. The variation operator acts according to the formula $$\operatorname{Var}_{\gamma_{t}}(\nabla)=(-1)^{\frac{N(N+1)}{2}} l(\nabla,\Delta_{t}')\tau(\Delta_{t}'),$$ and the corresponding monodromy operator \(\gamma_t\) acts on \(\nabla\) as follows: $\gamma_t(\nabla)=\nabla+i_*\operatorname{Var}_{\gamma_{t}}(\nabla),$ where $i_*$ is homomorphism of reduction modulo $\partial D.$ 	
	\end{remark}
	
	From the Picard--Lefschetz formula and from the fact that the monodromy commutes with the tube operator, we obtain the formula:
	\begin{equation}\label{yyj}
		\gamma_{j}(\tau(\Delta_{i}'))=\tau(\Delta_{i}')+(-1)^{\frac{N(N-1)}{2}}\langle\Delta_{i}',\Delta_{j}'\rangle_{A\cap X\cap S} \tau(\Delta_{j}').
	\end{equation}
	
	Let us now consider the operators defined by simple loops around the critical values of the function $p|_{A}$. They act according to the standard Picard--Lefschetz formula. It follows from the definition of the Leray tube that for any vanishing cycle $\Delta_s\in H_{N-1}(A\cap X\setminus S)$, one has $\langle\tau(\Delta_{t}'),\Delta_{s}\rangle_{A\cap X\setminus S} =0$, and the action of the operator $\gamma_{s}$ on the tubes is trivial:
	\begin{equation}\label{yy}
		\gamma_{s}(\tau(\Delta_{t}'))=\tau(\Delta_{t}')+(-1)^{\frac{N(Y+1)}{2}}\langle\tau(\Delta_{t}'),\Delta_{s}\rangle_{A\cap X\setminus S}=\tau(\Delta_{t}').
	\end{equation}
	Finally, 
	\begin{equation}\label{yyi}
		\gamma_{j}(\Delta_{i})=\Delta_{i}+(-1)^{\frac{N(N+1)}{2}}\langle\Delta_{i},\Delta_{j}\rangle_{A\cap X} \Delta_{j}.
	\end{equation}
	
	\begin{definition}\label{intmatr}
		The \textbf{matrix of generalized intersections} of the distinguished basis $\{\Delta_{s}\}_{s=1}^{\mu(f)}, \{\tau(\Delta_{t}')\}_{t=1}^{\mu(f|_{\hat{S}})}$ of the homology group $H_{N-1}(A\cap X\setminus S)$ is the $\mu\times\mu$ matrix whose elements $[\cdot,\cdot]$ are defined by pairs of basis elements as follows: 
		\begin{eqnarray*}
			[\Delta_{i},\Delta_{j}]=\langle\Delta_{i},\Delta_{j}\rangle_{A\cap X}, \quad [\tau(\Delta_{i}'),\tau(\Delta_{j}')]=(-1)^N\langle\Delta_{i}',\Delta_{j}'\rangle_{A\cap X\cap S},\\\label{0}
			[\Delta_{s},\tau(\Delta_{t}')]=l(\Delta_{s},\Delta_{t}'), \quad [\tau(\Delta_{t}'),\Delta_{s}]=\langle\tau(\Delta_{t}'),\Delta_{s}\rangle_{A\cap X\setminus S}=0.	
		\end{eqnarray*}
	\end{definition}
	
	\begin{remark}\label{ed}
		Two different systems of paths can define identical bases of the group $H_{N-1}(A\cap X\setminus S)$ with different matrices of generalized intersections. Suppose (see Example~\ref{ff4}) that for some pair $(\Delta_s,\tau(\Delta_t'))$ one has $l(\Delta_s,\Delta_t')\neq 0$, and the path $\tilde\gamma_t$ differs from the path $\gamma_t$ for $\Delta_t'$ by a simple turn around the critical value for $\Delta_s$. The paths $\gamma_t$ and $\tilde\gamma_t$ define the same element $\tau(\Delta_t')$ of the group $H_{N-1}(A\cap X\setminus S),$ but different classes $\overline{\Delta}_t'$ and $\overline{\Delta}_t'\pm l(\Delta_s,\Delta_t')\Delta_s$ in the group $H_{N-1}(A\cap X,S)$. Let two systems of paths be obtained from each other by replacing $\gamma_t$ with $\tilde\gamma_t$. Then for any vanishing cycle $\Delta_i$ such that \(\langle\Delta_i,\Delta_s\rangle\neq 0\), the values of the coefficients $l(\Delta_i,\Delta_t')$ defined by these two systems of paths turn out to be different.
	\end{remark}
	
	In Sec.~\ref{43}, we will show that the introduced matrix plays a role similar to that of the ordinary intersection matrix in classical theory.
	
	\subsection{Intersection matrices}\label{MP}
	
	\begin{PREDLOZHENIE}\label{pepe}
		Let a single distinguished system of paths determine the bases $\{\Delta_s, \tau(\Delta_t')\}$ and $\{\beta_s, \beta_t'\}$ (with orientations chosen arbitrarily). Then:
		
		1) $\langle\beta_i, \beta_j\rangle/2 \equiv \langle\Delta_i, \Delta_j\rangle \pmod 2$,
		
		2) $\langle\beta_s, \beta_t'\rangle/2 \equiv l(\Delta_s, \Delta_t') \pmod 2$,
		
		3) $\langle\beta_i', \beta_j'\rangle \equiv \langle \Delta_i', \Delta_j'\rangle \pmod 2$.
	\end{PREDLOZHENIE}
	
	\textbf{Proof.} We decompose the antisymmetric cycles into their components $\beta_s=\Delta_s^{\operatorname{I}}-\Delta_s^{\operatorname{II}}$ and $\beta_t'=\overline{\Delta_t'}^{\operatorname{I}}-\overline{\Delta_t'}^{\operatorname{II}}$ such that $\sigma_*\Delta_s^{\operatorname{I}}=\Delta_s$ and $\sigma_*\overline{\Delta_t'}^{\operatorname{I}}=\overline{\Delta_t'}$. Then we have $\langle\beta_i, \beta_j\rangle=\langle\Delta_i^{\operatorname{I}}-\Delta_i^{\operatorname{II}}, \Delta_j^{\operatorname{I}}-\Delta_j^{\operatorname{II}}\rangle=\langle\Delta_i^{\operatorname{I}}, \Delta_j^{\operatorname{I}}\rangle+\langle\Delta_i^{\operatorname{II}}, \Delta_j^{\operatorname{II}}\rangle-\langle\Delta_i^{\operatorname{I}}, \Delta_j^{\operatorname{II}}\rangle-\langle\Delta_i^{\operatorname{II}}, \Delta_j^{\operatorname{I}}\rangle.$ Each intersection point of the cycles $\Delta_i$ and $\Delta_j$ is counted twice in this sum with the same sign, since the involution $\rho:x_1\mapsto-x_1$ preserves the complex orientation of the fiber. The second statement is proved in exactly the same way. Let us clarify the third statement by using the definition of the half-cycle $\overline{\Delta_t'}$ via the path $I_t'\subset\mathbb C_{x_1}$ (see Sec.~\ref{Rassl}). These paths, generally speaking, can intersect away from zero for different $t$ (see Example~\ref{Xx9}). The short cycle $\beta_t'$ is defined by the full preimage of the path $I_t'$ under the map $x_1\mapsto x_1^2$, and the intersection points of the cycles $\beta'_i$ and $\beta'_j$ outside $x_1=0$ are also counted in pairs with identical signs. $\Box$
	
	Next, we show that there exist special bases for which the congruences modulo two can be replaced by exact equalities.
	
	By a small perturbation of the germ $\breve{f}_{\tilde\lambda}\equiv f_{\tilde\lambda}(x_1^2, x_2, \dots, x_N)$ and of the system of paths (slightly bifurcating the paths originating from the critical values of the function $f_{\tilde\lambda}$), we define a distinguished basis of the group $H_{N-1}(\breve A\cap X).$ The involution defined by the transformation $\rho:x_1\mapsto-x_1$ in composition with $-\operatorname{id}$ acts on the Dynkin diagram of this basis (we orient the permuted cycles such that under the action of $\rho_*$ one transforms into the negative of the other). It permutes the vertices corresponding to long cycles, while leaving the vertices corresponding to short ones fixed. The Dynkin diagram of the antisymmetric basis is obtained from this diagram by factorization under the action of this involution and doubling the edges connecting the images of short and long vertices. From this, its connectedness immediately follows (\cite{Lyash}, Lemma 2.1). 
	
	\begin{theorem}\label{ga}(A.~M.~Gabrielov, \cite{Ga}, Theorem 3)
		
		1) Let $\{h_j\}$ be a distinguished basis determined by a special system of paths for a special perturbation of the singularity $f|_{x_1=0}+x_1^2$. Then there exists a distinguished basis $\{\hat e_j^m\}$ for $\breve f$ with the intersection matrix:
		$$\begin{aligned}
			& \langle \hat e_j^m, \hat e_{j'}^m\rangle=\langle h_j, h_{j'}\rangle,\\ 
			& \langle \hat e_j^m, \hat e_{j}^{m'}\rangle=(|m'|-|m|)^{N-1} \quad \text{for} \quad |m'-m|=1,\\ 
			& \langle \hat e_j^m, \hat e_{j'}^{m'}\rangle=(-1)^N\langle h_j, h_{j'}\rangle \quad \text{for} \quad |m'-m|=1 \quad \text{and} \quad (|m'|-|m|)(j'-j)<0, \\ 
			& \langle \hat e_j^m, \hat e_{j'}^{m'}\rangle=0 \quad \text{for} \quad |m'-m|>1 \quad \text{or} \quad (|m'|-|m|)(j'-j)>0. 
		\end{aligned}$$
		
		2) The involution $\rho$ acts on the basis $\{\hat e_j^m\}$ according to the formula $\hat e_j^m\mapsto-\hat e_j^{-m}$, meaning that the basis $\{\hat e_j^m\}$ defines the antisymmetric basis of the boundary singularity $f$: $\hat e_j^0$ are short cycles, and $\hat e_j^m+\hat e_j^{-m}$ for $m>0$ are long cycles. 
	\end{theorem}
	
	From the last condition of item 1 of this theorem, it follows that the subdiagram in the Dynkin diagram of the basis $H^-$ from item 2, consisting of long vertices and the edges connecting them, is the Dynkin diagram of the basis $\{\sigma_*\hat e_j^m, m>0\}$ (or $m<0$) of the group $H_{N-1}(A\cap X)$ (\cite{Il},~Lemma~4). In other words, the diagram for $f$ lifts to the diagram for $\breve f$ as two disjoint copies.
	
	\begin{lemm}\label{hard}
		There exists a distinguished basis $\{\Delta_s^{\operatorname{I}}, \Delta_s^{\operatorname{II}}, \beta_t'\}$ of the group $H_{N-1}(\breve A\cap X)$ such that $\rho_*(\Delta_s^{\operatorname{I}})=\Delta_s^{\operatorname{II}}$ (and hence, $\{\beta_s=\Delta_s^{\operatorname{I}}-\Delta_s^{\operatorname{II}}, \beta_t'\}$ is a basis of the group $H^-$), and for the basis $\{\Delta_s, \tau(\Delta_t')\}$ of the group $H_{N-1}(A\cap X\setminus S)$ determined by the same system of paths, under a natural agreement of orientations of cycles vanishing over the same path, the following holds:
		$$\begin{aligned}
			& 1) \langle\beta_{i}, \beta_{j}\rangle=2\langle\Delta_{i}, \Delta_{j}\rangle,\\ 
			& 2) \langle\beta_{s}, \beta_{t}'\rangle=2\;l(\Delta_{s}, \Delta_{t}'),\\ 
			& 3) \langle \beta_{i}', \beta_{j}'\rangle=\operatorname{sgn}(j-i)(-1)^{N-1}\langle\Delta_{i}', \Delta_{j}'\rangle.  
		\end{aligned}$$
	\end{lemm}
	
	\textbf{Proof.} 1) Let us take a basis $\{\Delta_s^{\operatorname{I}}, \Delta_s^{\operatorname{II}}, \beta_t'\}$ satisfying Theorem~\ref{ga} after a renumbering and a reversal of the orientations of the cycles $\Delta_s^{\operatorname{II}}$. We orient the cycles $\{\Delta_s\}$ so that $\sigma_*(\Delta_s^{\operatorname{I}})=\Delta_s$ holds. Then $\sigma_*(\Delta_s^{\operatorname{II}})=\Delta_s$ as well, and since $\langle\Delta_{i}^{\operatorname{I}}, \Delta_{j}^{\operatorname{II}}\rangle=0$ for all $i,j$, the first condition of the lemma is satisfied.
	
	2) From the second condition of item 1 of Theorem~\ref{ga}, it follows that each short cycle $\beta_j'$ uniquely decomposes into a sum of two half-cycles $\overline{\Delta_j'}^{\operatorname{I}}-\overline{\Delta_j'}^{\operatorname{II}}$ such that the first half-cycle intersects $\Delta_{j}^{\operatorname{I}},$ and the second intersects $\Delta_{j}^{\operatorname{II}},$ where $\beta_j'\equiv\hat e_j^0, \Delta_{j}^{\operatorname{I}}\equiv \hat e_j^1, \Delta_{j}^{\operatorname{II}}\equiv -\hat e_j^{-1}.$ We orient the half-cycles $\{\overline{\Delta_j'}\}$ so that $\sigma_*\overline{\Delta_j'}^{\operatorname{I}}=\sigma_*\overline{\Delta_j'}^{\operatorname{II}}=\overline{\Delta_j'}.$ It is sufficient for us to check that $\langle\overline{\Delta_t'}^{\operatorname{I}}, \Delta_s^{\operatorname{II}}\rangle=0$ for all $t,s.$
	
	\begin{figure}[H]
		\center
		\makebox{\includegraphics[width=1\linewidth]{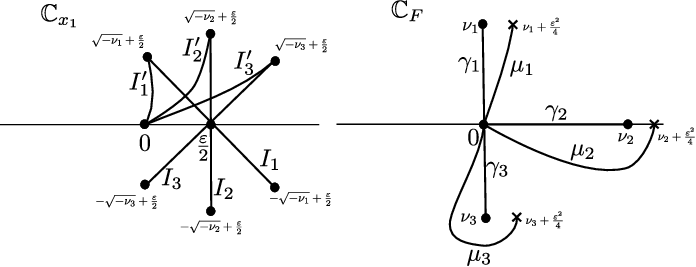}}
		\caption{The paths $I_j$ and $I_j'$}
		\label{lemma}	
	\end{figure}
	
	Let us turn to Gabrielov's construction from~\cite{Ga}. Consider the special perturbation $$f(x_1+\frac{\varepsilon}{2},x_2,\dots,x_N)-2\varepsilon(x_1+\frac{\varepsilon}{2})+\varepsilon^2,\quad \varepsilon>0$$ of the germ $f$ (if necessary, we replace it by a close morsification) and deform it into the perturbation $F=f(\frac{\varepsilon}{2},x_2,\dots,x_N)+(x_1-\frac{\varepsilon}{2})^2$ of the stabilization of the restriction $f(0,x_2,\dots,x_N)+x_1^2$. The deformation is carried out in two stages: first, the term $\delta (x_1+\frac{\varepsilon}{2})^2$, $\delta\in[0,1]$, is added, then the first (shifted) argument of the function $f$ is replaced by $(1-\delta) x_1+\frac{\varepsilon}{2}$, $\delta\in[0,1]$. According to~\cite{Ga} (see Lemma 2 and Theorem 2), along this deformation, the special (distinguished) paths for the cycles $\{\sigma_*\hat e_j^1\}\equiv\{\Delta_j\}\subset\{\Delta_s\}$ transform into the special paths $\{\gamma_j\}$ determining the basis $\{h_j\}$ for $F$. At the final moment of the deformation $f|_{x_1=\frac{\varepsilon}{2}}+(x_1-(1-\delta)\frac{\varepsilon}{2})^2$, $\delta\in[0,1]$, the latter paths $\{\gamma_j\}$ are identified with the special paths $\{\mu_j\}$ connecting the critical values of the function $F|_{x_1=0}$ to zero (defining the half-cycles $\{\overline{\Delta_j'}\}$).
	
	One can assume (\cite{Ga}, Propositions 4, 5) that all critical values $\nu_j$ of the function $F$ have distinct arguments, none of which is equal to $\pi$; and the special paths $\gamma_j$ are segments from $\nu_j$ to zero. Then the cycle $h_j$ in the fiber $F=0$ is determined by the segment $I_j=[\sqrt{-\nu_j}+\frac{\varepsilon}{2}, -\sqrt{-\nu_j}+\frac{\varepsilon}{2}]$ on the complex line $\mathbb{C}_{x_1}$. This segment is obtained by shifting by $\frac{\varepsilon}{2}$ the full preimage under the map $x_1\mapsto x_1^2$ of the trajectory, originating at zero, of the critical value of the function $x_1$ on the fibers $f|_{x_1=\frac{\varepsilon}{2}}+x_1=\nu$, $\nu\in\gamma_j$. We define the half-cycle $\overline{\Delta_j'}$ by the path $I_j'$ from zero to $\pm\sqrt{-\nu_j}+\frac{\varepsilon}{2}$. This path is the trajectory of the critical value of the function $x_1$ on the fibers $F=\nu$, $\nu\in\mu_j$. Moreover, for each path $\gamma_j$, the path $\mu_j$ is chosen such that in each such pair, the paths are identically ordered by the argument of entry into the origin (see Fig.~\ref{lemma}): this guarantees a consistent choice of the roots of $-\nu_j$ (for the endpoints of the paths $I_j'$) and the satisfaction of condition 3 of item 1 of Theorem~\ref{ga} for $m=0, m'=1$ (the path $I_t'$ intersects the path $I_s$ if and only if $s\leqslant t$). 
	
	The ray $\operatorname{Re}_{\leqslant0}\subset\mathbb C_{x_1}$ does not intersect the set $\bigcup_jI_j\cup I'_j$ anywhere except at zero. The required statement follows from this, since the short and long cycles $\hat e_j^0$ and $\hat e_j^1+\hat e_j^{-1}$ (identified with their images in the manifold $F(x_1^2,x_2,\dots,x_N)=0$ under the deformation specified above) are determined by the full preimages of these paths under the map $x_1\mapsto x_1^2$ (which are separated into sheets by the $\operatorname{Im}$ axis), while $\langle\hat e_j^0,\hat e_{j'}^m\rangle=0$ for $|m|>1$ by condition 4 of item 1 of Theorem~\ref{ga}.
	
	3) We orient the cycles $\{\Delta_t'\}$ from $\partial\overline{\Delta_t'}=\Delta_t'.$ From the first condition of item 1 of Theorem~\ref{ga} (or from the reasoning above), it follows that $\langle \beta_{i}',\beta_{j}'\rangle=\langle h_i,h_j\rangle,$ and the third condition of the lemma follows from Gabrielov's formula (\cite{GABRIEL}, Corollary 1 of Theorem 1) for stabilization by the square of one new variable. $\Box$
	
	\begin{Relation}
		The matrices of generalized intersections for the singularities from Table 1 of~\cite{Ga} are reconstructed from the data presented in that table by means of Theorem~\ref{ga} and Lemma~\ref{hard}.
	\end{Relation}
	
	\begin{Relation}\label{Pure}
		In the case of a purely boundary singularity, formula 3 of Lemma~\ref{hard} is valid for an arbitrary distinguished basis.
	\end{Relation}
	\textbf{Proof.} This follows from the same Gabrielov formula in~\cite{GABRIEL}. $\Box$
	
	The examples presented below show that Lemma~\ref{hard} does not hold for an arbitrary distinguished basis.
	
	The presence of a loop in the set $\cup_{i=1}^{\mu(f)}\Delta_i\subset A\cap X\setminus S$ with an odd linking number with the boundary $\hat S$ is an obstruction to the trivialization of the covering over this set. In particular, if the Dynkin diagram contains a cycle $\Delta_1\Delta_2\dots\Delta_k\Delta_1$ and such a loop exists in its support set, then it is impossible to adjust the orientations of the long cycles $\beta_i=\pm(\Delta_i^{\operatorname{I}}-\Delta_i^{\operatorname{II}})$ so that $\langle\beta_i,\beta_j\rangle=+2\langle\Delta_i,\Delta_j\rangle$ holds simultaneously for all $i,j$.
	
	\begin{figure}[H]
		\center{\includegraphics[width=0.33\linewidth]{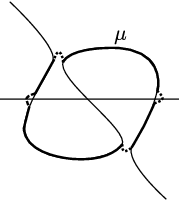}}
		\caption{Scheme of the path $\mu$}
		\label{mumu}	
	\end{figure} 
	
	\begin{example}\label{dd4} 
		Consider the {\em sabirization} (see~\cite{GZ}, \cite{VVA},~\S4.1) $x^3+y^3-\varepsilon(x+y)$ of a germ of class $D_4$. For its perturbation $F=x^3+y^3-\varepsilon((1-\delta)x+y)$, $\delta\ll\varepsilon$, we define a basis $\{\Delta_i\}_{i=1}^4$ of the group $H_1(A\cap X)$ using the canonical system of paths $\{\gamma_i\}_{i=1}^4$ (in $\operatorname{Im}\geqslant0$ with the distinguished point $0$). The cycles $\Delta_1$ and $\Delta_4$ vanish at the minimum and maximum points, while $\Delta_2$ and $\Delta_3$ vanish at the saddles. Consider the cycle $\Delta_1\Delta_2\Delta_4\Delta_3\Delta_1$ in the support of the Dynkin diagram. Introduce the boundary $\hat S:x=0$ and mark the critical values of the function $y^3-\varepsilon y$. The cycles $\Delta_1$ and $\Delta_4$, now considered as elements of the group $H_1(A\cap X\setminus S)$, can be realized in $\mathbb R^2$ except for sufficiently small neighborhoods of the saddle critical points of the function $F$ (where $\Delta_2$ and $\Delta_3$ are concentrated) and the set $A\cap X\cap S$. Consider a closed path $\mu$ in the set $\Delta_1\cup\Delta_2\cup\Delta_3\cup\Delta_4\subset A\cap X\setminus S$, starting at an arbitrary point of the cycle $\Delta_1$ and running along $\Delta_1$ to the point $\Delta_1\cap\Delta_2$, then along $\Delta_2$ to the point $\Delta_2\cap\Delta_4$, and so on (the cycles $\Delta_2$ and $\Delta_3$ can be assumed to be transversal to $\operatorname{Re}A\cap X$), see Fig.~\ref{mumu}. It is elementarily verified that $\operatorname{lk}(\mu,\hat S)=0$ (it suffices to take the projection of the loop $\mu$ onto $\mathbb C_{x}$). This value depends on how the paths $\gamma_1$ and $\gamma_4$ bypass the critical values of the function $F|_{x=0}$: if one of these paths is replaced by a path running in $\operatorname{Im}<0$, then the linking number changes its value by $\pm1$, and then the cycle in the Dynkin diagram for $F$ lifts to a cycle in the diagram for $\breve F$.
	\end{example}
	
	The next example relates to Remark~\ref{ed} and to item 3 of Lemma~\ref{hard}.
	
	\begin{example}\label{ff4}
		Consider a morsification of a boundary singularity of class $F_4$ of two variables in the form $F=x^2+y^3+\varepsilon(x-y)$ with the boundary $x=0$. In~\cite{NIKIT}, a system of paths $\{\gamma_i\}_{i=1}^4$ was chosen, a basis $\{\Delta_1, \Delta_2, \tau(\Delta_3'), \tau(\Delta_4')\}$ of the group $H_1(A\cap X\setminus S)$ was constructed, and its monodromy was computed (see~\eqref{C}). Let us replace the path $\gamma_3$ by a path $\tilde{\gamma_3},$ see Fig.~\ref{paths}. The new path determines the same elements $\Delta_3'\in \tilde H_0(A\cap X\setminus S)$ and $\tau(\Delta_3')\in H_1(A\cap X\setminus S)$, but a different element of the group $H_1(A\cap X,S)$. We have $l(\Delta_2, \Delta_3')_{\gamma_2, \gamma_3}=0$, but $l(\Delta_2, \Delta_3')_{\gamma_2, \tilde\gamma_3}=\langle\Delta_2, \overline{\Delta_3'}-\Delta_1\rangle=1$.  
		
		Suppose that item 3 of Lemma~\ref{hard} holds. Replacing $\gamma_3$ by $\tilde\gamma_3$ does not affect the numbering of paths for $F|_{\hat S}$ by the argument of entry into the origin. Then the relation $\langle\beta_3', \beta_4'\rangle=\langle\tilde\beta_3', \beta_4'\rangle$ must hold, where the cycle $\tilde\beta_3'$ is determined by the path $\tilde{\gamma_3}.$ By the Picard--Lefschetz formula, we have $\tilde\beta_3'=\gamma_1(\beta_3')=\beta_3'-\frac{\langle\beta_3', \beta_1\rangle}{2}\beta_1$. We have $\langle\tilde\beta_3', \beta_4'\rangle=\langle\beta_3', \beta_4'\rangle-\dfrac{\langle\beta_3', \beta_1\rangle}{2}\langle\beta_1, \beta_4'\rangle$. Since both values $\langle\beta_3', \beta_1\rangle$ and $\langle\beta_1, \beta_4'\rangle$ are equal to $\pm2$, we arrive at a contradiction. Here, changing the path alters the value of the intersection number of short cycles from $1$ to $-1$ or vice versa. 
	\end{example}
	
	\begin{figure}[H]
		\center{\includegraphics[width=0.6\linewidth]{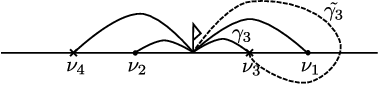}}
		\caption{Distinguished system of paths for $F_4$}
		\label{paths}	
	\end{figure}
	
	The next example justifies item 3 of Proposition~\ref{pepe} and strengthens the previous one.
	
	\begin{example}\label{Xx9}
		Consider the perturbation $F=(x+\varepsilon^2)^2+(y^2-\varepsilon^2)(y^2-4\varepsilon^2)+\varepsilon^5y-\dfrac{9}{4}\varepsilon^4$ of the germ $x^2+y^4$ at the origin with the boundary $x=0$. All critical values of the functions $F$ and $F|_{x=0}$ are real and are arranged as follows: $\nu_1<\nu_2<\nu_1'<\nu_2'<0<\nu_3<\nu_3'$. We define a basis of the group $H_1(A\cap X,S)$ by a canonical system of paths in $\operatorname{Im}\geqslant0$. The elements $\Delta_1$ and $\Delta_2$ can be realized by closed curves in $\mathbb R^2$ (bypassing the boundary in the imaginary domain), and the elements $\overline{\Delta_1'}$ and $\overline{\Delta_2'}$ by the arcs of these curves lying in $x>0$, see Fig.~\ref{X99}. We introduce orientations such that $l(\Delta_i,\Delta_j')\equiv\langle\Delta_i,\overline{\Delta_j'}\rangle=\delta_{ij}$ holds for $i\in\{1,2,3\}, j\in\{1,2\}$; $\langle\Delta_1,\Delta_3\rangle=\langle\Delta_2,\Delta_3\rangle=1$; obviously, $\langle\Delta_1',\Delta_2'\rangle=0$ and $\langle\Delta_1,\Delta_2\rangle=0$ hold. Consider the paths $\mathring{\gamma}_1'$ and $\mathring{\gamma}_2'$ shown in Fig.~\ref{x9}. These paths determine the same elements $\Delta_1'$ and $\Delta_2'$ of the group $\tilde H_0(A\cap X\cap S)$. By the Picard--Lefschetz formula, the path $\mathring{\gamma}_1'$ determines the element $\mathring{\overline{\Delta_1'}}=\gamma_1^{-1}\gamma_2^{-1}(\overline{\Delta_1'})=\overline{\Delta_1'}+\Delta_1$ of the group $H_1(A\cap X,S)$, and $\mathring{\gamma}_2'$ determines the element $\mathring{\overline{\Delta_2'}}=\gamma_3^{-1}\gamma_1^{-1}\gamma_2^{-1}\gamma_1'^{-1}(\overline{\Delta_2'})=\overline{\Delta_2'}+\Delta_2+\Delta_3$. Then we have $\langle\mathring{\overline{\Delta_1'}},\mathring{\overline{\Delta_2'}}\rangle_{A\cap X\setminus S}=\langle\overline{\Delta_1'}+\Delta_1, \overline{\Delta_2'}+\Delta_2+\Delta_3\rangle_{A\cap X\setminus S}=1$. It follows from this that the paths $\mathring{I}_1'$ and $\mathring{I}_2'$ must have a non-removable intersection point away from zero. For the covering short cycles, $\langle\mathring{\beta}_1',\mathring{\beta}_2'\rangle=\pm2\neq\langle\Delta_1',\Delta_2'\rangle$ holds. These cycles realize the basis elements for the unimodal singularity $\breve f=x^4+y^4$ of type $X_9$. Since the paths $\mathring{\gamma}_1'$ and $\mathring{\gamma}_2'$ do not intersect anywhere except at zero, they can be completed to a distinguished system, and one can construct bases for which item 3 of Lemma~\ref{hard} does not hold (even up to sign). 
	\end{example}
	
	\begin{figure}[H]
		\center{\includegraphics[width=1\linewidth]{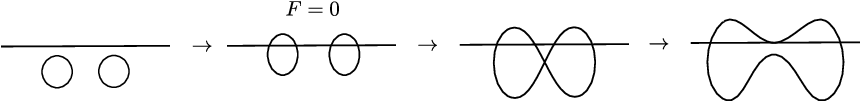}}
		\caption{Topological scheme of bifurcations of the level sets}
		\label{X99}	
	\end{figure} 
	
	\begin{figure}[H]
		\center{\includegraphics[width=0.6\linewidth]{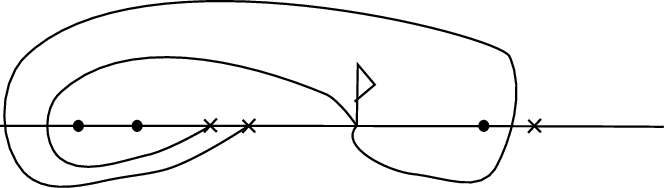}}
		\caption{The paths $\mathring{\gamma}_1'$ and $\mathring{\gamma}_2'$}
		\label{x9}	
	\end{figure} 
	
	\begin{remark}\label{auth}
		It is unknown to the author whether the conditions of Lemma~\ref{hard} are satisfied in the case of the canonical distinguished basis (see Fig.~\ref{systpath}) for an arbitrary real perturbation of a real singularity.
	\end{remark}
	
	For $N=2$, we consider (see~\cite{NIKIT}, pp.~9--14) certain distinguished bases for simple singularities $B_{k}, C_{k}, F_{4}$ and directly compute the matrices of generalized intersections. The results for $N>2$ can be obtained from the results for $N=2$ by applying Gabrielov's formulas for stabilizations (see~\cite{GABRIEL}).
	
	\begin{lemm}\label{BCF} 
		The matrices of generalized intersections of certain distinguished bases of the group $H_{1}(A\cap X\setminus S)$ for the singularities $B_{k}, C_{k}, F_{4}$ of two variables are of the form, respectively:
		\begin{equation}\label{A}	
			\begin{pmatrix}
				2 &  0 &          0 & \dots & \dots & \dots &  0 \\
				-1 &  0 &  (-1)^{k} &     0 & \dots & \dots &  0 \\
				\dots & (-1)^{k-1} &  0 & (-1)^{k-1} & 0 & \dots &  0 \\
				\dots &  0 & (-1)^{k-2} & \dots & \dots & \dots &  0 \\
				\dots & \dots &          0 & \dots &     0 &     1 &  0 \\
				\dots & \dots &      \dots & \dots &    -1 &     0 & -1 \\
				0 &  0 &          0 &     0 &     0 &     1 &  0
			\end{pmatrix},
		\end{equation}
		
		\begin{equation}\label{B}
			\begin{pmatrix}
				0 & -1 &  0 & \dots & \dots & \dots &  0 \\
				0 &  2 & -1 &     0 & \dots & \dots &  0 \\
				\dots & -1 &  2 &    -1 &     0 & \dots &  0 \\
				\dots &  0 & -1 & \dots & \dots & \dots &  0 \\
				\dots & \dots &  0 & \dots &     2 &    -1 &  0 \\
				\dots & \dots & \dots & \dots &    -1 &     2 & -1 \\
				0 &  0 &  0 &     0 &     0 &    -1 &  2
			\end{pmatrix},
		\end{equation}
		
		\begin{equation}\label{C}	
			\begin{pmatrix}
				0 &  1 & -1 &  1 \\
				-1 &  0 &  0 & -1 \\
				0 &  0 &  2 & -1 \\
				0 &  0 & -1 &  2
			\end{pmatrix}.
		\end{equation}
	\end{lemm}
	\textbf{Proof} see in~\cite{NIKIT} (Lemmas 2--4). $\Box$
	
	\subsection{The bilinear form associated with the monodromy group}\label{43}
	
	The matrix of generalized intersections defines a bilinear form $[\cdot,\cdot]$ on the homology group $H_{N-1}(A\cap X\setminus S).$ This form is invariant under the action of the monodromy, since its values are the intersection numbers of cycles in the manifolds $A\cap X\setminus S$ and $A\cap X\cap S$, and the monodromy itself can be realized by a composition of diffeomorphisms of the fiber $A\cap X$ that preserve the boundary.
	
	Let us denote by $\Theta$ a ball of radius $\delta$ (see Sec.~\ref{Rassl}) around zero in the parameter space $\mathbb C^l$ of an arbitrary deformation $F$ of the form \eqref{obsch}. By the generalized Zariski theorem (see~\cite{Zarisk}, \cite{GM}), the homomorphism $\pi_1(B_\delta\setminus\Sigma)\to\pi_1(\Theta\setminus\Sigma)$ is an epimorphism: the fundamental group of the complement to the discriminant is generated by simple loops around the critical values.
	
	\begin{definition}
		The image of the representation of $\pi_{1}(\Theta\setminus\Sigma(F))$ in the group $O(H_{N-1}(A\cap X\setminus S), [\cdot,\cdot])$ is called the \textbf{monodromy group} of the boundary singularity $f$ with a non-proper boundary.
	\end{definition}
	
	Let us consider the diagram:
	\begin{equation}\label{cd}
		\begin{CD}
			H_{N-1}(A\cap X,\partial D) @<\text{epi}<< H_{N-1}(A\cap X\setminus S,\partial D) @<\text{mono}<< H_{N-2}(A\cap X\cap S, \partial D)\\
			\otimes @. \otimes @. \otimes \\
			H_{N-1}(A\cap X) @>\text{mono}>>  H_{N-1}(A\cap X,A\cap X\cap S) @>\text{epi}>> H_{N-2}(A\cap X\cap S)\\
			@VVV @VVV @VVV \\
			\mathbb{Z} @>\operatorname{id}>>\mathbb{Z} @>\cdot(-1)^N>>\mathbb{Z}
		\end{CD}
	\end{equation}
	The upper row is taken from the exact sequence of the triple $((A\cap X)\cup\partial D, (A\cap X \setminus S)\cup\partial D, \partial D)$, and the middle row is from the exact sequence of the pair $(A\cap X, A\cap X\cap S)$. In the upper right corner, the group $H_{N}((A\cap X)\cup\partial D, (A\cap X \setminus S)\cup\partial D)$ is replaced by an isomorphic one using excision and the Thom isomorphism (see the proof of Lemma~\ref{baza}). Thus, the upper right arrow is given by the Leray tube operator. The columns of the diagram are non-degenerate pairings given by intersection numbers. The commutativity of the diagram follows from the proof of Proposition~\ref{bilin}.
	
	Let us denote by $i_*$ the homomorphism $H_{N-1}(A\cap X\setminus S)\rightarrow H_{N-1}(A\cap X\setminus S,\partial D),$ and by $\{\alpha_i\}_{i=1}^{\mu}$ the distinguished basis of the group $H_{N-1}(A\cap X\setminus S)$, i.e., an arbitrarily numbered set $\{\Delta_{s},\tau(\Delta_{t}')\}$. The system of paths determines an isomorphism between the groups $H_{N-1}(A\cap X\setminus S)$ and $H_{N-1}(A\cap X,S)$: the vanishing cycles $\Delta_s$ realize the classes of both groups, and the tubes $\tau(\Delta_{t}')$ correspond to the half-cycles $\overline{\Delta_{t}'}$ taken with their natural orientation. Of course, this isomorphism depends on the system of paths (see Remark~\ref{ed}).
	
	Let us denote by $\{\overline{\alpha_i}\}_{i=1}^{\mu}$ the basis of the group $H_{N-1}(A\cap X,S)$ corresponding to the basis $\{\alpha_i\}_{i=1}^{\mu}$ under the specified isomorphism. The following proposition can be taken as another definition of the form of generalized intersections $[\cdot,\cdot]$ on the group $H_{N-1}(A\cap X\setminus S).$
	
	\begin{PREDLOZHENIE}\label{bilin}
		$[\alpha_i,\alpha_j]=\langle i_*\alpha_i,\overline{\alpha_j}\rangle_{A\cap X\setminus S}$.
	\end{PREDLOZHENIE}
	
	\textbf{Proof.} If we pair two vanishing cycles or a tube with a vanishing cycle, the statement is obvious; if we pair a vanishing cycle with a tube, see Definition~\ref{chislol}; and if we pair a tube with a tube, it is easy to verify that $(-1)^N\langle\Delta_{i}',\Delta_{j}'\rangle_{A\cap X\cap S}=\langle\tau(\Delta_{i}'),\overline{\Delta_{j}'}\rangle_{A\cap X\setminus S}$. Indeed, let the cycles $\Delta_{i}'$ and $\Delta_{j}'$ intersect transversally in $A\cap X\cap S.$ The transversality of $A\cap X$ to the boundary allows us to identify a small tubular neighborhood of $A\cap X\cap S$ in $A\cap X$ with the direct product of $A\cap X\cap S$ and a small disk in $\mathbb{C}_{x_1}$ centered at the origin. Then the cycle $\tau(\Delta_{i}')$ is identified with $\{\varepsilon e^{\sqrt{-1}\varphi}\}\times\Delta_{i}',$ and the cycle $\overline{\Delta_{j}'}$ with $I_{j}'^0\times\Delta_{j}',$ where $I_{j}'^0$ is a finite segment (terminating at zero) of the path $I_{j}'$ in $\mathbb{C}_{x_1}$ that defines this cycle on the manifold $A\cap X$ (this path is the result of deforming the trajectory, originating at zero, of the critical value of the function $x_1$ on the fibers $f_{\tilde\lambda}=\nu$, where $\nu$ traces the distinguished path $\gamma_j$). Each intersection point of the cycles $\Delta_{i}'$ and $\Delta_{j}'$ in the manifold $A\cap X\cap S$ corresponds to a unique intersection point of the cycles $\tau(\Delta_{i}')$ and $\overline{\Delta_{j}'}$ lying in $A\cap X\cap S\times \varepsilon\cdot e^{\sqrt{-1}\operatorname{arg}I_{j}'}.$ Now, a straightforward accounting of orientations yields $\langle\Delta_{i}',\Delta_{j}'\rangle=(-1)^{N-2}\langle\tau(\Delta_{i}'),\overline{\Delta_{j}'}\rangle. \Box$
	
	Let $\gamma_i$ be a simple loop corresponding to the basis element $\alpha_i\in H_{N-1}(A\cap X\setminus S).$ Then the monodromy operator $\gamma_i$ acts on the element $\alpha_j$ according to the following formula:
	
	\begin{lemm}\label{Picard}(Generalized Picard--Lefschetz formula)
		\begin{equation}\label{PL}
			\gamma_i(\alpha_j)=\alpha_j+(-1)^{\frac{N(N+1)}{2}}[\alpha_j,\alpha_i]\alpha_i. \;\Box
		\end{equation}	
	\end{lemm}
	
	Let us fix a basis $\{\nabla_i\}_{i=1}^{\mu}$ in the group $H_{N-1}(A\cap X\setminus S, \partial D)$ that is Poincaré dual to the basis $\{\overline{\alpha_i}\}_{i=1}^{\mu}$ of the group $H_{N-1}(A\cap X,S)$.
	
	\begin{PREDLOZHENIE}
		The matrix of generalized intersections is the matrix of the homomorphism $i_*:H_{N-1}(A\cap X\setminus S)\rightarrow H_{N-1}(A\cap X\setminus S,\partial D)$ computed in the bases $\{\alpha_i\}_{i=1}^{\mu}$ and $\{\nabla_i\}_{i=1}^{\mu}$. $\Box$
	\end{PREDLOZHENIE}
	
	\begin{remark}
		In a similar way, one can introduce a form of generalized intersections on the group $H_{N-1}(A\cap X, A\cap X\cap S)$ using the pairing $H_{N-1}(A\cap X,S\cup\partial D)\otimes H_{N-1}(A\cap X\setminus S)\to\mathbb Z$ (see \eqref{c'd}), but this is not required for our purposes.
	\end{remark}
	
	\begin{figure}[H]
		\center{\includegraphics[width=1\linewidth]{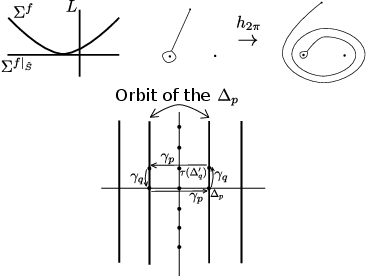}}
		\caption{To Example~\ref{po}.}
		\label{B2'}	
	\end{figure}  
	
	\begin{example}\label{po} 
		Let us compute $\pi_1(\Theta\setminus\Sigma)$ and the monodromy group for a singularity of class $B_2$. Consider the germ $f(x,y)=x^2+y^2+\sum_{i=1}^{N-2}z_i^2$ at $0\in\mathbb{C}^N$ with the boundary $x=0$.
		Consider a versal deformation of the form \eqref{miniversal} $F=f+\lambda_1 x+\lambda_0$. The sets $\Sigma^{f}$ and $\Sigma^{f|_{\hat{S}}}$ are given by the conditions $\lambda_0=\lambda_1^2/4$ and $\lambda_0=0$, respectively. Let us define a line $L$ by the equation $\lambda_1=\varepsilon>0$ and fix a distinguished point on $L$ (here, for convenience, we choose a point on $\partial B_\delta$). The group $\pi_1(\Theta\setminus\Sigma)$ is generated by two simple loops $\gamma_p$ and $\gamma_q$ on $L$ around the points $p=\Sigma^{f}\cap L$ and $q=\Sigma^{f|_{\hat{S}}}\cap L.$ By the Zariski theorem, all relations on these generators are computed by looping the line $L$ around the line $\lambda_1=0.$ Consider the loop $\{\lambda_1=\varepsilon\cdot e^{it},\;t\in[0,2\pi]\}$ and compute the action of the diffeomorphism $h_{2\pi}:(L, L\setminus\Sigma)\rightarrow(L, L\setminus\Sigma)$, determined by the covering homotopy, on $\gamma_p$ and $\gamma_q$. Along this loop, the point $p$ makes two counterclockwise turns around the point $q$ (see Fig.~\ref{B2'}), which yields the unique relation $(\gamma_p\gamma_q)^2=(\gamma_q\gamma_p)^2$. 
		
		Let $N$ be even. Let the paths $\gamma_p$ and $\gamma_q$ determine the basic cycles $\Delta_{p}$ and $\tau(\Delta'_q)$, taken with arbitrary orientations. The computation of $l(\Delta_{p},\Delta'_q)$ reduces to the situation shown in Fig.~\ref{u} (under stabilizations, the number $|l(\Delta_{p},\Delta'_q)|=|\langle\Delta_{p},\overline{\Delta'_q}\rangle|$ remains constant, see~\cite{GABRIEL}). The matrix of generalized intersections is of the form:
		\begin{equation*}	
			\begin{pmatrix}
				0 & \pm1\\
				0 & 2(-1)^{N/2}
			\end{pmatrix},
		\end{equation*}
		i.e., $[\Delta_p,\Delta_p]=0,\;[\Delta_{p},\tau(\Delta'_q)]=\pm1,\;[\tau(\Delta'_q),\Delta_{p}]=0$, and $[\tau(\Delta_q'),\tau(\Delta_q')]=\pm2.$
		Thus, the operator $\gamma_p$ acts trivially on the group $H_{N-1}(A\cap X\setminus S),$ while the operator $\gamma_q$ is of order 2 (see~\eqref{PL}), and the monodromy group is isomorphic to $\mathbb{Z}_2.$
		
		In the case of odd $N$, we have the matrix:
		\begin{equation*}	
			\begin{pmatrix}
				2(-1)^{(N-1)/2} & \pm1\\
				0 & 0
			\end{pmatrix}.
		\end{equation*}
		The orbit of the cycle $\tau(\Delta_q')$ consists of a single point, but the orbit of the cycle $\Delta_p$ is infinite, which implies that the monodromy group is also infinite. The operator $\gamma_p$ is of order 2, and the subgroup $\langle\gamma_q\rangle$ generated by the operator $\gamma_q$ is infinite cyclic. The operator $(\gamma_p\gamma_q)^2$ acts trivially on $H_{N-1}(A\cap X\setminus S)$ (see Fig.~\ref{B2'}). The subgroup $\langle\gamma_p\rangle$ acts on the normal subgroup $\langle\gamma_q\rangle$ by conjugations according to the formula $\varphi_p(\gamma_q^s)=\gamma_p\gamma_q^s\gamma_p=\gamma_q^{-s}.$ From this, it follows that the monodromy group is isomorphic to the infinite dihedral group $\mathbb{Z}\rtimes_{\varphi}\mathbb{Z}_2\simeq D_{\infty}$.
	\end{example}
	
	\begin{lemm}\label{monodromia}
		Let $f$ be an isolated boundary singularity with a non-proper boundary. 
		1) If $df(0)=0$, then the monodromy group is finite if and only if $N$ is even and this singularity belongs to the series $C_k$. In this case, it is isomorphic to the Weyl group $A_{k-1}$. 
		2) If $df(0)\neq0$, then the monodromy group is finite if and only if the singularity is either simple for even $N$ or Morse for odd $N$.
	\end{lemm}
	
	\begin{remark}
		The complex singularities $B_2$ and $C_2$ are equivalent, see~\cite{ARN}.
	\end{remark}
	
	\textbf{Proof.} If $df(0)\neq0,$ the statement reduces to the well-known result for ordinary singularities (see~\cite{Norm}).
	Now let $df(0)=0$. If the intersection form on one of the groups $H_{N-1}(A\cap X)$ or $H_{N-2}(A\cap X\cap S)$, the rank of which is at least two, is skew-symmetric, then it follows from the connectedness of the Dynkin diagram for an ordinary singularity and the Picard--Lefschetz formula that the orbit of any vanishing cycle of this group under the action of the monodromy is infinite. Then the orbits of the same vanishing cycles (or tubes around them) in the group $H_{N-1}(A\cap X\setminus S)$ are also infinite, which implies that the monodromy group is infinite. It could turn out to be finite only if the intersection form on one of the groups $H_{N-1}(A\cap X)$ or $H_{N-2}(A\cap X\cap S)$ is symmetric while the other group is of rank 1, because then the intersection form in the adjacent dimension will be skew-symmetric. Therefore, it is sufficient to consider the singularities of the series $B_k$ for odd $N$ and $C_k$ for even $N$. In the case of $B_k$, the obstruction is the connectedness of the Dynkin diagram for a boundary singularity. In the case of $C_k$, the monodromy operator corresponding to a vanishing cycle of the group $H_{N-1}(A\cap X)$ acts trivially on the group $H_{N-1}(A\cap X\setminus S)$ (see~\eqref{C}), and the monodromy group is isomorphic to the Weyl group $A_{k-1}$. $\Box$
	
	\subsection{Points of type $B_2$ and the indecomposability of the monodromy representation}\label{Dnk}
	
	By virtue of Lemma~\ref{hard}, the connectedness of the Dynkin diagram for a boundary singularity with $df(0)=0$ is equivalent to at least one value of $l(\Delta_{s},\Delta_{t}')$ being non-zero (the case $df(0)\neq0$ reduces to the case of an ordinary singularity). 
	
	\begin{lemm}\label{Lyas}
		For any distinguished basis $\{\Delta_{s}\}_{s=1}^{\mu(f)}, \{\tau(\Delta_{t}')\}_{t=1}^{\mu(f|_{\hat{S}})}$ of the homology group $H_{N-1}(A\cap X\setminus S)$ of a local non-singular level set of an isolated boundary singularity $f$ with a non-proper boundary and with $df(0)=0$, there exists a pair $(s,t)=(\tilde{s},\tilde{t})$ such that $l(\Delta_{\tilde{s}},\Delta_{\tilde{t}}')\neq0.$
	\end{lemm}
	
	We present below a proof of this lemma based on different considerations; the connectedness of the Dynkin diagram was discussed in Sec.~\ref{MP}.
	
	Consider the square of the maximal ideal $\textbf{m}^2$ of the algebra $\mathbb{C}\{x\}_{0}$ of germs of holomorphic functions at the point $0\in\mathbb{C}^N$, consisting of all germs $f$ with $f(0)=0$ and $df(0)=0$. Let us fix a hyperplane (boundary) $\hat S\simeq\mathbb{C}^{N-1}:x_1=0$ in $\mathbb{C}^N$. Any germ $f\in\textbf{m}^2$ automatically has a critical point $0$ on the boundary: $df|_{x_1=0}(0)=0.$ We denote by $\nabla_{f|_{x_1}}$ the boundary gradient ideal $(x_{1}\frac{\partial f}{\partial x_1}, \dots, \frac{\partial f}{\partial x_N})$ in $\mathbb{C}\{x\}_{0}$ and by $T_f$ the minimal local transversal in the space $\textbf{m}^2$ at the point $f$ to the orbit of the isolated boundary singularity $f$ under the action of the group of holomorphic diffeomorphisms of the triple $(\mathbb{C}^N, \mathbb{C}^{N-1}, 0)$.
	
	\begin{PREDLOZHENIE}
		Let a boundary singularity $f$ be isolated and $df(0)=0$ hold. Then $\dim_{\mathbb{C}}T_f=\mu-2.$
	\end{PREDLOZHENIE}
	
	\textbf{Proof.} The tangent space to the orbit at the point $f$ consists of all functions of the form $\varphi_1(x)x_1\frac{\partial f}{\partial x_1}(x)+\sum_{i=2}^{N}\varphi_i(x)\frac{\partial f}{\partial x_i}(x),$ where $\varphi_{i}(0)=0$ for $i=2, \dots, N.$ This space is the ideal $\textbf{I}=(x_1\frac{\partial f}{\partial x_1})+\textbf{m}(\frac{\partial f}{\partial x_2}, \dots, \frac{\partial f}{\partial x_N})$ of the ring $\mathbb C\{x\}_0$.
	
	Let us compute the dimension of the quotient space $\nabla_{f|_{x_1}}/\textbf{I}.$ It follows from the isomorphism theorem that $\nabla_{f|_{x_1}}/\textbf{I}\simeq (\frac{\partial f}{\partial x_2}, \dots, \frac{\partial f}{\partial x_N})/\textbf{m} (\frac{\partial f}{\partial x_2}, \dots, \frac{\partial f}{\partial x_N})$ provided that $(x_1\frac{\partial f}{\partial x_1})\cap(\frac{\partial f}{\partial x_2}, \dots, \frac{\partial f}{\partial x_N})\subset\textbf{m}(\frac{\partial f}{\partial x_2}, \dots, \frac{\partial f}{\partial x_N}),$ which is ensured by the condition that the boundary singularity $f$ is isolated. By Nakayama's lemma, we have $$\nabla_{f|_{x_1}}/\textbf{I}\simeq \left(\frac{\partial f}{\partial x_2}, \dots, \frac{\partial f}{\partial x_N}\right)\otimes_{\mathbb{C}\{x\}_0}\mathbb{C}\{x\}_0/\textbf{m},$$
	which implies $\dim \nabla_{f|_{x_1}}/\textbf{I}=N-1.$
	
	On the other hand, $\dim \mathbb{C}\{x\}_0/\textbf{m}^2=N+1.$ Then, since $\textbf{I}\subset\nabla_{f|_{x_1}}\subset\mathbb{C}\{x\}_0$ and $\textbf{I}\subset\textbf{m}^2\subset\mathbb{C}\{x\}_0,$ we obtain: 
	\begin{eqnarray*}
		\dim T_f & = & \dim \textbf{m}^2/\textbf{I}=\dim \mathbb{C}\{x\}_0/\textbf{I}-\dim \mathbb{C}\{x\}_0/\textbf{m}^2=\\    
		& = & \dim \nabla_{f|_{x_1}}/\textbf{I}+\dim \mathbb{C}\{x\}_0/\nabla_{f|_{x_1}}-\dim \mathbb{C}\{x\}_0/\textbf{m}^2=\mu-2. \quad \Box 
	\end{eqnarray*}
	
	Consider $T_f$ as a deformation of the germ $f$ with the base $\mathbb{C}^{\mu-2}$. By the versality theorem (Sec.~8.3 of~\cite{AVG}), the set of germs with a $B_2$ singularity (see Example~\ref{po}) on the zero level set is open and dense in $T_f$. Indeed, this set is open and dense in $\textbf{m}^2$, and any deformation of the germ $f$ in the space $\textbf{m}^2$ is equivalent to a deformation induced from $T_f$. Let $\chi_{T}:\mathbb{C}^{\mu-2}\rightarrow\mathbb{C}^{\mu}$ be the map inducing the deformation $T_f$ from the miniversal deformation of the germ $f$ (for instance, of the form \eqref{miniversal}) in the space $\mathbb{C}\{x\}_{0}$. Obviously, $\chi_{T}(\mathbb{C}^{\mu-2})\subset\Sigma$. This map is an analog of the {\em Gabrielov map} from~\cite{Gabr}, Theorem 1. At a generic point $g\in T_f$, one has $\operatorname{ker}d\chi_{T}(g)=0$. Indeed, if $v\in\operatorname{Ker}d\chi_{T}(g),$ then the function $v$ must vanish on a set of $\mu-2$ distinct non-zero critical points of the function $g$ and its restriction to the boundary, but a family of such functions does not exist in $T_f$ by Thom's transversality theorem. We have:
	
	\begin{PREDLOZHENIE}\label{p8}
		Let $f$ be an isolated boundary singularity with $df(0)=0.$ Then the stratum of functions with a $B_2$ singularity on the zero level set has codimension $2$ in the base of the miniversal deformation $\mathbb C^{\mu}. \Box$
	\end{PREDLOZHENIE}
	
	Along the $B_2$ stratum, the components of $\Sigma$ are tangent with tangency order $1$ (see Example~\ref{po}).
	
	\begin{remark} 
		In a similar way, the stratum of functions with a $B_2$ singularity on the zero level set is non-empty in the base of a $\mathcal{K}_{\hat S}$-versal deformation. Moreover, the codimension of the $\mathcal{K}_{\hat S}$-orbit in $\textbf{m}^2$ is at most $\mu-2$. It is equal to $\mu-2$ if the germ $f$ lies in the ideal $\nabla_{f|_{x_1}}.$ By K.~Saito's theorem, this condition is satisfied for all singularities that can be reduced to a quasi-homogeneous form by an admissible diffeomorphism of the argument space, in particular, for all simple ones.
	\end{remark}
	
	\textbf{Proof of Lemma~\ref{Lyas}}. Let us take a miniversal deformation $M$ of the germ $f$ with the base $\mathbb{C}^{\mu}\simeq(\nu_0, \nu_1, \dots, \nu_{\mu-1})$ in the form \eqref{miniversal}. Let $\nu'\in\Sigma^{f}\cap\Sigma^{f|_{\hat S}}$ be a point (close to $0\in\mathbb C^{\mu}$) such that the corresponding function $f_{\nu'}$ has a critical point of type $B_{2}$ with value $0$. We choose a point ${\tilde\nu}\in\mathbb{C}^{\mu}\setminus\Sigma$ very close to $\nu'$ and consider the line $L_{\tilde\nu}:f_{\tilde\nu}-c, c\in\mathbb{C}$. We can assume that the line $L_{\tilde\nu}$ is transversal to $\Sigma$, since this can always be achieved by a small perturbation of the point $\tilde\nu$. Then on $L_{\tilde\nu}$ in a sufficiently small neighborhood of zero (we introduce the coordinate $c$ on $L_{\tilde\nu}$), there is only a pair of critical values $\theta$ and $\theta'$ of the restrictions $p|_{A_{\tilde\nu}}$ and $p|_{{A_{\tilde\nu}}\cap S}$. They are taken at the critical points into which the $B_2$ singularity has split. Connect the points $\theta$ and $\theta'$ to the point $0\equiv\tilde{X}\in L_{\tilde\nu}$ by distinguished paths $\gamma, \gamma'$ (not leaving the small neighborhood under consideration) and define the vanishing elements $\{\alpha, \tau(\alpha')\}\in H_{N-1}(A_{\tilde\nu}\cap \tilde{X}\setminus S)$. Then $l(\alpha,\alpha')=\pm1$ and $\operatorname{Var}_{\gamma'}(\alpha)=\pm\tau(\alpha')$ hold (see Example~\ref{po}).
	
	Now we take an arbitrary deformation $F(x,\lambda)$ of the form \eqref{obsch} with the base $\mathbb{C}^l\simeq(\lambda_0, \lambda_1, \dots, \lambda_{l-1})$. Let us write it in the form $F'(x,\lambda')+\lambda_0$ and consider the truncated deformation $F'$. By the definition of versality, $F'(x,\lambda')=M(\mathcal{G}_{\lambda'}(x), \chi_{F'}(\lambda')),$ where $\chi_{F'}(\lambda')=(\chi_{F'}^0(\lambda'), \chi_{F'}^1(\lambda'), \dots, \chi_{F'}^{\mu-1}(\lambda'))$. Consequently, we have $F(x,\lambda)=M(\mathcal{G}_{\lambda'}(x), \chi_{F}(\lambda))$, where $\chi_{F}(\lambda)=(\chi_{F'}^0(\lambda')+\lambda_0, \chi_{F'}^1(\lambda'), \dots, \chi_{F'}^{\mu-1}(\lambda'))$. Then, for any $\lambda=(\lambda_0, \lambda')$, the line $L_{\lambda}=\{f_{\lambda}-c \mid c\in\mathbb{C}\}\subset\mathbb C^l$ is mapped under $\chi_{F}$ isomorphically onto the line $L_{\chi_F(\lambda)}\subset\mathbb C^{\mu}$ (which consists of the same functions up to the boundary-preserving diffeomorphism $\mathcal{G}_{\lambda'}$).
	
	Let us fix a generic value of the parameter $\lambda=\tilde\lambda$. We connect the images of the points $\tilde\nu$ and $\chi_F(\tilde{\lambda})\equiv\mathring\nu$ under the projection $\mathbb{C}^\mu\rightarrow\mathbb{C}^{\mu-1}$ to the base of the truncated deformation by a path lying entirely in the complement $\Theta\cap\mathbb{C}^{\mu-1}$ to the bifurcation diagram of functions (which consists of all $\nu$ for which the total number of distinct critical values of the restrictions $p|_{A_{\nu}}$ and $p|_{A_{\nu}\cap S}$ is less than $\mu$). This path determines a path $\xi$ in the space of lines of the form $f_{\nu}-c$ transversal to $\Sigma(M)$, connecting the line $L_{\tilde\nu}$ with the line $L_{\mathring\nu}$.
	
	Let us choose an arbitrary distinguished system of paths on the line $L_{\mathring\nu}$ and determine its basis $\{\Delta_{s}, \tau(\Delta_{t}')\}$ of the homology group $H_{N-1}(A_{\mathring\nu}\cap X\setminus S)$. On the other hand, we extend the paths $\gamma, \gamma'$ to a distinguished system on $L_{\tilde\nu}$ and construct a basis of the group $H_{N-1}(A_{\tilde\nu}\cap X\setminus S)$. The transition from one basis $\mathbb{Z}^{\mu}\simeq H_{N-1}(A\cap X\setminus S)$ to another induced by the path $\xi$ is carried out by generalized Picard--Lefschetz transformations; moreover, the tube $\alpha'$ (i.e., its class) transforms into a non-zero linear combination of the tubes $\tau(\Delta_{t}')$, and the cycle $\alpha$ transforms into a combination of the cycles $\Delta_{s}$ and, possibly, some tubes $\tau(\Delta_{t}')$. 
	
	Suppose that for all $s,t$, one has $l(\Delta_{s},\Delta_{t}')=0$. Then the cycle $\xi_{*}\alpha\in H_{N-1}(A_{\mathring\nu}\cap X\setminus S)$ belongs to the linear span of the cycles $\Delta_{s}$, and its variation along the loop $\xi_{*}\gamma'$ is equal to zero. This leads to a contradiction since $\xi_{*}\tau(\alpha')\neq0$. $\Box$
	
	\section{Stable wave fronts}\label{UsT}
	
	\subsection{Generating functions and generating families}\label{Pro} 
	
	Let $\Xi$ be an algebraic hypersurface in $\mathbb C^{N+1}$, and $\overline\Xi\subset\mathbb{CP}^{N+1}$ be its projective compactification. The {\em generating function} (see~\cite{VVA},~Ch.~III,~\S 4) of the surface $\overline\Xi$ at a non-singular point $a\in\overline\Xi$ is a holomorphic algebraic function $f:\mathbb C^{N}\to \mathbb C$ whose graph $c=f(x)$ in some {\em adapted} affine coordinate system $c,x$ centered at $a$ defines the germ of $\overline\Xi$. Being adapted means that the hyperplane $c=0$ is tangent to $\overline\Xi$ at the point $a$, which is a critical point for the function $f$. The existence of a generating function follows from the implicit function theorem. Any two generating functions defined in this way at the same point are $\mathcal{K}$-equivalent. If the surface $\overline\Xi$ is {\em real}, then a real generating function is considered at any non-singular point $a\in\operatorname{Re}\overline\Xi$.  
	
	A point $a\in\overline\Xi$ is called {\em non-degenerate} if the tangent hyperplane to $\overline\Xi$ at this point is not tangent at any other point in a sufficiently small neighborhood of $a$. Note that in the case of a real surface $\overline\Xi$, non-degeneracy is verified algebraically for its complexification. The point $a$ is non-degenerate if and only if the singularity of the generating function $f$ at the point $a$ is isolated. 
	
	Let us generalize the notion of a generating function to the case of a pair $(\overline\Xi, \overline\Xi\cap S)$, where $S$ is the hyperplane at infinity in $\mathbb {CP}^{N+1}$ for $\mathbb{C}^{N+1}\supset\Xi$. Let $\overline\Xi$ be transversal to $S$ at a point $a\in\overline\Xi\cap S$ which is non-singular for $\overline\Xi$. We call such a point {\em suitable}. An adapted (local) affine coordinate system $c,x$ centered at a suitable point $a$ is chosen such that $S$ is given by the equation $x_1=0$ and the hyperplane $c=0$ (in which we introduce coordinates $x=(x_1, \dots, x_N)$) is non-transversal to $\overline\Xi\cap S$ at the point $a$. Then, by the implicit function theorem, the germ of $\overline\Xi$ at the point $a$ is also given by the graph $c=f(x)$ of an algebraic function, but the germ $f$ here is considered as a boundary singularity germ ($df|_{\hat S}=0$) with the boundary $\hat S:x_1=0$. It is natural, moreover, to choose the hyperplane $c=0$ to be tangent to $\overline\Xi$ if the point $a$ is non-degenerate for $\overline\Xi$ (in this case $df(a)=0$). In the case of a degenerate point, the hyperplane $c=0$ is chosen to be transversal to $\overline\Xi$, which corresponds to a purely boundary singularity ($df(a)\neq0$). The function $f$ defined in this way will be called the {\em generating function} of the pair $(\overline\Xi, \overline\Xi\cap S)$ at a suitable (infinite) point $a$. Any two such generating functions defined at the same point are $\mathcal{K}_{\hat S}$-equivalent. The point $a$ is a non-degenerate point for $\overline\Xi\cap S$ if and only if the boundary singularity $f$ is isolated. Such a suitable point will be called {\em simple}.
	
	Let us denote by $\Lambda_{f}$ the linear deformation $f-\xi_{1}x_{1}-\xi_{2}x_{2}-\dots-\xi_{N}x_{N}-\xi_0$ of the boundary singularity $f$ with parameters $\xi_0, \xi_1, \dots, \xi_N$. Any hyperplane $X_{\tilde{\xi}}\subset\mathbb{C}^{N+1}_{c,x}$ close to $X_0:c=0$ is given by the equation $c=\tilde{\xi}_{1}x_{1}+\tilde{\xi}_{2}x_{2}+\dots+\tilde{\xi}_{N}x_{N}+\tilde{\xi}_0$, where $\tilde{\xi}=(\tilde{\xi}_0, \tilde{\xi}_{1}, \dots, \tilde{\xi}_{N})$. We identify the point $\tilde{\xi}$ of the base of the deformation $\Lambda_{f}$ with its corresponding hyperplane $X_{\tilde{\xi}}$. Then the deformation parameters determine local affine coordinates centered at the point $X_{0}\in(\mathbb{CP}^{N+1})^*,$ where $(\mathbb{CP}^{N+1})^*$ is the projectively dual space to the projective closure of $\mathbb{C}^{N+1}$.
	
	\begin{definition}
		The \textbf{projective generating family} of the pair $(\overline\Xi, \overline\Xi\cap S)$ at a suitable point is the linear deformation $\Lambda_f$ of the generating function $f$.
	\end{definition}
	
	The concept of a generating function traditionally refers not only to a hypersurface, but also to the discriminant set of its generating family.
	
	Let us denote by $A_0$ (cf. Sec.~\ref{Rassl}) the germ of $\overline\Xi$ at a simple point $a$ in an adapted coordinate system. Each hyperplane $X_{\tilde\xi}$ non-transversal to $A_0\cup S$ corresponds to a non-generic value of the parameter $\xi=\tilde\xi$, and the local discriminant $\Sigma(\Lambda_f)$ defines the hypersurface germ $\Sigma(\Xi_a)_{X_0}\subset(\mathbb{CP}^{N+1})^*$ at the point $X_0$. In the case where $df(a)=0$, it has two irreducible components: the first, $\Sigma^{f}(\Lambda_f)$, is the germ of the projectively dual surface to $A_0$, denoted by $(A_0)^*$, and the second, $\Sigma^{f|_{\hat S}}(\Lambda_f)$, is diffeomorphic to $(A_0\cap S)^*\times\mathbb C_{\xi_1}$; in the case where $df(a)\neq 0$, only the latter remains. The germ of the global discriminant $\Sigma(\Xi)_{X_0}$ may also contain other components related to the non-transversality of hyperplanes $X_{\xi}$ close to $X_0$ to an arbitrary analytic Whitney stratification (see~\cite{VVA}, Ch.~I,~\S8) of the set $\overline{\Xi}\cup S$ at points outside $A_0$ --- these are not considered here.
	
	Let the hypersurface $\Xi\subset\mathbb{R}^{N+1}$ be smooth and its closure $\overline\Xi$ be transversal to $S$. The closure of the set of all hyperplanes in $\mathbb{RP}^{N+1}$ non-transversal to the set $\overline\Xi\cup S$ will be called the {\em wave front} for $\Xi$ (following~\cite{VVA}, Ch.~V,~\S3). The local wave front $\Sigma_{\mathbb R}(\Xi_a)$ is given by the germ of the real discriminant $\Sigma_{\mathbb R}(\Lambda_f)$ of the generating family. We are interested in local fronts for simple points $a\in\overline{\Xi}\cap S$.
	
	\begin{figure}[H]
		\center{\includegraphics[width=1\linewidth]{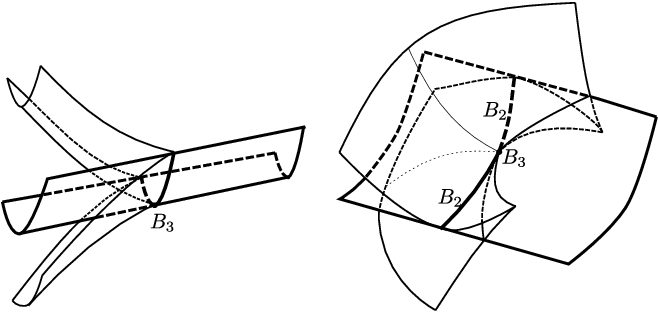}}
		\caption{Unstable and stable fronts near a singular point of type $B_3$}
		\label{fr}	
	\end{figure}
	
	\begin{definition}\label{usto}
		The germ of a boundary singularity $f$ at $0\in\mathbb C^N$ with $df(0)=0$ (for $df(0)\neq0$, see Definition~\ref{ust'}) is called \textbf{stable} if its linear deformation $\Lambda_f$ is $\mathcal{K}_{\hat S}$-versal (cf. Definition 3, \S 3, Ch.~V of~\cite{VVA}).
	\end{definition}
	The singularity of a stable germ is of finite multiplicity and is automatically isolated by the local version of Hilbert's Nullstellensatz.
	
	The $\mathcal{K}_{\hat S}$-classification of real singularities determines a certain classification of singular points of wave fronts. A front germ is called stable if the germ of its generating function is stable.
	
	\subsection{Stable germs and parabolic points}\label{UP}
	
	A point $a\in\overline{\Xi}$ is called {\em parabolic} if it is a critical point for the Gauss map $\overline{\Xi}\to(\mathbb{CP}^{N+1})^*$ that associates each non-singular point with its tangent hyperplane. A parabolic point is non-degenerate if and only if the singularity of the generating function for $\overline{\Xi}$ at the point $a$ is an isolated non-Morse singularity. 
	
	\begin{PREDLOZHENIE}\label{linvers} 
		If the germ of a boundary singularity $f$ with $df(0)=0$ is stable, then the set of points in $A_{\tilde\lambda}\cap S$ that are non-parabolic for $A_{\tilde\lambda}$ has codimension zero in $A_{\tilde\lambda}\cap S$. 
	\end{PREDLOZHENIE}
	
	\textbf{Proof.} If almost all (and hence, all) points of the set $A_0\cap S$ are parabolic for $A_0$, then the generating function of the surface $A_0$ at any such point has a degenerate second differential. This implies that the linear deformation does not intersect the $B_2$ stratum, which contradicts its $\mathcal{K}_{\hat S}$-versality. Since versality is equivalent to transversality to the orbit, the statement also holds for a sufficiently small $\tilde\lambda$. $\Box$ 
	
	In the case of a purely boundary singularity $f$, the versality of the linear deformation is not responsible for the non-parabolicity of the points of the set $A_{\tilde\lambda}\cap S$ for $A_{\tilde\lambda}$, since the discriminant set $\Sigma^f(\Lambda_f)$ is empty (all hyperplanes close to the adapted asymptotic hyperplane $c=0$ are transversal to $A_0$, and hence to $A_{\tilde\lambda}$ for small $\tilde\lambda$). However, even in this case, the non-parabolicity condition is generic. For example, if the germ $A_0$ in an adapted affine coordinate system is given by the equation $c=x+y^a$, then all points of $A_0$ are degenerate parabolic points. The diffeomorphism $x\mapsto x, y\mapsto y-\varepsilon x^2$ transforms $A_0$ into a surface whose set of (again, degenerate) parabolic points forms a curve $y=\varepsilon x^2$ on $A_0$, intersecting $S:x=0$ only at the origin.   
	
	\begin{PREDLOZHENIE}\label{parab}
		Let $f:(\mathbb{C}^N, \hat{S}, 0)\to(\mathbb C, 0)$ be the germ of an isolated boundary singularity. Then, by a boundary-preserving diffeomorphism of the argument space arbitrarily close to the identity, the germ $f$ can be reduced to a form such that the set of points in $A_{\tilde\lambda}\cap S$ that are non-parabolic for $A_{\tilde\lambda}$ has codimension zero in $A_{\tilde\lambda}\cap S$. 
	\end{PREDLOZHENIE}
	
	\textbf{Proof.} At almost all $x\in\hat S\equiv\{x_1=0\}$ (such that $(x, f(x))$ is a non-parabolic point of $A_0\cap S$), one has:
	$$\operatorname{det}\operatorname{Hess}f(x)=\operatorname{det}\operatorname{Hess}f|_{\hat S}(x)\cdot\left(\operatorname{a}(x)-\operatorname{b}(x)(\operatorname{Hess}f|_{\hat S}(x))^{-1}\operatorname{b}^{T}(x)\right),$$ 
	where $\operatorname{a}=\frac{\partial^2f}{\partial x_1^2}$ and $\operatorname{b}=(\frac{\partial^2 f}{\partial x_1\partial x_2}, \dots, \frac{\partial^2 f}{\partial x_1\partial x_N})$.
	
	Consider the following subset of the orbit of $f$: 
	$$g_{\varepsilon}(x)=f(\mathcal{G}_\varepsilon(x))=f(x_1, x_2-\varepsilon x_1^2, x_3, \dots, x_N).$$ 
	It is verified directly that for $x\in\hat S$, the maps $\operatorname{Hess}g_{\varepsilon}|_{\hat S}(x)$ and $\operatorname{b}_{g_\varepsilon}(x)$ do not depend on $\varepsilon,$ but $\operatorname{a}_{g_\varepsilon}(x)=\operatorname{a}_f(x)-2\varepsilon\frac{\partial f}{\partial x_2}(x)$. Since the singularity $f|_{\hat S}$ is isolated, the function $\frac{\partial f}{\partial x_2}(x)$ is non-zero almost everywhere on $\hat S$. Then, by Sard's lemma, for almost all $\varepsilon$, zero is a regular value of the function $\operatorname{det}\operatorname{Hess}g_{\varepsilon}(x)$ on a Zariski-open subset of $\hat S$, which implies the required statement for $\tilde\lambda=0$. Finally, for any generic fixed $\varepsilon$, the same will hold for all functions $f_{\tilde\lambda}$ with sufficiently small $\tilde\lambda$, since regularity is an open property. $\Box$
	
	\begin{definition}\label{ust'}
		The germ of a purely boundary singularity is called \textbf{stable} if its linear deformation is $\mathcal{K}_{\hat S}$-versal and Proposition~\ref{parab} holds for it. 
	\end{definition}
	
	This definition is consistent with Definition~\ref{usto}, since by Proposition~\ref{linvers}, a stable germ $f$ with $df(0)=0$ is already in the form required by Proposition~\ref{parab}.
	
	\subsection{Simple stable real singularities}\label{simplust}
	
	In the quasi-homogeneous normal form (see Table~\ref{tablo}, \cite{Lyash}), the germ of any simple boundary singularity $f$ lies in its boundary gradient ideal $\nabla_{f|_{x_1}}$. The condition $f\in\nabla_{f|_{x_1}}$ is invariant under the action of diffeomorphisms of $\mathbb{R}^N$ that preserve the boundary $\hat{S}:x_1=0$. Therefore, in the case of a simple boundary singularity, the $\mathcal{K}_{\hat S}$-versality of its deformation is equivalent to the ordinary ($\mathcal{R}_{\hat S}$-)versality. On the other hand, since for quasi-homogeneous singularities the $\mathcal{K}_{\hat S}$- and $\mathcal{R}_{\hat S}$-orbits coincide up to sign, it is sufficient to study Arnold's normal forms.
	
	Let the germs of classes ${B}_{k}, C_k, F_4$ at the point $0\in\mathbb{R}^N$ be given in the form $\pm \eta_{1}^k \pm \eta_{2}^2 \pm \eta_{3}^2 \dots \pm \eta_{N}^2$, $\pm \eta_{1}\eta_{2} \pm \eta_{2}^k \pm \eta_{3}^2 \dots \pm \eta_{N}^2$, and $\pm \eta_{1}^2 \pm \eta_{2}^3 \pm \eta_{3}^2 \dots \pm \eta_{N}^2$, respectively, in some (generally speaking, curvilinear) coordinate system, where the boundary germ is everywhere defined by the condition $\eta_{1}=0$.
	
	\begin{PREDLOZHENIE}\label{ust}
		The germ of class $B_{k}$ is stable if and only if the first $k-1$ derivatives of the curve $\{\eta_{2}=\dots=\eta_{N}=0\}$ are linearly independent at $0\in\mathbb{R}^N.$ The germ of class $C_{k}$ is stable if and only if the first $k-2$ derivatives of the curve $\{\eta_{1}=\eta_{3}=\dots=\eta_{N}=0\}$ are linearly independent at $0\in\mathbb{R}^N.$ 
	\end{PREDLOZHENIE}
	
	\textbf{Proof.} 
	Consider a germ $f(\eta)$ of type $B_k$. Let $\mathcal{G}$ be a boundary-preserving diffeomorphism of $\mathbb{R}^N$ that brings the germ in the normal form $f(\eta)$ into a stable form $\mathring{f}(x)=f(\mathcal{G}(x))$. Consider the linear (and therefore versal) deformation $\Lambda(x,\xi)=\mathring{f}(x)+\xi_1x_1+\dots+\xi_Nx_N+\xi_0$ of the germ $\mathring{f}(x)$. Then $\Phi(\eta,\xi)\equiv\Lambda(\mathcal{G}^{-1}(\eta),\xi)$ is a versal deformation of the germ $f(\eta)$. Let us now take a miniversal deformation of the form \eqref{miniversal} $F(\eta,\lambda)=f(\eta)+\lambda_{k-1}\eta_{1}^{k-1}+\dots+\lambda_1\eta_1+\lambda_0$ of the germ $f(\eta)$. By the definition of versality \eqref{VerS}, there exist such $\mathcal{H}$ and $\varphi$ that $F(\mathcal{H}(\eta,\xi),\varphi(\xi))=\Phi(\eta,\xi)$. Differentiating with respect to $\xi$ and substituting $\xi=0,$ we obtain:
	\begin{equation} \label{ioi}
		\operatorname{grad}f(\eta)\cdot\dfrac{\partial \mathcal{H}}{\partial \xi}(\eta,\xi)\bigg|_{\xi=0}+(1,\eta_{1},\dots,\eta_{1}^{k-1})\cdot d\varphi\bigg|_{\xi=0} =(1,x_1(\eta),x_{2}(\eta),\dots,x_{N}(\eta)),
	\end{equation}
	where $(x_1(\eta),x_2(\eta),\dots,x_N(\eta))=\mathcal{G}^{-1}(\eta)$. 
	
	Projecting this equality into $\mathbb{R}\{\eta\}_{0}/\nabla_{f|_{\eta_1}}$, we verify that $\operatorname{rk}d\varphi|_{\xi=0}=k,$ that is, $\varphi$ is a submersion. Indeed, since $\mathcal{H}_{\xi}(\eta)$ are boundary-preserving diffeomorphisms, all elements of the first row of the matrix $\dfrac{\partial \mathcal{H}}{\partial \xi}(\eta,\xi)\big|_{\xi=0}$ belong to the ideal $(\eta_1)$, and thus all components of the map $\operatorname{grad}f(\eta)\cdot\dfrac{\partial \mathcal{H}}{\partial \xi}(\eta,\xi)\big|_{\xi=0}$ belong to the ideal $\nabla_{f|_{\eta_1}}$. Meanwhile, the classes of the set on the right-hand side of \eqref{ioi} generate the $k$-dimensional vector space $\mathbb{R}\{\eta\}_{0}/\nabla_{f|_{\eta_1}}$ by the versality theorem (\cite{AVG},~\S 8.3).
	
	Let us return to \eqref{ioi} and substitute $\eta_2=\dots=\eta_N=0$. Now the right-hand side of this equality contains exactly (1, the parametrization of the curve $\{\eta_{2}=\dots=\eta_{N}=0\}$). Differentiating this equality further with respect to $\eta_1$ and substituting $\eta_{1}=0$, we obtain the required result. Indeed, the components of the map $\operatorname{grad}f(\eta)\cdot\dfrac{\partial \mathcal{H}}{\partial \xi}(\eta,\xi)\big|_{\xi=\eta_2=\dots=\eta_N=0}$ belong to the ideal $(\eta_1^k),$ and the first $k-1$ derivatives of the curve $\{\eta_{2}=\dots=\eta_{N}=0\}$ coincide with the rows of a matrix of rank $k-1$ (the matrix $d\varphi|_{\xi=0}$ from which the first row and the first column $(1,0,\dots,0)$ have been removed). Conversely: if the derivatives are linearly independent, then $\operatorname{rk}d\varphi|_{\xi=0}=k$, and by the versality theorem, the deformation $\Phi$ is versal, which implies that the linear deformation $\Lambda$ is also versal. 
	
	The case $C_k$ is analyzed analogously. Namely, we arrive at the equation
	
	\begin{equation}\label{qq}
		\operatorname{grad}f(\eta)\cdot\dfrac{\partial \mathcal{H}}{\partial \xi}(\eta,\xi)\bigg|_{\xi=0}+(1,\eta_{2},\dots,\eta_{2}^{k-1})\cdot d\varphi\bigg|_{\xi=0} =(1,x_1(\eta),x_{2}(\eta),\dots,x_{N}(\eta)),
	\end{equation}
	where $\operatorname{rk}d\varphi|_{\xi=0}=k$. Since $\mathcal{G}^{-1}(\eta)$ is boundary-preserving, $x_1(\eta)|_{\eta_1=0}=0$ holds. Since $\eta_1 \pm \eta_2^{k-1}\in\nabla_{f|_{\eta_1}},$ the second column of the matrix $d\varphi|_{\xi=0}$ has a unique non-zero element in the last row. Let us substitute $\eta_1=\eta_3=\dots=\eta_N=0$ into \eqref{qq}. Differentiating with respect to $\eta_2$ and substituting $\eta_2=0$, we obtain the required result. $\Box$
	
	\textbf{Examples.} The germs of classes $B_2, C_2, C_3$ of two variables are stable in their normal forms. In the case where $k>2$, the germ $B_k$ is stable for $\eta_{1}=x_1$, $\eta_{2}=x_2-x_1^2$, $\eta_{3}=x_3-x_1^3$, $\dots$, $\eta_{k-1}=x_{k-1}-x_1^{k-1}$, $\eta_{k}=x_{k}$, $\dots$, $\eta_{N}=x_{N}$. For $k>3$, the germ $C_k$ is stable for $\eta_{1}=x_1$, $\eta_{2}=x_2$, $\eta_{3}=x_{3}-x_2^2$, $\dots$, $\eta_{k-1}=x_{k-1}-x_2^{k-2}$, $\eta_{k}=x_{k}$, $\dots$, $\eta_{N}=x_{N}$. The germ $F_4$ in the form $\pm x_1^2 \pm x_2^3 \pm (x_{3}-x_1x_2)^2 \pm x_{4}^2 \pm \dots \pm x_{N}^2$ is stable.
	
	Thus, stable germs of classes $B_k, C_k$ for $k>2$ first appear in functions of $N=k-1$ variables, while the germ of class $F_4$ appears in three variables. The classification of simple stable purely boundary germs (with $df(0)\neq 0$) reduces to the classification of simple stable Legendrian germs (see~\cite{AVG},~Sec.~21.8).
	
	\section{Ramification of the volume function close to an asymptotic hyperplane and quasilacunas}\label{vetv}
	Consider the problem of the ramification of the volume cut off by a hyperplane from a domain in a Euclidean space $\mathbb R^{N+1}$ with a smooth semi-algebraic boundary. An affine hyperplane $Y$ is called {\em asymptotic} if its projective closure is non-transversal to the ``part at infinity'' of the boundary. 
	
	In local affine coordinates centered at a simple point (see Sec.~\ref{Pro}), the closure of the boundary is given by the germ of the graph $c=f(x)$ of a generating function with an isolated real boundary singularity. We study the variations of the local integration contours $\Omega_i$ (see Fig.~\ref{F17}) along paths that originate in a generic real hyperplane $X_{\tilde\xi}$ and do not leave a small neighborhood of $Y$. 
	
	It is sufficient to restrict our attention to paths on a line in general position in $(\mathbb{CP}^{N+1})^*$, which models a pencil of parallel hyperplanes in $\mathbb C^{N+1}$. Such a line is embedded into the base of a versal deformation of the germ $f$ of the form \eqref{obsch}. By means of the versality equation, the problem reduces to studying the ramification in a pencil of hyperplanes of the form $c=\operatorname{const}$ for the deformed graph germ $A_{\lambda}:c=f_{\lambda}(\mathcal{G}_{\xi}x)$. Here $f_{\lambda}(x)=f(x)-(\xi_1x_1+\dots+\xi_Nx_N+\xi_0)$, and $\mathcal{G}_{\xi}$ is a boundary-preserving diffeomorphism. A generic value $\tilde \xi$ corresponding to the hyperplane $X_{\tilde \xi}: c=\tilde \xi_1x_1+\dots+\tilde \xi_Nx_N+\tilde \xi_0$, maps under the inducing map to the value $\lambda(\tilde \xi)=\tilde\lambda$. The coordinate $c$ is introduced on the line $L_{\tilde\lambda}=\{f_{\tilde\lambda}-c, c\in\mathbb C\}$ in the base of the versal deformation and in the new coordinates $X_{\tilde\xi}$, will henceforth be denoted by $X_0$.
	
	For any real non-discriminant perturbation $f_{\tilde{\lambda}}$, a set of germs at the point $X_0$ of the \textit{volume functions} $\{V_i(X)\}$ is introduced. These germs are defined by the integrals of the volume form over the cycles $\{\Omega_i(X)\}$. Each integration contour $\Omega_i(X_0)$ corresponds to a cycle $\Pi_i(X_0)$ given by a set of connected components of the set $\operatorname{Re}\{f_{\tilde\lambda}=0\}.$ The stability of $f$ introduced in Sec.~\ref{UsT} guarantees that the integral of the volume form over any vanishing integration contour (which form a distinguished basis of the group $H_{N+1}(D\setminus S, A\cup X\setminus S)$, see Sec.~\ref{j}) is not identically equal to zero, see Lemma~\ref{Vol1}. It is shown (Sec.~\ref{osn}) that the triviality of the class $\Pi_i$ is necessary (and sometimes sufficient) for the local finite-sheetedness of the volume function $V_i$. Such connected components of the complement to $\Sigma_{\mathbb R}$ (\textit{quasilacunas}), for which all functions $V_i$ are finite-sheeted, are enumerated for versal deformations of simple boundary singularities in Sec.~\ref{quasisimpl}. The problem of the ramification of the total cut-off volume lying on one side of a cutting hyperplane is related to the generalized local even Petrovsky cycle $\Pi$. This cycle is defined under the condition $\operatorname{Re}\{f_{\tilde\lambda}=0\}\cap S=\varnothing$ and is cooriented by the vector $\operatorname{grad}f_{\tilde\lambda}$. The class of the cycle $\Pi$ in the homology group $H_{N-1}(A\cap X\setminus S, \partial D)$ is computed (see Sec.~\ref{Petrov}) in terms of the matrix of generalized intersections introduced in Sec.~\ref{43}.
	
	\subsection{Distinguished basis of vanishing integration contours}\label{j}
	
	The linear deformation $f_{\tilde{\lambda}}-\xi_{1}x_{1}-\xi_{2}x_{2}-\dots-\xi_{N}x_{N}-\xi_0$ of the function $f_{\tilde{\lambda}}$ with parameters $\xi_0, \xi_1, \dots, \xi_N$ is denoted by $\Lambda_{f_{\tilde{\lambda}}}$. Let $B\subset(\mathbb{CP}^{N+1})^*$ be a $\delta$-ball centered at $X_{0}$ (see Sec.~\ref{Rassl}). Simple loops in $B_{\delta}=B\cap L_{\tilde{\lambda}}$ around the critical values of the restrictions $p|_{A}$ and $p|_{A\cap S}$ generate the fundamental group $\pi_{1}(B\setminus\Sigma(\Lambda_{f_{\tilde{\lambda}}}))$ (see~\cite{Zarisk}, \cite{GM}).
	
	For any hyperplane $X\in B\setminus\Sigma(\Lambda_{f_{\tilde{\lambda}}})$, consider the chain of isomorphisms (see~\cite{VVA}, Prop. 4, p. 123):
	\begin{equation}\label{1}
		H_{N+1}(D\setminus S, A\cup X\cup\partial D\setminus S)\rightarrow H_{N}(A\cup\partial D\setminus S, (A\cap X)\cup\partial D\setminus S)\rightarrow H_{N-1}(A\cap X\setminus S, \partial D\setminus S).
	\end{equation}
	
	We also consider an analogous chain without reduction modulo $\partial D$:
	\begin{equation}\label{-1}
		H_{N+1}(D\setminus S, A\cup X\setminus S)\rightarrow H_{N}(A\setminus S, A\cap X\setminus S)\rightarrow H_{N-1}(A\cap X\setminus S).
	\end{equation}
	
	The first homomorphism of this chain is an isomorphism, and the second is an isomorphism for $N>2$.
	
	\begin{PREDLOZHENIE}\label{Izol}
		In the case of an arbitrary isolated boundary singularity and its deformation of the form \eqref{obsch}, the basis of the group $H_{N+1}(D\setminus S, A\cup X\setminus S)$ is realized by a set of vanishing cycles and tubes $\{\text{Ш}_i\}$, which is mapped by the homomorphism $H_{N+1}(D\setminus S, A\cup X\setminus S)\rightarrow H_{N-1}(A\cap X\setminus S)$ into the distinguished basis of the group $H_{N-1}(A\cap X\setminus S)$.  
	\end{PREDLOZHENIE} 
	
	\textbf{Proof.} From the exact sequence of the triple $(D\setminus S, A\cup X\setminus S, X\setminus S)$, the isomorphism $$H_{N+1}(D\setminus S, A\cup X\setminus S)\simeq H_{N}({A\cup X\setminus S, X\setminus S})$$ follows, since $H_{N}(D\setminus S, X\setminus S)=0$ for $N>1$. By the excision property, $$H_{N}(A\cup X\setminus S, X\setminus S)\simeq H_{N}(A\setminus S, A\cap X\setminus S).$$ Consider a fragment of the exact sequence of the pair $(A\setminus S, A\cap X\setminus S):$    
	$$\dots\rightarrow H_{N}(A\setminus S)\rightarrow H_{N}(A\setminus S, A\cap X\setminus S)\rightarrow H_{N-1}(A\cap X\setminus S)\rightarrow H_{N-1}(A\setminus S)\rightarrow\dots$$ 
	If $N>2,$ then $H_{N-1}(A\setminus S)\simeq H_{N-1}(\mathbb{C}^N\setminus \mathbb{C}^{N-1})=0.$ For $N=2$, the group $H_{N-1}(A\setminus S)=H_{1}(A\setminus S)\simeq\mathbb{Z}$ is generated by a tube around an arbitrary point of the curve $A\cap S.$ In this case, the statement of the proposition follows from Remark~\ref{N2} on p.~\pageref{N2}. $\Box$
	
	By Thom's isotopy theorem (see~\cite{PHam}), for distinct non-discriminant hyperplanes $X_{\tilde\xi}$, the groups of each of the six types participating in \eqref{1} and \eqref{-1} are isomorphic to each other. The isomorphisms are realized by transporting cycles along paths in the base via the Gauss--Manin connection and depend on the homotopy class of the path. The monodromy operators, acting by the transport of cycles along simple (and therefore any) loops of the group $\pi_{1}(B\setminus\Sigma(\Lambda_{f_{\tilde{\lambda}}}))$, commute with the homomorphisms of the chains \eqref{1} and \eqref{-1}.
	
	For any element $\gamma\in\pi_{1}(B\setminus\Sigma;X)$, the chains $\eqref{1}$ and $\eqref{-1}$ can be arranged into a commutative diagram:
	\begin{equation}\label{cd0.5}
		\begin{CD}
			H_{N+1}(D\setminus S,A\cup X\cup\partial D\setminus S) @>>> H_{N-1}(A\cap X\setminus S,\partial D\setminus S)\\
			@VVV @VVV \\
			H_{N+1}(D\setminus S,A\cup X\setminus S) @>>> H_{N-1}(A\cap X\setminus S),\\
		\end{CD}
	\end{equation}
	whose vertical arrows are the \textit{variation} operators $\operatorname{Var}_{\gamma}$ (see Remark~\ref{rem}, and also~\cite{VVA}, Ch.~I,~\S1, Sec.~1.3) determined by the action of the loop $\gamma$.
	
	The following commutative diagrams also hold:
	\begin{equation}\label{cd1}
		\begin{CD}
			H_{N}(D\cap S,(A\cup X\cup\partial D)\cap S) @>>> H_{N-2}(A\cap X\cap S,\partial D)\\
			@VVV @VVV \\
			H_{N+1}(D\setminus S,(A\cup X\cup\partial D)\setminus S) @>>>  H_{N-1}(A\cap X\setminus S,\partial D),\\
		\end{CD}
	\end{equation}
	
	\begin{equation}\label{cd2}
		\begin{CD}
			H_{N}(D\cap S,(A\cup X)\cap S) @>>> H_{N-2}(A\cap X\cap S)\\
			@VVV @VVV \\
			H_{N+1}(D\setminus S,A\cup X\setminus S) @>>>  H_{N-1}(A\cap X\setminus S).\\
		\end{CD}
	\end{equation}
	
	The vertical arrows of both diagrams are the Leray tube operators, and the lower arrows are the isomorphisms \eqref{1} and \eqref{-1}. The upper arrow in \eqref{cd2} is an isomorphism according to~\cite{VVA}, Prop. 3, p. 122. The upper arrow in \eqref{cd1} is an isomorphism because it factors through the isomorphism $$H_{N}(D\cap S,(A\cup X\cup\partial D)\cap S)\rightarrow H_{N-1}(A\cap S,(A\cap X\cap S)\cup\partial D)$$ (see~\cite{VVA}, Prop. 3, p. 122), and zeros stand in the appropriate places of the exact sequence of the triple $$((A\cap S)\cup\partial D,(A\cap X\cap S)\cup\partial D,\partial D)\text{:}$$ $H_*(A\cap S,\partial D)\simeq H_*(B,\partial B)\simeq H_*(S^{2N-2})=0$ for $*=N-2, N-1$. Here $B$ is a standard ball in $\mathbb{R}^{2N-2},$ and the first isomorphism follows from Lemma 2.12 of~\cite{MilnOs}.
	
	The diagrams \eqref{cd1} and \eqref{cd2} can be included into one large commutative diagram by adding arrows that define the variation operators. Indeed, any loop $\gamma\in\pi_1(B\setminus\Sigma^{f|_{\hat{S}}};X)$ determines the operators $$\operatorname{Var}_{\gamma}:H_{N-2}(A\cap X\cap S,\partial D)\rightarrow H_{N-2}(A\cap X\cap S)$$ and $$\operatorname{Var}_{\gamma}:H_{N}(D\cap S,(A\cup X\cup\partial D)\cap S)\rightarrow H_{N}(D\cap S,(A\cup X)\cap S)$$ which commute with the horizontal isomorphisms. Moreover, the operators $$\operatorname{Var}_{\gamma}:H_{N-2}(A\cap X\cap S,\partial D)\rightarrow H_{N-2}(A\cap X\cap S)$$ and the right arrow in \eqref{cd0.5} (respectively, $\operatorname{Var}_{\gamma}:H_{N}(D\cap S,(A\cup X\cup\partial D)\cap S)\rightarrow H_{N}(D\cap S,(A\cup X)\cap S)$ and the left arrow in \eqref{cd0.5}) commute with the Leray tube operators. 
	
	\begin{figure}[H]
		\center{\includegraphics[width=0.8\linewidth]{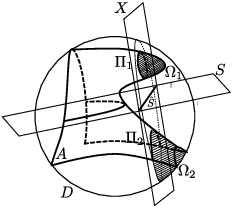}}
		\caption{The actors involved in the definition of the volume function}
		\label{F17}	
	\end{figure}
	
	\subsection{The volume function and the behavior of the integral over a vanishing tube}
	
	Consider any real non-discriminant perturbation $f_{\tilde\lambda}$. Define a set of classes $\{\Omega_{i}\}\subset H_{N+1}(D\setminus S, A\cup X_0\cup\partial D\setminus S)$ as follows (see Fig.~\ref{F17}). Each such class is represented by a relative cycle $\Omega_{i}$ given by a connected domain in the space $\mathbb{R}^{N+1}\cap D$ that is not adjacent to $S$, and is bounded by the sphere $\partial D$, the hyperplane $X_0$, and the hypersurface $\operatorname{Re}A$. This domain is equipped with the orientation induced by a fixed positive orientation of $\mathbb{R}^{N+1}$. Such domains may lie on different sides of $X_0$.
	
	Let us denote by ${\Pi_i}\in H_{N-1}(A\cap X_0\setminus S, \partial D\setminus S)$ the image of the class $\Omega_{i}$ under the composition of the isomorphisms \eqref{1} (hence, it is automatically non-zero). The cycle $\Pi_i$ is realized by a subset in $\operatorname{Re}A\cap X_0$ that does not intersect the boundary $\hat S$. This subset is not always connected.
	
	\begin{figure}[H]
		\center{\includegraphics[width=0.5\linewidth]{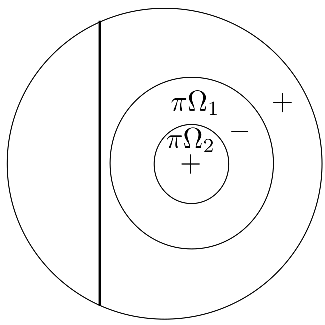}}
		\caption{A finite volume is cut off from different sides of $X_0$; the cycle $\Pi_1$ is disconnected}
		\label{df}	
	\end{figure} 
	
	\begin{example}\label{Oo}
		Consider the perturbation $f_{\tilde\lambda}=(x-\frac{3}{2}\varepsilon)^4+y^4-\varepsilon^2((x-\frac{3}{2}\varepsilon)^2+y^2)+\frac{3}{16}\varepsilon^4$ of the singularity $X_9$ with the boundary $x=0$. Fig.~\ref{df} shows the topological scheme of the projection $\pi$ of the domain $\Omega_1$, which corresponds to a disconnected cycle $\Pi_1$. The domains $\Omega_1$ and $\Omega_2$ lie on different sides of plane $X_0$.
	\end{example}
	
	Let $F$ be a deformation of the germ $f$ of the form \eqref{obsch}, and let $\Gamma$ be an arbitrary local connected component of the complement in $\Theta\cap\mathbb{R}^l$ to $\Sigma_{\mathbb{R}}(F)$. For each pair of points $\tilde{\lambda}_{1}, \tilde{\lambda}_{2}$ in $\Gamma$, the classes ${\Pi_i}$ are identified along any path in $\Gamma$ connecting these points. This identification, generally speaking, depends on the homotopy class of the path. E.~Looijenga showed in~\cite{Loo} that the connected components of the complements to the discriminants of versal deformations of ordinary simple singularities are contractible. M.~A.~Gudiev established a similar result in~\cite{MAG} for the boundary singularities $B_k$ and $C_k$. In these cases, the class $\Pi_i$ is defined invariantly inside $\Gamma$.
	
	\begin{example}\label{tyt}
		The homotopy $f_t(x,y)=(x-3\varepsilon)^4-5\varepsilon^2((x-3\varepsilon)\cos t - y\sin t)^2+4\varepsilon^4+y^4$, where $\varepsilon>0$ is close to zero, realizes a loop in the complement to $\Sigma_{\mathbb{R}}$ as $t$ varies from $0$ to $\pi$, along which the connected components of the set $\operatorname{Re}\{f_0=0\}$ are permuted: see Fig.~\ref{ovals}.
	\end{example}
	
	\begin{figure}[H]
		\center{\includegraphics[width=0.5\linewidth]{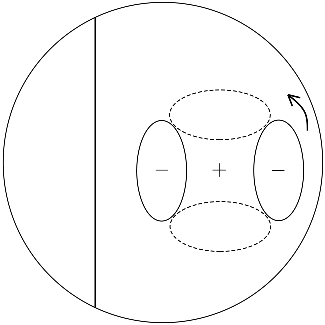}}
		\caption{Permutation of components under the action of $\pi_1(\Gamma)$}
		\label{ovals}	
	\end{figure}
	
	\begin{PREDLOZHENIE}
		In the case of a linear deformation $F=\Lambda_f$, for any $\tilde{\lambda}\in\Gamma,$ where $\Gamma$ is an arbitrary local connected component of the complement to $\Sigma_{\mathbb{R}}(\Lambda_f),$ the class ${\Pi_i}(\tilde{\lambda})$ is defined invariantly under the action of $\pi_{1}(\Gamma)$.
	\end{PREDLOZHENIE}
	
	\textbf{Proof.}
	It is sufficient to prove this statement for an ordinary singularity $f$ (without boundary), assuming that each cycle $\Pi_i$ is realized by a connected component.
	Let us denote by $D_{\varepsilon}$ a small ball around zero in the space $\mathbb{R}^{N+1}$ such that the sphere $\partial(D_{\varepsilon}\cap X_0)$ is transversal to the set $\Lambda_f(x,0)\equiv f(x)=0.$ Then the spheres $\partial(D_{\varepsilon}\cap X_{\xi})$ are transversal to the sets $\Lambda_f(x,\xi)=0$ for $\xi$ close to zero. Consider the fiber bundle $p$ of the real Milnor fibers of the ordinary singularity $f$ over the base $\Gamma$. The total space $E(p)$ consists of all possible pairs $(X_{\tilde\xi}, x)\in\Gamma\times \pi(D_{\varepsilon}\cap X_{\tilde\xi})\subset\Gamma\times\mathbb{R}^N$ such that $\Lambda_f(x,{\tilde\xi})=0;$ here $\pi:(c,x)\mapsto x$ is the projection onto $\mathbb{R}^N.$ The fibers are realized by the sets $\operatorname{Re}A_{0}\cap X_{\tilde\xi}$ in the ball $D_{\varepsilon}\subset\mathbb{R}^{N+1}$. The bundle $p$ contains a trivializable subbundle of the boundaries of the Milnor fibers --- the sets of those $x\in \pi(\partial D_{\varepsilon}\cap X_{\tilde\xi})$ for which $\Lambda_f(x,{\tilde\xi})=0$.
	
	The fundamental group $\pi_{1}(\Gamma)$ acts on the set $\{\Pi_i\}$ by permutations, without mixing cycles with a non-empty boundary in \(\partial D_{\varepsilon}\) with absolute cycles. The fact that the subbundle of the boundaries is trivializable implies that the action of $\pi_{1}(\Gamma)$ on cycles with boundary is trivial. Suppose there exists an element $\gamma\in\pi_{1}(\Gamma)$ such that the orbit of some cycle $\Pi_i$ under the action of the cyclic subgroup generated by the class of the loop $\gamma$ is non-trivial. Then for some $r$ (among which we take the minimal positive one), $h_*(\Pi_i)={\Pi_i}$ holds for $h=\gamma^r$.
	
	Let us map the cylinder $[0,1]\times\Pi_i$ over the cycle $\Pi_i$ into the set $\operatorname{Re}A_0$ by the composition of maps $\kappa\equiv\psi\circ \tilde{h}$, where $\tilde{h}$ is determined by the covering homotopy: $(t,y)\mapsto (X_{\xi(t)},h_t(y))\in E(p)$, and the map $\psi:E(p)\rightarrow\operatorname{Re}A_0$ associates a point $(X_{\tilde\xi}, x)$ with the point $(x, f(x))\in\mathbb{R}^{N+1}$. For any $t$, the cycle $\Pi_i$ is mapped diffeomorphically onto its image $\psi\circ \tilde{h}_t(\Pi_i)$. 
	
	The set $\operatorname{Re}A_{0}$ is contractible since it is diffeomorphic to $\mathbb{R}^N$. However, the set $\operatorname{Im}\kappa$ cannot have a boundary (in particular, it cannot be simply connected): otherwise, the images of the cycles $\Pi_i$ and $\gamma^{r-1}\Pi_i$ under the action of $\psi\circ\tilde{\gamma}_s$ would intersect for some $s$, which is incompatible with the smoothness of the set $\operatorname{Re}A_{0}\cap X_{\tilde\xi(s)}$. Thus, $\operatorname{Im}\kappa$ realizes a connected, non-simply-connected, closed $N$-dimensional submanifold in $\operatorname{Re}A_{0}$, which leads to a contradiction. $\Box$
	
	For each $i$, we define near $X_0$ the germ of the {\em volume function} $V_{i}$:
	\begin{definition}\label{Vol}
		\begin{equation}\label{2}
			V_{i}(X)=\int_{\Omega_{i}(X)} \omega,
		\end{equation} 	
	\end{definition}
	\noindent where $\omega$ is a rational differential form on $\mathbb{CP}^{N+1}$ whose restriction to the affine chart $\mathbb{C}^{N+1}\subset\mathbb{CP}^{N+1}$ with the hyperplane at infinity $S$ coincides with the holomorphic volume form. It is verified directly that the hyperplane $S$ is a polar set of order $N+2$ for the form $\omega$, and in our local coordinates, it is written (up to a constant factor --- the determinant of the affine transformation) in the form:
	\begin{equation}
		\omega=dc\wedge\frac{dx_1}{x_1^{N+2}}\wedge dx_2\wedge\dots\wedge dx_N.
	\end{equation}
	
	In~\cite{VVA}, Ch.~III,~\S 6, p. 201, a {\em different} differential form is considered as $\omega$.
	
	\begin{lemm}\label{Vol1}
		Let either $N$ be even, or $N$ be odd and the critical points of the function $p|_{A_{\tilde\lambda}\cap S}$ be non-parabolic points of the surface $A_{\tilde\lambda}$. Then the integral of the form $\omega$ over any basic cycle $\text{\CYRSH}(X)\in H_{N+1}(D\setminus S, A_{\tilde\lambda}\cup X\setminus S)$ (see Proposition~\ref{Izol}) is not identically equal to zero.
	\end{lemm}
	
	\begin{remark}
		If Proposition~\ref{parab} holds (or, more restrictively, if the germ $f$ is stable), then the condition that all critical points are non-parabolic is satisfied for almost all $\lambda$.
	\end{remark}
	
	The case where the element \(\text{\CYRSH}(X)\in H_{N+1}(D\setminus S, A\cup X\setminus S)\) is realized by a vanishing cycle (a ``cap'') was studied in~\cite{VVA}, Ch.~III,~\S3. We will also consider it below.
	
	Let a basis element $\text{\CYRSH}(X)\in H_{N+1}(D\setminus S, A\cup X\setminus S)$ be realized by a tube $\tau(\text{\CYRSH}'(X))$ around a vanishing cycle of the group $H_{N}(D\cap S, A\cup X)$. By the Leray residue theorem, see~\cite{Lere}, we have:
	\begin{equation}\label{vych}
		\int_{\text{\CYRSH}(X)}\omega=2\pi i\int_{\text{\CYRSH}'(X)}\operatorname{Res}\omega,
	\end{equation}
	where $\operatorname{Res}\omega$ is the {\em Leray residue form} of the form $\omega$. It can be computed by finding a form $\omega'$ with a simple pole along $S$ that is cohomologous to the form $\omega$ in the cohomology group $H^{N+1}(D\setminus S, A\cup X\setminus S).$ Namely, the relation $\omega'-\omega=d\eta$ must hold, where $[\eta]\in H^N(D\setminus S, A\cup X\setminus S).$ The value of the integral \eqref{vych} depends on the cohomology class of the form $\omega$ in the group $H^{N+1}(D\setminus S, A\cup X\setminus S)$. Therefore, although in all cases of interest to us here the germ of the set $A\cap S$ is represented by the graph of a function with a Morse singularity, the behavior of the integral depends on the geometry of the surface $A$ at this critical point.  
	
	From the exact Leray cohomology sequence \eqref{Leray}, see~\cite{Lere}, it follows that $[\operatorname{Res}\omega]\neq 0$ whenever $[\omega]\neq 0\in H^{N+1}(D\setminus S, A\cup X\setminus S).$ 
	\begin{equation}\label{Leray}
		\dots\to H^{N+1}(D,A\cup X) \to H^{N+1}(D\setminus S,A\cup X\setminus S)\overset{\operatorname{Res}}\to H^N(D\cap S,A\cup X)\to\dots 
	\end{equation}
	
	\textbf{Proof of Lemma~\ref{Vol1}.} 
	
	1) Let a cycle $\text{\CYRSH}$ vanish at a point $x_{\theta},$ which is a Morse critical point of the function $p|_{A}$ (= the function $f_{\tilde{\lambda}}$) with the value $\theta$. By a local diffeomorphism preserving the function $p|_{A}$ and defined near the point $x_{\theta}$, the germ of the surface $A$ can be reduced to the standard form $c=\theta+z_1^2+\dots+z_N^2$. For any small real $\varepsilon>0$, the cycle $\text{\CYRSH}(X_{\theta+\varepsilon})$ is realized by the interior of a Lefschetz thimble (a ``cap'') in the space \(\mathbb{R}^{N+1}\) with coordinates \(c, z_1, \dots, z_N\) (see Fig.\(~\ref{vol}\)). It is obvious that its Euclidean volume is non-zero, and consequently, our integral is non-zero as well. Moreover, it is easy to verify directly that the expansion of the function \(\int_{\text{\CYRSH}(X)}\omega\) into a power series in a neighborhood of the point \(X_{\theta}\) on \(B_{\delta}\) contains only integer exponents for even N and only half-integer exponents for odd N. 
	
	Indeed, let \(\mathcal{G}\) be a diffeomorphism of \(\mathbb C^N\), and let \(J\mathcal{G}\) be its Jacobian. Then we have:
	\[\int_{\text{\CYRSH}(X_{\theta+\varepsilon})}\omega=\int_{0}^{\varepsilon}\left(\int_{\sum z_j^2\leqslant c}J\mathcal{G}(z)dz\right)dc=\int_{0}^{\varepsilon}\sum a_k c^{\frac{N}{2}+k}dc=\sum b_k\varepsilon^{\frac{N}{2}+k+1},\]
	where \(a_k\sim \operatorname{Vol}(B^N_1)\Delta^k(J\mathcal{G}(0))\), meaning that the coefficients \(a_k\) arise from integrating the even terms of the Taylor series of the function \(J\mathcal{G}(z)\) at the origin (monomials containing variables to odd powers yield zero due to the symmetry of the integration contour), \(\operatorname{Vol}(B^N_1)\) is the volume of the unit N-ball, and $\Delta$ is the Laplace operator.
	
	From this, it is clear that the analytic continuation of the function $\int_{\text{\CYRSH}(X)}\omega$ along the loop $X_{c(t)}$, where $c(t)\equiv\{\theta+\varepsilon e^{it}, t\in[0,2\pi]\}$, maps it to itself for even $N$ and to its negative for odd $N$, which is also confirmed by the Picard--Lefschetz formula.
	
	2) Now let the cycle $\text{\CYRSH}$ vanish at a point $x_{\theta'},$ which is a Morse critical point of the function $p|_{A\cap S}$ (= the function $f_{\tilde{\lambda}}|_{\hat{S}}$) with the value $\theta'$. 
	
	\begin{figure}[H]
		\center{\includegraphics[width=1\linewidth]{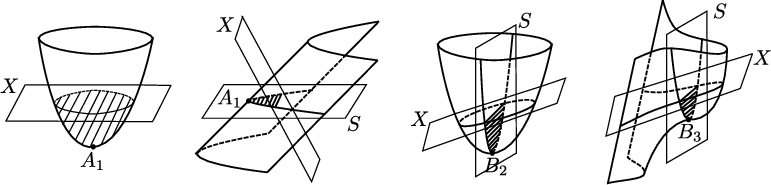}}
		\caption{Vanishing integration contours}
		\label{vol}	
	\end{figure}
	
	By the generalized Morse lemma, the equation of the surface $A$ near the point $x_{\theta'}$ can be reduced via an $S$-preserving holomorphic diffeomorphism of the space $\mathbb {C}^{N+1}$ to the equation of the graph of a function with a singularity of class $B_a$ (see~\cite{ARN}):
	\begin{equation}\label{eque}
		z=x_1^a+z_2^2+\dots+z_N^2,
	\end{equation}
	where $z=0$ is the tangent hyperplane to $A$ at the point $x_{\theta'}$ (if it is non-degenerate; in the case of a degenerate point, we require tangency to $A\cap S$ and transversality to $A$), and $x_1=0$ is the equation of $S$. Since, moreover, the diffeomorphism is required to preserve the function $p\equiv f_{\tilde\lambda}$, in the general case the equation must differ from \eqref{eque} by higher-order terms. The hyperplane $X\equiv X_{\theta'+\varepsilon}$ in our adapted coordinate system is given by the equation $z=\varepsilon+\xi x_1,$ where $\xi$ is relatively far from zero, and $\varepsilon$ is a small positive number.
	
	Consider the {\em reference example}: the germ of $A$ at the point $x_{\theta'}$ is given by equation \eqref{eque}. The contour $\text{\CYRSH}(X)$ is equal to the Leray tube $\tau(\text{\CYRSH}'(X))$ around the vanishing contour $\text{\CYRSH}'(X)$ on the boundary $S$, enclosed in $\mathbb R^N=\operatorname{Re}S$ between $A$ and $X$ (see Fig.~\ref{vol}).
	
	By Fubini's theorem, we have:
	$$\int_{\tau(\text{\CYRSH}')}\omega=\int_{|x_1|=\delta'}\left(\int_{\sum_{j=2}^{N}z_j^2\leqslant \varepsilon+\xi x_1-x_1^a}\left(\int_{x_1^a+\sum_{j=2}^{N}z_j^2}^{\varepsilon+\xi x_1}dz\right)dz_2\wedge\dots\wedge dz_N\right)\frac{dx_1}{x_1^{N+2}}.$$ 
	Let $R^2(x_1)=\varepsilon+\xi x_1-x_1^a$, $\rho^2(z)=\sum_{j=2}^{N}z_j^2$, and $dV=dz_2\wedge\dots\wedge dz_N$. Then the inner integral transforms into:
	$$\int_{B^{N-1}(R)}(R^2-\rho^2)dV=\int_0^R (R^2-\rho^2)(N-1)\operatorname{Vol}(B^{N-1}_1)\rho^{N-2}d\rho.$$ 
	Evaluating it, we obtain:
	\begin{equation}\label{eq}
		\int_{\tau(\text{\CYRSH}')}\omega=\frac{2}{N+1}\operatorname{Vol}(B^{N-1}_1)\int_{|x_1|=\delta'}(\varepsilon+\xi x_1-x_1^a)^{\frac{N+1}{2}}\frac{dx_1}{x_1^{N+2}}.
	\end{equation}
	
	\begin{enumerate}
		\item $N$ is odd. For $a=2$, the integral \eqref{eq} does not depend on the values of $\xi$ and $\varepsilon$ (and in general, on any small linear perturbations of the hyperplane $X$) and is equal to:
		$$(-1)^{\frac{N+1}{2}}\frac{4\pi i}{N+1}\operatorname{Vol}(B_1^{N-1}).$$
		
		The residue form $\operatorname{Res}\omega$ depends only on $\varepsilon$ and is equal to:
		$$(-1/\varepsilon)^{\frac{N+1}{2}}dz\wedge dz_2\wedge\dots\wedge dz_N.$$
		
		If $a=1$ or $a>N+1,$ then the integral \eqref{eq} is zero. It can be shown that for $2<a<N+2$, the integral \eqref{eq} is identically equal to zero only in exceptional cases.
		
		\item For even $N$, the integral \eqref{eq} is non-zero for all $a$, since the power series expansion of the function $(\varepsilon+\xi x_1-x_1^a)^{\frac{N+1}{2}}$ near $x_1=0$ contains monomials of all degrees with non-zero coefficients. Moreover, $\varepsilon$ enters the coefficient of $x_1^{N+1}$ only with negative half-integer exponents.
		
		Let us set $m=\frac{N+1}{2}$. For $a=2$, we have:
		$$(\varepsilon+\xi x_1-x_1^2)^{m}=\varepsilon^{m}\left(1+\frac{\xi x_1-x_1^2}{\varepsilon}\right)^{m}=\sum_{k=0}^{\infty}\binom{m}{k}\varepsilon^{m-k}(\xi x_1-x_1^2)^k.$$
		Now,
		$$(\xi x_1-x_1^2)^k=\sum_{j=0}^k\binom{k}{j}(\xi x_1)^{k-j}(-x_1^2)^j=\sum_{j=0}^k\binom{k}{j}(-1)^j\xi^{k-j}x_1^{k+j}.$$
		
		The coefficient $\alpha_{N+1}$ of $x_1^{N+1}$ is given by those pairs $(j,k)$ such that $k+j=N+1.$ Since $0\leqslant j \leqslant k,$ we obtain:
		$$\frac{N+1}{2}\leqslant k\leqslant N+1.$$
		
		The minimal such integer $k$ is equal to $\frac{N+2}{2}.$
		
		Then,
		$$\alpha_{N+1}(\varepsilon)=(-1)^{N+1-k}\sum_{k=\frac{N+2}{2}}^{N+1}\binom{m}{k}\varepsilon^{m-k}\binom{k}{N+1-k}\xi^{2k-(N+1)},$$
		and $\alpha_{N+1}(\varepsilon)$ can be written in the form:
		$$\sum_{s=0}^{\frac{N}{2}}\beta_s\varepsilon^{-\left(\frac{1+2s}{2}\right)}.$$
		
		It follows from this that $\left|\int_{\tau(\text{\CYRSH}'(\varepsilon))}\omega\right|\to\infty$ as $\varepsilon\to 0.$
		
	\end{enumerate}
	
	By the Picard--Lefschetz formula, the analytic continuation of the function $\int_{\text{\CYRSH}(X)}\omega$ along the loop $X_{c(t)},$ where $c(t)\equiv\{\theta'+\varepsilon e^{it}, t\in[0,2\pi]\}$, maps it to its negative for even $N$ and to itself for odd $N$, which is confirmed by our computation.
	
	In the {\em general case}, there exists a diffeomorphism $\mathcal{G}$ of the space $\mathbb C^{N+1}$ that reduces the equation of $A$ to the form \eqref{eque}, while preserving $S$ and the function $z$ (and therefore, the function $p|_{A\cap S}\equiv f_{\tilde\lambda}|_{\hat S}$). However, $\mathcal{G}$ may not preserve the function $p\equiv f_{\tilde\lambda}$, meaning that it may deform the hyperplane $X$ while preserving $X\cap S$. Let us denote by $\tilde X$ the image $\mathcal{G}(X).$ The desired integral is equal to the integral of the form $\omega$ multiplied by the function $\frac{J\mathcal{G}}{g^{N+2}}$ (where $J\mathcal{G}$ is the Jacobian of the diffeomorphism $\mathcal{G}$, and the function $g$ is defined by the condition $\mathcal{G}:x_1\mapsto x_1g(x_1, z_2, \dots, z_N)$) over the tube around the same cycle $\text{\CYRSH}'$, but now this tube realizes an element of the homology group $H_{N+1}(D\setminus S, A\cup\tilde X\setminus S).$ Since the diffeomorphism $\mathcal{G}$ is a small perturbation of the identity, the values of its Jacobian and of the function $g$ near the point $x_{\theta'}$ lie in a small neighborhood of the point $1\in\mathbb C.$ On the other hand, the integral \eqref{eq} analytically depends on the parameters of the diffeomorphism $\mathcal{G}$ (which acts only on the equation of $X$). Consequently, in the general case for $a=2$, the integral \eqref{eq} is also non-zero, and the lemma is proved. $\Box$
	
	\begin{example}
		\textbf{Monodromy of a smooth affine quadric}
	\end{example}
	Consider the surface $Q\subset\mathbb{C}^{N+1}$ given by the equation $x^2+y^2+z_1^2+\dots+z_{N-1}^2=1$. It is smooth, and its closure $\overline Q$ is transversal to $S$. The monodromy of its hyperplane sections was studied by D.~Yang in~\cite{Dao}, but bypasses around simple asymptotes were not considered there. Let us take a line $L\subset(\mathbb{CP}^{N+1})^*$ defined by a pencil of hyperplanes $\{x=c(y-2), c\in\mathbb C\}$ in $\mathbb{C}^{N+1}$. The set $\Sigma(Q)\subset(\mathbb{CP}^{N+1})^*$ consists of a smooth quadric $Q^*$ (projectively dual to $\overline Q$) and a cone $\hat Q$ over its hyperplane section with the vertex at $S$. The set $Q^*\cap\hat Q$ is the $B_2$ stratum (see Sec.~\ref{Dnk}). The line $L$ transversally intersects $\Sigma$ at the points $c=\pm1/\sqrt{3}$ and $\pm i$. Let us connect these points to $X:c=0$ by segments. By the Zariski theorem, the set of corresponding simple loops $\{\gamma_1, \mu_1, \gamma_2, \mu_2\}$ (in the counterclockwise order starting from the segment $[0, 1/\sqrt{3}]$) generates the fundamental group $\pi_1((\mathbb{CP}^{N+1})^*\setminus\Sigma)$. The homology group $H_{N+1}(\mathbb{C}^{N+1}, Q\cup X)$ is isomorphic to $\mathbb Z^2$, and as generators one can take the classes of domains $\Omega_1, \Omega_2\subset\mathbb{R}^{N+1}$ (equipped with the induced orientation) vanishing along the paths $\gamma_1, \gamma_2$, into which the hyperplane $X$ divides $\operatorname{Re}Q$. The cycle $\Omega=\Omega_1+\Omega_2$ generates the group $H_{N+1}(\mathbb{C}^{N+1}, Q)\simeq\mathbb Z$. 
	
	\begin{PREDLOZHENIE}
		The relation 
		\begin{equation}\label{try1}
			\gamma_1=\mu_1\gamma_2\mu_1^{-1}
		\end{equation}
		holds in $\pi_1((\mathbb{CP}^{N+1})^*\setminus \Sigma)$.
	\end{PREDLOZHENIE}
	
	\textbf{Proof.} Consider the loop $\omega(t)\equiv\{x=c(y-e^{it}(1+\delta)), t\in[0,2\pi], \delta>0\}\subset(\mathbb{CP}^{N+1})^*\setminus \Sigma$. Let us transport the line $L$ along the path $\{x=c(y-s), s\in[1+\delta, 2]\}$ to the line $\omega(0)$, make a bypass along $\omega$, and return back. Under this operation, $\gamma_1$ transforms into $\mu_1\gamma_2\mu_1^{-1},$ which implies the required result. $\Box$
	
	In~\cite{Dao}, the set $\{\gamma_1, \gamma_2, \kappa\}$ was considered as a set of generators of the group $\pi_1((\mathbb{CP}^{N+1})^*\setminus \Sigma)$. The loop $\kappa$ is determined by the family of hyperplanes $\{x\cos t+y\sin t=0, t\in[0,\pi]\}$. The operator $\kappa$ obviously permutes $\Omega_1$ and $\Omega_2$.
	
	\begin{PREDLOZHENIE}
		The relation 
		\begin{equation}\label{try}
			\kappa=\gamma_1\mu_1
		\end{equation}
		holds in $\pi_1((\mathbb{CP}^{N+1})^*\setminus \Sigma)$.
	\end{PREDLOZHENIE}
	
	\textbf{Proof.} Transporting the line $L$ to the line $L':\{x=cy, c\in\mathbb{C}\}$ along the path consisting of three parts: 1)~$\{x=c(y-s), s\in[1+\delta, 2]\}$, 2)~$\{x=c(y-(1+\delta e^{i\varphi})), \varphi\in[0,\pi]\}$, 3)~$\{x=c(y-s), s\in[0, 1-\delta]\}$, and lifting the homotopy to the homotopy of the loop $\gamma_1\mu_1$, we obtain the required result. $\Box$
	
	Let us denote by $\text{\CYRSH}'$ the class of the group $H_{N}(S, \overline Q\cup \overline X)$ represented by the cycle vanishing along the path $\mu_1$.
	
	\begin{figure}[H]
		\center{\includegraphics[width=1\linewidth]{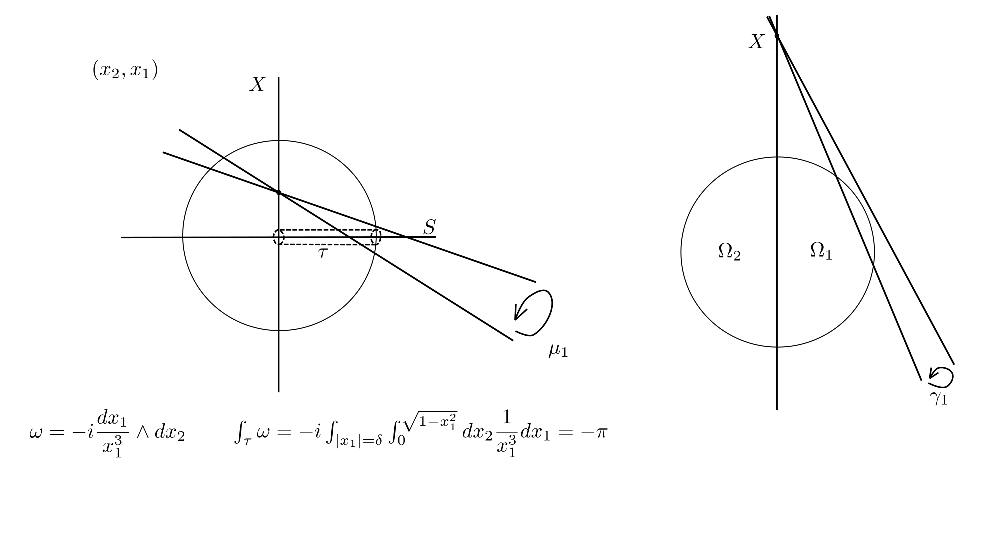}}
		\caption{Example: $N=1$}
		\label{Q}	
	\end{figure}  
	
	\begin{lemm} Let $N+1$ be even. Then:
		$$\mu_1(\Omega_1)=\Omega_1-\Omega, \quad \mu_1(\Omega_2)=\Omega_2+\Omega.$$
		
		Let $N+1$ be odd. Then:
		$$\mu_1(\Omega_1)=\Omega_2, \quad \mu_1(\Omega_2)=\Omega_1.$$
	\end{lemm}
	
	\textbf{Proof.} Let $N+1$ be even and $N>1$. By the generalized Picard--Lefschetz formula, we have $\mu_1(\Omega_1)=\Omega_1+x$ and $\mu_1(\Omega_2)=\Omega_2-x,$ where:
	$$x=(-1)^{\frac{N(N+1)}{2}} l(\Omega_1, \text{\CYRSH}')\tau(\text{\CYRSH}'),$$
	and the coefficient $l(\Omega_1, \text{\CYRSH}')$ is determined via localization (see Sec.~\ref{PLF}) and reduction (see~\eqref{1}). Let us act by the loop \eqref{try1} on $\Omega_1.$ On the one hand, we obtain $-\Omega_1.$ On the other hand, we have $\Omega_1 \mapsto \Omega_1+x \mapsto \Omega_1+2\Omega_2+x \mapsto \Omega_1+2\Omega_2+x-x+2x=\Omega_1+2\Omega_2+2x.$ From this, it follows that $x=-\Omega_1-\Omega_2=-\Omega.$ The action of the loop \eqref{try} yields the same answer even more simply.
	
	Now let $N=1$. The cycle of the ``group $H_{-1}(\overline{Q}\cap S\cap\overline{X})$'' vanishing along $\mu_1$ is equal to zero (the value $l(\Omega_1, \text{\CYRSH}')$ is not defined) and $\operatorname{Var}_{\mu_1}(\Delta)=0,$ where $\Delta$ is a generator of the group $\tilde H_0(Q\cap X)$. Let us denote by $\Omega_i'$ the image of the cycle $\Omega_i$ under the isomorphism $H_2(\mathbb{C}^2, Q\cup X)\rightarrow H_1(Q, Q\cap X).$ The variation of the cycle $\Omega_1'$ is computed directly: it is equal to a tube around one of the points $\overline{Q}\cap S$ and is homologous to the generator $-\Omega'$ of the group $H_1(Q).$ Hence, $\operatorname{Var}_{\mu_1}(\Omega_1)=-\Omega$ as well.
	
	If $N+1$ is odd, the answer follows immediately from \eqref{try}, since in this case the action of $\gamma_1$ and $\gamma_2$ is trivial (the monodromy group is isomorphic to $\mathbb{Z}_2$). $\Box$
	
	It follows from the proof that for odd $N$, the tube $\tau(\text{\CYRSH}')$ is homologous to the cycle $\pm\Omega$ in the group $H_{N+1}(\mathbb C^{N+1}, Q\cup X)$; hence, the integral of the volume form over it does not depend on $X$ (which can also be easily verified by direct computation, cf.~\eqref{eq}). In the case of even $N$, the cycle $\tau(\text{\CYRSH}')$ is homologous to $\pm(\Omega_1-\Omega_2)$, and the integral over it is zero. However, this value is unstable under a small perturbation of $X$ (this is analogous to the fact that the integral \eqref{eq} vanishes only for $\xi=0$).
	
	\subsection{The main theorem}\label{osn}
	Let $F$ be an arbitrary deformation of the form \eqref{obsch} of an isolated real boundary singularity $f$, and let $\Theta$ be a sufficiently small neighborhood of zero in the parameter space.
	
	\begin{definition}
		The ramification of the analytic continuation of the function $V_i$ (see \eqref{2}) is called \textbf{locally finite-sheeted} if the integral of the form $\omega$ over all possible variations of the cycle $\Omega_i$ under the action of the group $\pi_1(B\setminus\Sigma)$ takes only a finite number of values. 
	\end{definition} 
	
	\begin{definition}\label{quasi}
		A component $\Gamma$ of the complement in $\Theta\cap\mathbb{R}^l$ to $\Sigma_{\mathbb{R}}(F)$ is called a (strong) \textbf{quasilacuna} if, for some (and therefore, for any) $\tilde{\lambda}\in\Gamma$, the ramification of the analytic continuation of each function $V_{i}$ is locally finite-sheeted.  
	\end{definition}
	
	\begin{theorem}\label{o} 
		a) Let $N$ be odd and the germ $f$ satisfy Proposition~\ref{parab}. Then each function $V_i$ is infinitely ramified in $B\subset(\mathbb{CP}^{N+1})^*$. 
		b) Let $N$ be even. If the class of the cycle $\Pi_i$ is non-zero in the homology group $H_{N-1}(A\cap X, \partial D)$, then the function $V_i$ is infinitely ramified in $B\subset(\mathbb{CP}^{N+1})^*$. 
	\end{theorem}
	
	\textbf{Proof.}
	Consider first the case of odd $N$.
	
	Let us choose a distinguished system of paths on the line $L_{\tilde{\lambda}}$ with the base point $X=X_{0}$ and define a basis $\{\Delta_{s}, \tau(\Delta_{t}')\}$ of the homology group $H_{N-1}(A\cap X\setminus S)$ (see Proposition~\ref{baza}).
	
	If the image of the cycle $\Pi_i$ under the epimorphism $H_{N-1}(A\cap X\setminus S, \partial D)\rightarrow H_{N-1}(A\cap X, \partial D)$ (see~\eqref{cd}) is zero (which is always satisfied in the case where $df(0)\neq0$, see also Example~\ref{newex}), then there exists an element $\tau(\Delta_t')$ for $\Pi_i$ such that $\operatorname{Var}_{\gamma_t}\Pi_i\neq 0$. Then the variation operator $\operatorname{Var}_{\gamma_t}^r$ corresponding to an $r$-fold bypass along the loop $\gamma_t$ acts on $\Pi_i$ according to the formula:
	$$\operatorname{Var}_{\gamma_t}^r(\Pi_i)=(-1)^{\frac{N(N+1)}{2}} r[\Pi_i, \tau(\Delta_t')]\tau(\Delta_t')$$ 
	(see Lemma~\ref{Picard} and Remark~\ref{rem}). By consecutively applying the loop $\gamma_t$ to $\Pi_i$, we obtain the infinite ramification of the volume function from \eqref{cd0.5}, Proposition~\ref{Izol}, and Lemma~\ref{Vol1}. 
	
	If $\Pi_i\neq0\in H_{N-1}(A\cap X, \partial D)$ (which is possible only in the case where $df(0)=0$), then it may happen that $[\Pi_i, \tau(\Delta_{t}')]=0$ holds for all $t$. Let $\{\Delta_{j}\}_{j=1}^{q}$ be the set of all $\Delta_s$ (after a renumbering) for which $[\Pi_i, \Delta_s]\neq0$.
	
	Let us denote by $\mho$ the set of those $\Delta_s$ for each of which there exists at least one element $\tau(\Delta_t')$ along with the condition $[\Delta_s, \tau(\Delta_t')]\neq0$. By virtue of Lemma~\ref{Lyas}, the set $\mho$ is non-empty.
	
	If $\{\Delta_{j}\}_{j=1}^{q}\cap\mho\neq\varnothing$, then let us consider any cycle $\Delta_s\in\{\Delta_{j}\}_{j=1}^{q}\cap\mho$. First, we act on $\Pi_i$ by the loop $\gamma_{s}$, and then we sequentially act on the result by the loop $\gamma_t$ (the very one from the definition of $\mho$). Again, by virtue of formulas \eqref{PL}, \eqref{cd0.5}, and Lemma~\ref{Vol1}, the volume function is infinitely ramified.
	
	On the other hand, if $\{\Delta_{j}\}_{j=1}^{q}\cap\mho=\varnothing$, then we consider the shortest path $\Delta_{r}\Delta_{s_{1}}\Delta_{s_{2}}\dots\Delta_{s_{m}}\Delta_s$ in the support of the Dynkin diagram (p.~\pageref{otmb}) of the basis $\{\Delta_s\}$ between the set $\{\Delta_{j}\}_{j=1}^{q}$ and the set $\mho$, originating at the vertex $\Delta_{r}$ for some $r\in\{1, \dots, q\}$ and terminating at the vertex $\Delta_s\in\mho.$ We act on $\Pi_i$ by the composition of operators $\gamma_s\circ\gamma_{s_m}\circ\dots\circ\gamma_{s_1}\circ\gamma_r$. As a result of the action of the $(k+1)$-th operator, the cycle $\Delta_{s_{k}}$ is added to the intermediate result with a non-zero coefficient, since in this chain each cycle intersects only with the adjacent ones. By consecutively applying the loop $\gamma_t$ to the final result, we complete the proof. 
	
	In the case of even $N$, see the proof of Lemma~\ref{VV} in Sec.~\ref{kasal}. $\Box$
	
	\begin{Relation}\label{Rem}
		The proof of Theorem~\ref{o} is repeated without changes for any function $\sum_{i=1}^k n_iV_{i}$, $n_i\in\mathbb{R}$ and the corresponding cycle $\sum_{i=1}^k n_i\Pi_i$, provided that it is non-trivial: a)~in the homology group $H_{N-1}(A\cap X\setminus S, \partial D; \mathbb R)$, b)~in the homology group $H_{N-1}(A\cap X, \partial D; \mathbb R)$. $\Box$
	\end{Relation}
	
	For even $N$, in the case of a simple singularity $f|_{\hat S}$, the condition $\Pi_i\neq0\in H_{N-1}(A\cap X, \partial D)$ is indispensable.
	
	\begin{example}\label{b2} 
		Consider a versal deformation of a singularity $f$ of type $B_{2}^+$ in the form $x^2+y^2+\sum_{j=1}^{N-2}z_{j}^2+\lambda_1x+\lambda_0$, where the boundary $\hat{S}$ is defined by the condition $x=0$ (see Example~\ref{po}). Let $\lambda_0>0$ and $\lambda_0<\lambda_1^2/4$. Then the unique cycle $\Pi$ is a vanishing cycle realized by a sphere that does not intersect $\hat{S}$ (see Fig.~\ref{F0}). For even $N$, the class of $\Pi$ is zero in $H_{N-1}(A\cap X, \partial D)$. The vanishing cycle of the group $H_{N-2}(A\cap X\cap S)$ yields the infinite-sheetedness of function \eqref{2} only for odd $N$. Therefore, for even $N$, all 4 connected components of the complement to $\Sigma_{\mathbb{R}}$ (see Fig.~\ref{B2'}) are quasilacunas: if $\lambda_0<0$, the set $\operatorname{Re}A\cap X$ is connected and intersects the boundary, and if $\lambda_0>\lambda_1^2/4$, then $\operatorname{Re}A\cap X=\varnothing$.
	\end{example}
	
	\begin{example}\label{A2k}
		Consider a purely boundary singularity $f=x-(y-\varepsilon x^2)^{2k}-\sum_{j=1}^{N-2} z_{j}^2$ of type $A_{2k-1}$, where the boundary $\hat{S}$ is given by the condition $x=0$. Under the perturbation $\tilde{f}=f-\varepsilon'$ ($\varepsilon'>0$), the unique cycle $\Pi$ lies in the domain $x>0$ (see Fig.~\ref{F0}). Let $f_{\tilde{\lambda}}$ be a real morsification of $f$ close to $\tilde{f}$. Then, for even $N$, any bypasses around the critical values of $f_{\tilde{\lambda}}|_{\hat{S}}$ yield only a finite ramification of function \eqref{2}, since the singularity $f|_{\hat{S}}$ is simple (see Lemma~\ref{simpl}). However, for odd $N$, sequentially bypasses around the maximal critical value of $f_{\tilde{\lambda}}|_{\hat{S}}$ realize an infinite ramification (see the proof of Proposition~\ref{Pred7}).
	\end{example}
	
	\begin{figure}[H]
		\center{\includegraphics[width=1\linewidth]{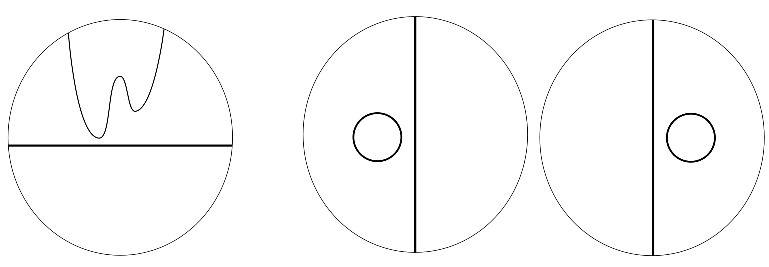}}
		\caption{Examples of generic fibers of $A_{3}$ (left) and $B_{2}^+$ (right)}
		\label{F0}	
	\end{figure}
	
	\begin{lemm}\label{simpl}
		Let $N$ be even, the singularity $f|_{\hat{S}}$ be simple, and $\Pi_i=0\in H_{N-1}(A\cap X, \partial D)$. Then function $\eqref{2}$ has a finite ramification in $B$.
	\end{lemm}
	
	\textbf{Proof.} Suppose that function \eqref{2} is infinitely ramified; then the set of variations of the cycle $\Omega_i$ (and therefore, of the cycle $\Pi_i$) under the action of the group $\pi_1(B_{\delta}\setminus\Sigma)$ is infinite, and integral \eqref{2} over the cycles of this set takes an infinite number of values. By assumption, the group $\pi_1(B_{\delta}\setminus\Sigma^{f})$ ``does not act'' on the cycle $\Pi_i$, and the monodromy group of a simple singularity $f|_{\hat{S}}$ of an odd number of variables is finite (see~\cite{Norm}). Then there must exist an element $\gamma\in\pi_1(B_{\delta}\setminus\Sigma^{{f}|_{\hat{S}}})$ that acts trivially on the group $H_{N-2}(A\cap X\cap S)$, and for which $\operatorname{Var}_{\gamma}\Pi_i\neq0$ holds. However, in this case, for any vanishing cycle $\Delta'\in H_{N-2}(A\cap X\cap S)$, one has:
	$$[\Pi_i, \tau(\Delta')]=\langle\Pi_i, \overline{\Delta'}\rangle=\langle\gamma_*\Pi_i, \gamma_*\overline{\Delta'}\rangle=\langle\Pi_i+i_*\operatorname{Var}_{\gamma}\Pi_i, \overline{\Delta'}\rangle=\langle\Pi_i, \overline{\Delta'}\rangle+\langle i_*\operatorname{Var}_{\gamma}\Pi_i, \overline{\Delta'}\rangle,$$
	since $\operatorname{Var}_{\gamma}\overline{\Delta'}=0$. It follows from this that $i_*\operatorname{Var}_{\gamma}\Pi_i=0$. This implies that $\tau^{-1}\operatorname{Var}_{\gamma}\Pi_i\in\operatorname{Ker}\langle\cdot,\cdot\rangle_{A\cap X\cap S}$. Since the intersection form of a simple singularity of an odd number of variables is elliptic (see~\cite{Norm}), we arrive at a contradiction. $\Box$
	
	\begin{theorem}\label{criteria}
		(Criteria for local finite-sheetedness)
		\begin{enumerate}	
			\item Let $N$ be odd and the germ $f$ satisfy Proposition~\ref{parab}. Then $\Gamma$ is a quasilacuna if and only if $\{\Pi_{i}\}=\varnothing$. 
			
			\item Let $N$ be even. If $\Gamma$ is a quasilacuna, then the classes of all cycles $\Pi_{i}$ in the homology group $H_{N-1}(A\cap X, \partial D)$ are zero. 
			a)~If the singularity $f|_{\hat{S}}$ is simple, then the converse is also true. 
			b)~If both inertia indices of the quadratic form of the singularity $f|_{\hat{S}}$ are non-zero and $\{\Pi_{i}\}\neq\varnothing$, then the converse is false.
		\end{enumerate}
	\end{theorem}
	
	\textbf{Proof.} Everything except 2b) follows from Theorem~\ref{o} and Lemma~\ref{simpl}, while the latter follows from Theorem~\ref{Th} proved in Sec.~\ref{kasal} and from Lemma~\ref{Vol1}. $\Box$ 
	
	Consider separately the case of a purely boundary singularity.
	\begin{PREDLOZHENIE}\label{Pred7}
		Let $N$ be odd, $df(0)\neq 0$, and Proposition~\ref{parab} hold. Then the set of non-quasilacunas is non-empty (and consists of a single element) if and only if the function $f|_{\hat S}$ has an extremum at the origin.
	\end{PREDLOZHENIE} 
	
	\textbf{Proof.} Any germ of such a singularity is $\mathcal{R}_{\hat S}$-equivalent to a germ of the form $f=x_1+\varphi(x_2, \dots, x_N)$, where $d\varphi(0)=0$. Then the set $\operatorname{Re}A\cap X$ is connected, and if the function $\varphi$ has no extremum, then $\{\Pi_i\}=\varnothing$ for all $\Gamma$.
	
	Let the function $\varphi$ have a maximum at the origin. Then the cycle $\Pi=D\cap\operatorname{Re}\{f=\varepsilon'>0\}$ is well-defined, and by Theorem~\ref{o}, the component $\Gamma$ containing the function $f-\varepsilon'$ is not a quasilacuna.
	
	Let us explicitly specify the obstructing element \(\tau(\Delta')\) (such that \([\Pi, \tau(\Delta')]\neq0\)). Choose a generic point $\lambda'\in\partial\Gamma$ close to zero. The function $f_{\lambda'}|_{\hat{S}}$ has a non-degenerate real critical point of maximum with value zero. Let us denote by $\Delta'\in H_{N-2}(A\cap X_{\varepsilon'}\cap S)$ the cycle vanishing at this point. Then, for the cycle $\Pi$ given by the set $\operatorname{Re}\{f_{\lambda'}=\varepsilon'\}\cap D$, one has $[\Pi, \tau(\Delta')]=\pm1$.
	
	Indeed, near a simple tangency with the boundary, the cycle $\Pi$ in local Morse coordinates is given by the equation $x_1=z_2^2+\dots+z_{N}^2+\varepsilon',$ and the cycle $\Delta'$ is realized by the sphere $\{(\operatorname{Im}z_{2})^2+\dots+(\operatorname{Im}z_{N})^2=\varepsilon', x_1=0\}$. Let $\pi:A\cap X_{\varepsilon'}\rightarrow\hat{S}$ be the natural projection. Then the half-cycle $\overline{\Delta'}$ is given by the set $\pi^{-1}(\{(\operatorname{Im}z_{2})^2+\dots+(\operatorname{Im}z_{N})^2\leqslant\varepsilon'\}),$ and we obtain $[\Pi, \tau(\Delta')]=\langle\Pi, \overline{\Delta'}\rangle=\langle\pi(\Pi), \pi(\overline{\Delta'})\rangle=\langle\operatorname{Re}\hat{S}, \{(\operatorname{Im}z_{2})^2+\dots+(\operatorname{Im}z_{N})^2\leqslant\varepsilon'\}\rangle=\pm1$. $\Box$ 
	
	In the case of even $N$ and a simple purely boundary singularity, all components are quasilacunas by Theorem~\ref{criteria}.
	
	\begin{example}
		Consider the parabolic singularity $x_1+y^4+z^4+Q$, where $Q$ is a non-degenerate quadratic form of an odd number of additional variables with a zero negative inertia index. It can be shown that the component $\Gamma$ containing the function $x_1+y^4+z^4+Q+\varepsilon$, $\varepsilon>0$, is not a quasilacuna. For a small Morse perturbation of this function within $\Gamma$, the cycle $\Pi$ is well-defined and there exists an element $\tau(\Delta')$ such that $[\Pi, \tau(\Delta')]\neq0$, see the proof of Proposition~\ref{Pred7}. This element is not obstructing in the same sense, since its corresponding monodromy operator is of order $2$ due to the evenness of $N$. However, the analytic continuation of the function $\int_{\tau(\text{\CYRSH}'(X))}\omega$ for the corresponding element $\tau(\text{\CYRSH}')$ is infinitely-valued at the point $X_0$ (in Sec.~\ref{kasal} this is proved for the Euclidean volume form on $S$ and the cycle $\text{\CYRSH}'$), which implies the required result (see the proof of Theorem~\ref{Th}). 
	\end{example}
	
	\begin{PREDLOZHENIE}\label{K-K}
		The numbers of quasilacunas of the miniversal deformations of the form \eqref{obsch} of $\mathcal{K}_{\hat S}$-equivalent singularities coincide.
	\end{PREDLOZHENIE}
	
	\textbf{Proof.} Since the miniversal deformations of two $\mathcal{K}_{\hat S}$-equivalent germs $f$ and $f'$ are $\mathcal{K}_{\hat S}$-equivalent (see~\cite{AVG}, Sec.~8.6), i.e., $f'_{\lambda}(x)=\psi_{\lambda}(x)f_{\chi(\lambda)}(\mathcal{G}_{\lambda}(x))$, then together with the diffeomorphism $\chi$ of the parameter space $\mathbb{R}^{\mu}$ inducing one such deformation from another, there are defined boundary-preserving diffeomorphisms $\mathcal{G}_{\lambda}$ of the space $(\mathbb{C}^N, \mathbb{R}^N)$ depending smoothly on $\lambda\in\mathbb{R}^{\mu}$, which determine diffeomorphisms of the zero-level sets of the functions $f'_{\tilde{\lambda}}$ and $f_{\chi(\tilde{\lambda})}$, and consequently of the triples $(A\cap X\cup\partial D, (A\cap X\setminus S)\cup\partial D, \partial D)$ (both real and complex). Thus, the homology groups for $f'_{\tilde{\lambda}}$ and $f_{\chi(\tilde{\lambda})}$ of any of the types participating in diagram \eqref{cd} are isomorphic, from which the statement of the proposition for odd $N$ follows immediately by Theorem~\ref{criteria}. The diffeomorphism of the space $(\mathbb{C}^{N+1}, \mathbb{R}^{N+1})$ given by the formula $\mathcal{J}:(c,x)\mapsto(c\cdot(\psi_{\tilde{\lambda}}(x))^{-1}, \mathcal{G}_{\tilde{\lambda}}(x))$ defines a diffeomorphism of the surfaces $(A, \operatorname{Re}A)$ --- the graphs of the functions $f_{\chi(\tilde{\lambda})}$ and $f'_{\tilde{\lambda}}$. Let us fix a distinguished system of paths on the complex line with the coordinate $c$ for the function $f_{\chi(\tilde{\lambda})}(x)$ and define its basis $\{\text{\CYRSH}_j\}$ from Proposition~\ref{Izol}. Consider the homotopy $h(t,x)=t(\pm1-\psi_{\tilde{\lambda}}(x))+\psi_{\tilde{\lambda}}(x)$, $t\in[0,1]$, connecting the germ of the function $\psi_{\tilde{\lambda}}$ at the origin with the germ of the constant map from $\mathbb{C}^N$ to $\pm1$ (depending on the sign of the number $\psi_{\tilde{\lambda}}(0)$). If necessary, we perturb this homotopy so that at each moment $t$ the function $h_t(x)f_{\chi(\tilde\lambda)}(\mathcal{G}_{\tilde\lambda}(x))$ and its restriction to the boundary are strictly Morse. Under the action of this homotopy, the fixed system of paths transforms into a system of paths for the function $f'_{\tilde\lambda}(x)$ determining the basis $\{\mathcal{J}_*\text{\CYRSH}_j\}$. Then the sets of variations of the cycles $\Omega_i$ and $\mathcal{J}_*\Omega_i$ are isomorphic, and $\int_{\mathcal{J}_*\text{\CYRSH}_j}\omega=\int_{\text{\CYRSH}_j}\frac{\operatorname{det}\mathcal{J}'}{h^{N+2}}\omega$, where the Jacobian $\operatorname{det}\mathcal{J}'$ and the function $h$ (defined by the condition $\mathcal{G}_{\tilde\lambda}:x_1\mapsto x_1\cdot h$) take values near the origin in small neighborhoods of non-zero constants. Finally, the multiplication by $-1$ (in the case where $\psi_{\tilde\lambda}(0)<0$) is analyzed trivially. The statement of the proposition is proved for any $N$. $\Box$
	
	For examples, see Remark~\ref{ZaM}.
	
	\subsection{The generalized local even Petrovsky cycle}\label{Petrov}
	
	\begin{definition}\label{obobsch}
		Suppose that $\operatorname{Re}A_{\tilde\lambda}\cap X\cap S=\varnothing$ holds. The \textbf{generalized local even Petrovsky class} is the class of the canonically oriented cycle $\Pi=\operatorname{Re}A_{\tilde\lambda}\cap X$ in the homology group $H_{N-1}(A_{\tilde\lambda}\cap X\setminus S, \partial D)$.
	\end{definition}
	The canonical orientation means that at any point of $\Pi$, the set of vectors tangent to $X$ ($\operatorname{grad}f_{\tilde{\lambda}}$, the orienting frame of the cycle $\Pi$) determines an orientation that coincides with the fixed positive orientation of $X\simeq\mathbb{R}^N$. This cycle is related to the problem of the ramification of the total volume cut off by a hyperplane on one side of an unbounded domain in a Euclidean space.
	
	Let us denote by $\Omega$ the localized integration contour in $D$ corresponding to $\Pi$, and by $V(X)$ the function $\int_{\Omega(X)}\omega$.
	
	\begin{definition}\label{slow}
		A component $\Gamma$ of the complement in $\Theta\cap\mathbb{R}^l$ to $\Sigma_{\mathbb{R}}(F)$ is called a \textbf{weak} quasilacuna if, for any $\tilde{\lambda}\in\Gamma$, the ramification of the analytic continuation in $B$ of the function $V$ is finite-sheeted (in the case where the cycle $\Pi$ is not defined or $\Pi=\varnothing$, we set $V\equiv0$). A weak quasilacuna is called \textbf{strict} if the cycle $\Pi$ is defined and non-empty.
	\end{definition}
	
	From Corollary~\ref{Rem}, we obtain:
	\begin{PREDLOZHENIE}
		Suppose that Proposition~\ref{parab} holds. Then, for odd $N$, the set of strict weak quasilacunas is empty for any deformation. $\Box$
	\end{PREDLOZHENIE} 
	
	\begin{example}\label{newex} 
		Consider the perturbation $f_{\tilde\lambda}-z^2$ of the singularity $x^4+y^4-z^2$, where the function $f_{\tilde\lambda}(x,y)$ is defined in Example~\ref{Oo}. The set $\Pi=\operatorname{Re}\{f_{\tilde\lambda}-z^2=0\}$ lies in the domain $x>0$ and is homeomorphic to a two-dimensional torus. The component of the complement to $\Sigma_{\mathbb R}^f$ (the tabular versal deformation) containing $\tilde\lambda$ is an even formal lacuna (i.e., $\Pi=0\in H_{N-1}(A\cap X, \partial D)$), see~\cite{VVA}, Ch.~VI,~\S2. After introducing the boundary $\hat S:x=0$, the corresponding component of the complement to $\Sigma_{\mathbb R}=\Sigma^f_{\mathbb R}\cup\Sigma^{f|_{\hat S}}_{\mathbb R}$ is not a weak quasilacuna by Theorem~\ref{o}. 
	\end{example}
	
	\begin{figure}[H]
		\center{\includegraphics[width=0.5\linewidth]{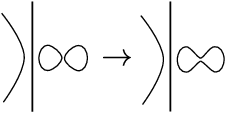}}
		\caption{Example for Proposition~\ref{DB} --- bifurcation of levels for a perturbation of a type $B_5$ germ (see Sec.~\ref{quasisimpl}) for $N=2$. In both cases, the intersection number of $\Pi$ with the cycle vanishing at the saddle point is equal to $\pm2$}
		\label{F34}	
	\end{figure}
	
	The conditions on the signature given below are analogous to the {\em Davydova--Borovikov condition} from the theory of hyperbolic PDEs (see~\cite{VVA}, Ch.~VI,~\S 1.4).
	
	\begin{PREDLOZHENIE}\label{DB}
		Let $\lambda'\in\Sigma_{\mathbb R}^{f}$ be a generic point of the discriminant close to zero, $\theta$ be a Morse critical point of the function $f_{\lambda'}$ with value zero, and $D\cap\operatorname{Re}\{f_{\lambda'}=0\}\cap S=\varnothing$ hold. 
		a)~Let $N$ be even. Then, if the Morse indices of the point $\theta$ are odd, both components containing the functions $f_{\lambda'}+\varepsilon$ and $f_{\lambda'}-\varepsilon$ are not weak quasilacunas. 
		b)~Let $N$ be odd. Then, if the positive Morse index of the point $\theta$ is even, the cycle $\Pi$ defined for the function $f_{\lambda'}+\varepsilon$ is non-zero in the homology group $H_{N-1}(A\cap X, \partial D)$; if it is odd, the same holds for the function $f_{\lambda'}-\varepsilon$.
	\end{PREDLOZHENIE}
	
	\textbf{Proof} see in~\cite{VVA}. $\Box$
	
	\begin{Relation}
		Suppose that for $\tilde\lambda\in\Theta\setminus\Sigma_{\mathbb R}$, the condition $D\cap\operatorname{Re}\{f_{\tilde\lambda}=0\}\cap S=\varnothing$ holds and there exists a path in $\Theta\setminus\Sigma^f_{\mathbb R}$ connecting the functions $f_{\tilde\lambda}$ and $f_{\lambda'}+\varepsilon$ (or $f_{\lambda'}-\varepsilon$). Then statements a) and b) remain valid if the corresponding function is replaced by $f_{\tilde\lambda}$. $\Box$
	\end{Relation}
	
	By the Poincaré--Lefschetz duality, the class of the cycle $\Pi$ in the homology group $H_{N-1}(A\cap X\setminus S, \partial D)$ is determined by the vector of its intersection numbers with the vanishing cycles of the group $H_{N-1}(A\cap X, S)$. Further, we generalize Vassiliev's formulas obtained in Theorem 3 of~\cite{VVA}, Ch.~VI,~\S 1. Before that, an auxiliary problem is solved for the cycle $\breve\Pi=\operatorname{Re}\{f_{\tilde{\lambda}}(x_1^2, x_2, \dots, x_N)=0\}\cap D$ --- its intersection numbers with the antisymmetric basic cycles are computed (see Sec.~\ref{Arncyc}).
	
	Let us take a generic real perturbation $f_{\tilde{\lambda}}$ and its corresponding function $\breve{f}_{\tilde{\lambda}}\equiv f_{\tilde{\lambda}}(x_1^2, x_2, \dots, x_N)$, and consider the fibrations $p$ and $p\circ\sigma$ (let us define the surfaces $A_{\tilde{\lambda}}$ and $\breve{A}_{\tilde{\lambda}}$, etc., see Fig.~\ref{F167} on p.~\pageref{F167}). We choose a distinguished system of paths on $L_{\tilde{\lambda}}$ (see Fig.~\ref{systpath}) in a canonical way (in what follows, we follow~\cite{VVA}, Ch.~VI,~\S 1.6):
	\begin{enumerate}
		\item The paths connecting the real critical values of the function $\breve f_{\tilde\lambda}$ to zero pass through the domain $\operatorname{Im}\geqslant0$ such that the imaginary critical values do not lie between these paths and the real axis; the initial segment of each such path runs along the real axis towards smaller values.
		
		\item The paths connecting the complex conjugate critical values to zero are also complex conjugate and do not intersect the $\operatorname{Re}$ axis anywhere except at the origin.
		
		\item The paths enter the point $0$ with different tangents.
	\end{enumerate}
	
	\begin{figure}[H]
		\center{\includegraphics[width=0.8\linewidth]{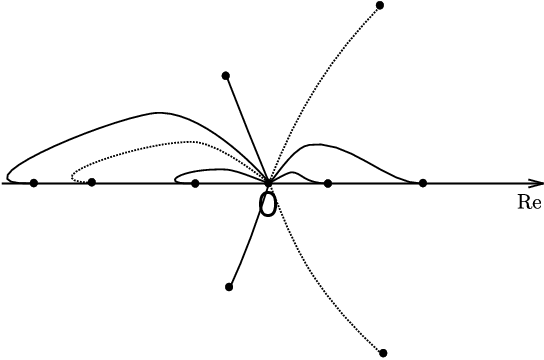}}
		\caption{Distinguished system of paths for a real boundary singularity. The dashed lines denote the paths corresponding to the function $f_{\tilde{\lambda}}|_{\hat S}$}
		\label{systpath}	
	\end{figure}

We will determine by this system of paths: 1)~a basis $\{\overline{\alpha}_{\nu_i}\}\subset H_{N-1}(A\cap X, A\cap X\cap S)$ consisting of vanishing cycles and half-cycles; 2)~an antisymmetric basis $\{\beta_{\nu_i}\}\subset H^-$ of long and short cycles (see Sec.~\ref{Arncyc}); 3)~a distinguished basis $\{\alpha_{\nu_i}\}\subset H_{N-1}(A\cap X\setminus S)$ of vanishing cycles and tubes (see Proposition~\ref{baza}). The subscript attached to each cycle denotes its critical value (here we depart from the notation of Sec.~\ref{PLF} and switch to one more convenient for the present problem).

a) The vanishing cycles of the group $H_{N-1}(A\cap X)$ are oriented in the following canonical way:
\begin{enumerate}
	\item Let a cycle $\alpha_\nu\equiv\alpha_\nu(0)$ vanish at a critical point of the function $f_{\tilde\lambda}$ with value $\nu\in\mathbb R$. By the Morse lemma, in local coordinates $c,z$ near the critical point, the equation of the surface $A$ is of the form:
	\begin{equation}\label{vid}
		c=z_1^2+\dots+z_r^2-z_{r+1}^2-\dots-z_N^2+\nu.	
	\end{equation} 
	Then, for a small $\varepsilon>0$, in a generic hyperplane section $A\cap X_{\nu-\varepsilon}$, the cycle $\alpha_\nu(\nu-\varepsilon)$ is realized by a sphere of radius $\sqrt{\varepsilon}$ in the $N$-plane spanned in $\mathbb{C}^N\simeq\mathbb{R}^{2N}$ by the set of vectors:
	\begin{equation}\label{posl}
		i\partial/\partial z_1, \dots, i\partial/ \partial z_r, \partial/ \partial z_{r+1}, \dots, \partial/ \partial z_N.
	\end{equation} 
	
	We orient this plane as follows: the sequence \eqref{posl} defines a positive orientation if and only if the sequence $\partial/\partial z_1, \dots, \partial/ \partial z_N$ positively orients $\mathbb{R}^N=\operatorname{Re}\mathbb{C}^N.$ Finally, we define the orientation of the vanishing sphere by the condition that at any of its points, the set of tangent vectors ($\operatorname{grad}f_{\tilde\lambda}$, the orienting frame) determines a positive (just defined) orientation of this plane. 
	
	{\em Remark.} By $\operatorname{grad}f_{\tilde\lambda}$ we mean $\operatorname{grad}\operatorname{Re}f_{\tilde\lambda}$.
	
	\begin{PREDLOZHENIE}
		This orientation of the cycle $\alpha_{\nu}$ (defined by transporting the cycle $\alpha_\nu(\nu-\varepsilon)$ over the distinguished path) does not depend on the choice of local coordinates in which the equation of the surface $A$ is of the form \eqref{vid}. $\Box$ 
	\end{PREDLOZHENIE}
	
	\item We choose the orientations of the cycles vanishing over complex conjugate paths such that the involution of complex conjugation acting on the homology maps them into each other preserving the orientations.
\end{enumerate}

b) Near a critical point of the function $f_{\tilde{\lambda}}|_{\hat S}$ with value $\nu\in\mathbb R$, equation \eqref{vid} takes the form $c=\pm x_1+z_2^2+\dots+z_r^2-z_{r+1}^2-\dots-z_N^2+\nu$.
We orient the half-cycle $\overline{\alpha}_{\nu}(\nu-\varepsilon)$ near the boundary point by the set (the outward normal $\pm\partial/\partial x_1$, the orienting frame of the cycle $\partial\overline{\alpha}_{\nu}$) provided that the vanishing cycle $\partial\overline{\alpha}_{\nu}\in H_{N-2}(A\cap X\cap S)$ is canonically oriented.

c) Let us introduce canonical orientations of the vanishing cycles of the group $H_{N-1}(\breve A\cap X)$. The pairs of cycles $\overline{\alpha}_{\nu}^{\operatorname{I}}$ and $(-\overline{\alpha}_{\nu}^{\operatorname{II}})$ vanishing with the same critical value $\nu\in\mathbb R$ form in sum the antisymmetric long cycles $\beta_{\nu}=\overline{\alpha}_{\nu}^{\operatorname{I}}-\overline{\alpha}_{\nu}^{\operatorname{II}}$, while $\sigma_*(\overline{\alpha}_{\nu}^{\operatorname{I}})=\sigma_*(\overline{\alpha}_{\nu}^{\operatorname{II}})=\overline{\alpha}_{\nu}$ holds. Indeed, the transformation $x_1\mapsto-x_1$ maps two real Morse coordinate systems near the critical points into each other with a reversal of orientation, but identifies the vectors $\operatorname{grad}\breve{f}_{\tilde\lambda}$. The component $\overline{\alpha}_{\nu}^{\operatorname{I}}$ vanishes at the critical point with $\operatorname{arg}{x_1}\in\{0, \pi/2\}$. Analogously, a short cycle $\beta_{\nu}$ with $\nu\in\mathbb R$ decomposes into the sum $\overline{\alpha}_{\nu}^{\operatorname{I}}-\overline{\alpha}_{\nu}^{\operatorname{II}}$, where the component half-cycle $\overline{\alpha}_{\nu}^{\operatorname{I}}$ has $\operatorname{arg}I'(\nu-\varepsilon)\in\{0, \pi/2\}$ (see Sec.~\ref{Rassl}).

d) Item 2 of a) applies to both b) and c).

e) The orientations of the tubes are induced by the orientations of the half-cycles via the composition of taking the boundary and taking the tube.  

Now let us proceed to the computations.

Consider the canonically oriented cycle $\breve{\Pi}(0)=\operatorname{Re}\breve{A}_{\tilde{\lambda}}\cap X_0$. It follows from Vassiliev's jump theorem (see~\cite{VVA}, Theorem 2, p. 305) that if the critical points of the function $\breve{f}_{\tilde{\lambda}}$ in the fiber over the value $\nu\in\mathbb R$ are real, then the classes of the cycles $\breve{\Pi}(\nu-\varepsilon)$ and $\mu_{*}\breve{\Pi}(\nu+\varepsilon)$ in the homology group $H_{N-1}(\breve{A}\cap X_{\nu-\varepsilon}, \partial D)$, where $\mu(t)\equiv\{\nu+\varepsilon e^{it}, t\in[0,\pi]\}$, differ by the class of the antisymmetric cycle $\beta_{\nu}$ vanishing over the point $\nu$:
\begin{equation}\label{skachok}
	\breve{\Pi}(\nu-\varepsilon)=\mu_{*}\breve{\Pi}(\nu+\varepsilon)+i_*\beta_{\nu}(\nu-\varepsilon).
\end{equation}
If there is a pair of imaginary conjugate critical points in the fiber over a real value $\nu$, then, of course, $\breve{\Pi}(\nu-\varepsilon)=\mu_{*}\breve{\Pi}(\nu+\varepsilon)$.

From Theorem 3 of~\cite{VVA}, p. 288, and \eqref{skachok}, we obtain:
\begin{PREDLOZHENIE}
	Suppose that a short (resp. long) cycle $\beta_{\nu}$ vanishes at a real critical point (resp. a pair of real critical points) with a positive critical value $\nu$. Then
	\begin{equation}\label{Petr}
		\langle\breve{\Pi}, \beta_{\nu}\rangle=\langle\breve{\Pi}(\nu-\varepsilon), \beta_{\nu}(\nu-\varepsilon)\rangle+\sum_i\langle\beta_{\nu_i}, \beta_{\nu}\rangle,
	\end{equation}
	where the sum is taken over all real critical values between $0$ and $\nu$ that are attained at real critical points. $\Box$
\end{PREDLOZHENIE}

\begin{remark}
	Everywhere below, where a real critical value $\nu$ is mentioned, we will assume $\nu>0$. The case $\nu<0$ is studied analogously (see Theorem 3 of~\cite{VVA}, p. 288). 
\end{remark}
	
	From Theorem 1, p. 304 of~\cite{VVA}, it follows that:
	\begin{equation}\label{q}
		\langle\breve{\Pi}(\nu-\varepsilon), \beta_{\nu}(\nu-\varepsilon)\rangle=(-1)^{\frac{(N-2)(N-1)}{2}}\chi(S^{N-r-1}),
	\end{equation} 
	if $\beta_{\nu}$ is a short cycle, and 
	\begin{equation}\label{w}
		\langle\breve{\Pi}(\nu-\varepsilon), \beta_{\nu}(\nu-\varepsilon)\rangle=2(-1)^{\frac{(N-2)(N-1)}{2}}\chi(S^{N-r-1}),
	\end{equation}
	if $\beta_{\nu}$ is a long cycle. Here, $\chi$ is the Euler characteristic, and $r$ is the positive Morse index (see \eqref{vid}) of the corresponding real critical point (or of any of the pair of real critical points, whose indices obviously coincide).
	
	Let us now return to the original problem --- the computation of the class of the cycle $\Pi=\operatorname{Re}A\cap X$ in the homology group $H_{N-1}(A\cap X\setminus S, \partial D)$.
	
	\begin{lemm}\label{Petra}
		Suppose that the critical value $\nu$ of the function $f_{\tilde{\lambda}}$ is positive. Then 
		\begin{equation}\label{sk}
			[\Pi, \alpha_{\nu}]=\langle\Pi_{ev}(\nu-\varepsilon), \alpha_{\nu}(\nu-\varepsilon)\rangle_{A\cap X_{\nu-\varepsilon}}+\sum_i[\alpha_{\nu_i}, \alpha_{\nu}],
		\end{equation}
		where $\Pi_{ev}(\nu-\varepsilon)$ is the class of the cycle $\operatorname{Re}A\cap X_{\nu-\varepsilon}$ in the homology group $H_{N-1}(A\cap X_{\nu-\varepsilon}, \partial D)$, the summation runs over all real critical values of the function $f_{\tilde{\lambda}}$ between $0$ and $\nu$, and the bilinear form $[\cdot,\cdot]$ is defined in Proposition~\ref{bilin} on p.~\pageref{bilin}. The value $\langle\Pi_{ev}(\nu-\varepsilon), \alpha_{\nu}(\nu-\varepsilon)\rangle_{A\cap X_{\nu-\varepsilon}}$ is computed by Theorem 1, p. 304 of~\cite{VVA} from the positive Morse index of the critical point with value $\nu$.
	\end{lemm}
	
	\textbf{Proof.} Since the value $[\Pi, \alpha_{\nu}]=\langle\Pi, \alpha_{\nu}\rangle_{A\cap X}$ depends only on the classes of these cycles in the groups $H_{N-1}(A\cap X, \partial D)$ and $H_{N-1}(A\cap X)$ (see~\eqref{cd}), the boundary can be neglected, and formula \eqref{sk} follows immediately from formula VI.5, p. 288 of~\cite{VVA}. 
	
	In other words, consider the transport of the cycle $\Pi$ along the distinguished path $\gamma_{\nu}$ to the fiber $X_{\nu-\varepsilon}.$ Let $\hat{\nu}<\nu$ be the smallest positive critical value of the function $f_{\tilde{\lambda}}|_{\hat{S}}$. Transports over the critical values $\nu_i$ of the function $f_{\tilde{\lambda}}$ for $\nu_i<\hat{\nu}$ are described by Vassiliev's jump theorem (\cite{VVA}, Theorem 2, p. 305), and the transport of the cycle $\Pi( \hat{\nu}-\varepsilon)$ along the arc $\mu^{-1}$ over the point $\hat{\nu}$ can be chosen so that the result coincides with the set $\operatorname{Re}A\cap X_{\hat{\nu}+\varepsilon}$ outside a tubular neighborhood of $A\cap X_{\hat{\nu}+\varepsilon}\cap S$ in $A\cap X_{\hat{\nu}+\varepsilon}.$ From this, it follows that $\langle\mu^{-1}_{*}\Pi(\hat{\nu}-\varepsilon), \alpha_{\nu}(\hat{\nu}+\varepsilon)\rangle_{A\cap X_{\hat{\nu}+\varepsilon}\setminus S}=\langle\Pi_{ev}(\hat{\nu}+\varepsilon), \alpha_{\nu}(\hat{\nu}+\varepsilon)\rangle_{A\cap X_{\hat{\nu}+\varepsilon}}$, and the statement follows by induction. $\Box$ 
	
	\begin{PREDLOZHENIE}\label{Petra'}
		Let $\nu$ and $\bar{\nu}$ be two complex conjugate critical values of the function $f_{\tilde{\lambda}}$. Then 
		\begin{equation}		
			|[\Pi, \alpha_{\nu}]|=|[\Pi, \alpha_{\bar{\nu}}]|\equiv[\alpha_{\nu}, \alpha_{\bar{\nu}}] \pmod 2.
		\end{equation}
	\end{PREDLOZHENIE}
	\textbf{Proof.} The statement follows from the proof of Theorem~\ref{Petra} and item 3.3 on p. 308 of~\cite{VVA}. $\Box$
	
	Now let $\nu$ be a critical value of the function $f_{\tilde{\lambda}}|_{\hat{S}}$. For the computations, we need to pass to the two-sheeted covering of the manifold $A\cap X$ branched along the boundary $\hat{S}:x_1=0$. Let us decompose the cycle $\Pi$ into the union $\Pi=\Pi^+\cup\Pi^-$, where $\Pi^+=\Pi\cap\{x_1>0\}$ and $\Pi^-=\Pi\cap\{x_1<0\}$. We define the cycle $\sigma^{-1}\Pi^+$ as the full preimage in the covering space, taken there with its canonical orientation. Then $\sigma^{-1}\Pi^+=\breve{\Pi}$ and $[\Pi^+,\alpha_{\nu_i}]\equiv 1/2\langle\breve{\Pi},\beta_{\nu_i}\rangle \pmod 2$ hold, see the proof of Proposition~\ref{pepe}.
	
	Let us introduce the following notation (see Fig.~\ref{Pep}):
	\begin{eqnarray}
		f_{\tilde{\lambda}}^+\equiv f_{\tilde{\lambda}},\\
		f_{\tilde{\lambda}}^-\equiv f_{\tilde{\lambda}}(-x_1,x_2,\dots,x_N),\\ 
		\breve{f}^{\pm}_{\tilde{\lambda}}\equiv f_{\tilde{\lambda}}^{\pm}(x_1^2,x_2,\dots,x_N),\\
		\label{p1}\breve{\Pi}_{\pm}\equiv\operatorname{Re}\{\breve{f}^{\pm}_{\tilde{\lambda}}=0\}\cap D,\\
		\Pi^{\pm}_{\pm} - \text{the cycle }\Pi^{\pm}\text{ defined for the function } f_{\tilde{\lambda}}^{\pm}.\label{p2}
	\end{eqnarray} 
	We also denote by $\alpha_{\nu_i}^{\pm},\overline{\alpha}_{\nu_i}^{\pm},\beta_{\nu_i}^{\pm}$ the corresponding basic cycles (defined for the functions $f_{\tilde{\lambda}}^{\pm}$) taken with their canonical orientation.
	
	\begin{figure}[H]
		\center{\includegraphics[width=0.4\linewidth]{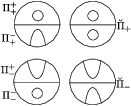}}
		\caption{Scheme of the cycles \eqref{p1}--\eqref{p2} for the function $f_{\tilde{\lambda}}=x_2^2+(x_1-\varepsilon)(x_1-2\varepsilon)(x_1+\varepsilon)$}
		\label{Pep}	
	\end{figure}
	
	\begin{remark}\label{zam}
		The involution $\rho:x_1\mapsto-x_1$ of the space $\mathbb{C}^N$ identifies the level sets of the functions $f^{+}_{\tilde{\lambda}}$ and $f^{-}_{\tilde{\lambda}}$ (preserving their complex orientations); however, the orientations of the images of the cycles of interest to us diverge from the canonical ones: $\rho_*(\overline{\alpha}_{\nu_i}^+)=-\overline{\alpha}_{\nu_i}^-$ for $\nu_i\in\mathbb{R}$ and $\rho_*(\Pi_+)=-\Pi_-.$
	\end{remark} 
	
	\begin{remark}\label{zamech}
		The level sets of the functions $\breve{f}^{+}_{\tilde{\lambda}}$ and $\breve{f}^{-}_{\tilde{\lambda}}$ are identified by the diffeomorphism $x_1\mapsto ix_1$ (or $x_1\mapsto -ix_1$) of the space $\mathbb{C}^N$. If a long cycle $\beta_{\nu_i}^+$ vanished at real critical points, then the cycle $\beta_{\nu_i}^-$ vanishes at imaginary conjugate points (and vice versa), while the critical points of short cycles remain in place.	
	\end{remark}
	
	\begin{PREDLOZHENIE}\label{chet}
		Let $\nu$ be a real critical value of the function $f_{\tilde{\lambda}}|_{\hat{S}}$. Then
		\begin{equation}\label{ol}
			[\Pi,\alpha_{\nu}]\equiv 1/2(\langle\breve{\Pi}_{+},\beta_{\nu}^+\rangle+\langle\breve{\Pi}_{-},\beta_{\nu}^-\rangle) \pmod 2.
		\end{equation}
	\end{PREDLOZHENIE}
	\textbf{Proof.} $[\Pi,\alpha_{\nu}]=[\Pi^+_{+}+\Pi^-_{+},\alpha_{\nu}^+]=[\Pi^+_{+},\alpha_{\nu}^+]+[\Pi^-_{+},\alpha_{\nu}^+]=\langle\Pi^+_{+},\overline{\alpha}^{+}_{\nu}\rangle+\langle\Pi^-_{+},\overline{\alpha}^{+}_{\nu}\rangle\equiv 1/2\langle\sigma^{-1}\Pi^+_{+},\sigma^{-1}\overline{\alpha}^{+}_{\nu}\rangle+\langle\rho_*\Pi^-_{+},\rho_*\overline{\alpha}^{+}_{\nu}\rangle=1/2\langle\breve{\Pi}_{+},\beta_{\nu}^+\rangle+\langle-\Pi^+_{-},-\overline{\alpha}_{\nu}^-\rangle=1/2\langle\breve{\Pi}_{+},\beta_{\nu}^+\rangle+\langle\Pi^+_{-},\overline{\alpha}_{\nu}^-\rangle\equiv 1/2\langle\breve{\Pi}_{+},\beta_{\nu}^+\rangle+1/2\langle\sigma^{-1}\Pi^+_{-},\sigma^{-1}\overline{\alpha}_{\nu}^-\rangle=1/2(\langle\breve{\Pi}_{+},\beta_{\nu}^+\rangle+\langle\breve{\Pi}_{-},\beta_{\nu}^-\rangle) \pmod 2$. $\Box$
	
	Let us denote $r^+\equiv\langle\breve{\Pi}_{+}(\nu-\varepsilon),\beta_{\nu}^+(\nu-\varepsilon)\rangle$ and $r^-\equiv\langle\breve{\Pi}_{-}(\nu-\varepsilon),\beta_{\nu}^-(\nu-\varepsilon)\rangle.$ The numbers $r^{\pm}$ are computed by formula \eqref{q}.
	
	\begin{lemm}\label{Petra1} 
		Suppose that the critical value $\nu$ of the function $f_{\tilde{\lambda}}|_{\hat{S}}$ is positive. Then
		\begin{equation}\label{form}
			[ \Pi, \alpha_{\nu} ] \equiv \frac{r^+ + r^-}{2} + \sum_i [ \alpha_{\nu_i}, \alpha_{\nu} ] \pmod 2,
		\end{equation}
		where the bilinear form $[\cdot,\cdot]$ is defined in Proposition~\ref{bilin} on p.~\pageref{bilin}, and the sum runs over all real critical values of the functions $f_{\tilde{\lambda}}$ and $f_{\tilde{\lambda}}|_{\hat{S}}$ from $0$ to $\nu$.
	\end{lemm}
	
	\textbf{Proof.} By formula \eqref{Petr}, we have:
	\begin{equation}
		\langle\breve{\Pi}_{+}, \beta_{\nu}^+\rangle = r^+ + \sum_i \langle \beta^+_{\nu_i}, \beta^+_{\nu} \rangle,
	\end{equation}
	\begin{equation}
		\langle\breve{\Pi}_{-}, \beta_{\nu}^-\rangle = r^- + \sum_i \langle \beta^-_{\nu_i}, \beta^-_{\nu} \rangle,
	\end{equation}
	where both sums are taken over the real critical values from $0$ to $\nu$ attained at real critical points. Taking into account Remarks~\ref{zam} and \ref{zamech}, from Proposition~\ref{chet} we obtain the formula:
	$$[ \Pi, \alpha_{\nu} ] \equiv \frac{r^+ + r^-}{2} + 1/2 \sum_j \langle \beta_{\nu_j}, \beta_{\nu} \rangle + \sum_i \langle \beta_{\nu_i}, \beta_{\nu} \rangle \pmod 2.$$ 
	It remains to apply Proposition~\ref{pepe}. $\Box$
	
	\begin{remark}
		If Lemma~\ref{hard} holds, the congruences modulo $2$ in \eqref{ol} and \eqref{form} can be replaced by exact equalities, see also Remark~\ref{auth}.
	\end{remark}
	
	\begin{Relation}\label{PuRe}
		Let $f$ be a real purely boundary singularity, and let the function $f|_{\hat{S}}$ have a minimum at the origin. We choose a perturbation $f_{\tilde{\lambda}}$ close to the function $f+\varepsilon$ and satisfying the condition $\operatorname{Re}\{f_{\tilde{\lambda}}|_{\hat{S}}=0\}\cap D = \varnothing.$ Then, for any real critical value $\nu$ of the function $f|_{\hat{S}}$, we have:
		\begin{equation}
			[ \Pi, \alpha_{\nu} ] = \frac{r^+ + r^-}{2} + \sum_i [ \alpha_{\nu_i}, \alpha_{\nu} ].
		\end{equation}
	\end{Relation}
	
	\textbf{Proof.} By Corollary~\ref{Pure}, for arbitrary short cycles $\beta_{\nu_i}$ and $\beta_{\nu_j}$, we have $\langle\beta_{\nu_i}, \beta_{\nu_j}\rangle = \operatorname{sgn}(j-i)(-1)^{N-1}\langle\partial\overline{\alpha}_{\nu_i}, \partial\overline{\alpha}_{\nu_j}\rangle$. The path $\gamma_{\nu}$ has the smallest index under the canonical numbering in decreasing order of the arguments of entry into the origin. Then $\sum_i \langle \beta_{\nu_i}, \beta_{\nu} \rangle = (-1)(-1)^{N-1}\sum_i \langle \partial\overline{\alpha}_{\nu_i}, \partial\overline{\alpha}_{\nu} \rangle = (-1)^N\sum_i \langle \partial\overline{\alpha}_{\nu_i}, \partial\overline{\alpha}_{\nu} \rangle = \sum_i [ \alpha_{\nu_i}, \alpha_{\nu} ]$ by Definition~\ref{intmatr}. $\Box$
	
	\begin{PREDLOZHENIE}
		Let $\nu$ and $\overline\nu$ be two complex conjugate critical values of the function $f_{\tilde\lambda}|_{\hat S}$. Then:
		\begin{equation}
			|[ \Pi, \alpha_{\nu} ]| = |[ \Pi, \alpha_{\bar{\nu}} ]| = |\langle \overline{\alpha}_{\nu}, \overline{\alpha}_{\overline\nu} \rangle_{A\cap X\setminus S}|.
		\end{equation}
	\end{PREDLOZHENIE}
	
	\textbf{Proof.} The first equality follows from the invariance of the cycle $\Pi$ under complex conjugation. The half-cycles $\overline{\alpha}_{\nu}$ and $\overline{\alpha}_{\overline\nu}$ can be defined by complex conjugate paths $I'_{\nu}$ and $I'_{\overline\nu}$ in $\mathbb C_{x_1}$ entering the origin (see Sec.~\ref{Rassl}), which intersect the $\operatorname{Re}$ axis transversally (and possibly originate on it). In general position, the tangent spaces to the half-cycles at their intersection points outside $x_1=0$ do not contain any straight lines from $\mathbb R^N$, which implies the second equality. $\Box$
	
	\subsection{Quasilacunas of versal deformations of simple boundary singularities}\label{quasisimpl}
	
	In this subsection, we enumerate the quasilacunas for the classification of simple real boundary singularities introduced by V.~I.~Arnold (see~\cite{ARN}). 
	
	In the case where $df(0)=0$, it is sufficient to consider the following germs $f(x,y)$ and their miniversal deformations $F(x,y,\lambda)$ (and their stabilizations):
	\begin{table}[h]
		\centering
		\begin{tabular}{|l|l|l|}
			\hline
			$B_{k}^\pm$ & $x^{k}\pm y^2$ & $x^{k}\pm y^2+\lambda_{k-1}x^{k-1}+\lambda_{k-2}x^{k-2}+\dots+\lambda_{1}x+\lambda_{0}$ \\ \hline
			$C_{k}^\pm$ & $xy\pm y^{k}$ & $xy\pm y^{k}+\lambda_{k-1}y^{k-1}+\lambda_{k-2}y^{k-2}+\dots+\lambda_{1}y+\lambda_{0}$ \\ \hline
			$F_{4}^{\pm}$ & $\pm x^2+y^3$ & $\pm x^2+y^3+\lambda_{3}x+\lambda_{2}y+\lambda_{1}xy+\lambda_{0}$ \\ \hline  
		\end{tabular}
		\label{tablo}
	\end{table}
	\noindent The boundary $\hat{S}$ is everywhere defined by the condition $x=0$. 
	
	\begin{lemm}\label{l3} 
		Let $f_{\tilde\lambda}$ be an arbitrary real perturbation of a singularity stably equivalent to any of the singularities from list \eqref{tablo} except for $B_2^+$. Then each connected component of the set $\operatorname{Re}A_{\tilde{\lambda}}\cap X$ that does not intersect $\hat{S}$ realizes a non-zero class $\Pi_{i}\in H_{N-1}(A_{\tilde\lambda}\cap X, \partial D)$.
	\end{lemm}	
	
	\textbf{Proof.} see in~\cite{NIKIT} (Lemma 1). $\Box$
	
Thus, for all cases except for $B_2^+$, all quasilacunas are trivial (non-strict in the sense of Def.~\ref{slow}): a component $\Gamma$ is a quasilacuna if and only if $\{\Pi_i\}=\varnothing$ or $\operatorname{Re}A\cap X=\varnothing$.

	We denote the $r,s$-stabilization of the singularity $f$ by $f^{r,s}=f+Q$, where $Q(z)=\sum_{i=1}^{r}z_{i}^{2}-\sum_{j=r+1}^{r+s}z_{j}^{2}$ is a non-degenerate quadratic form of additional variables. 
	
	\begin{theorem}\label{perech}
		The numbers of quasilacunas for the miniversal deformations of simple boundary singularities of two variables from list \eqref{tablo} (if necessary, we replace the germ and its deformation by equivalent ones so that Proposition~\ref{parab} holds) are given in the following table:
		\begin{center}
			\begin{tabular}{|l|l|l|}
				\hline
				Name & Number of components & Number of quasilacunas \\ \hline
				$B_{2}^{+}$ & $4$ & $4$ \\ \hline
				$B_{2k}^{+}, k>1$ & $(k+1)^2$ & $2$ \\ \hline
				$B_{2k}^{-}$ & $(k+1)^2$ & $1$ \\ \hline
				$B_{2k+1}^{\pm}$ & $(k+1)(k+2)$ & $1$ \\ \hline  
				$C_{2k}^{\pm}$ & $(k+1)^2$ & $k^2$ \\ \hline 	
				$C_{2k+1}^{\pm}$ & $(k+1)(k+2)$ & $k^2+k$ \\ \hline 
				$F_{4}^\pm$ & $8$ & $4$ \\ \hline 
			\end{tabular}
		\end{center} 
		
		In the case of $B_{2}^+$, the answer is 2 for stabilizations with $s=0$ for odd $N$, and for all other stabilizations for any $N$, the answer is 4.
		
		In the cases of $B_{k}^+ (k\neq 2)$, $F_{4}^+$, $C_{2k}^+$ (respectively, $B_{k}^-$, $F_{4}^-$, $C_{2k}^-$), for all stabilizations with $s=0$ (respectively, $r=0$), the answer is the same as for $N=2$. In the case of $C_{2k+1}^{\pm}$, for all stabilizations with $r=0$ or $s=0$, the number of quasilacunas is equal to $(k+1)^2$. 
		
		In all cases, for any stabilization not specified above, all local components of the complement to $\Sigma_{\mathbb{R}}(F)$ are quasilacunas. 
	\end{theorem}
	
	\textbf{Proof.} see in~\cite{NIKIT} (Theorem 1). $\Box$
	
	\begin{remark}\label{ZaM}
		The numbers of quasilacunas for $\mathcal{K}_{\hat S}$-equivalent singularities coincide (see Proposition~\ref{K-K}): the transformation $x\mapsto-x$ in composition with multiplication by $-1$ determines a $\mathcal{K}_{\hat S}$-equivalence of the germs $B_{2k+1}^+$ and $B_{2k+1}^-$ ($C_{k}^+$ and $C_{k}^-$) in their normal form, and the germs $F_4^{+}$ and $F_4^{-}$ are $\mathcal{K}_{\hat S}$-equivalent via $y\mapsto-y$ and $\cdot(-1)$. 
	\end{remark}
	
	\begin{figure}[H]
		\center{\includegraphics[width=1\linewidth]{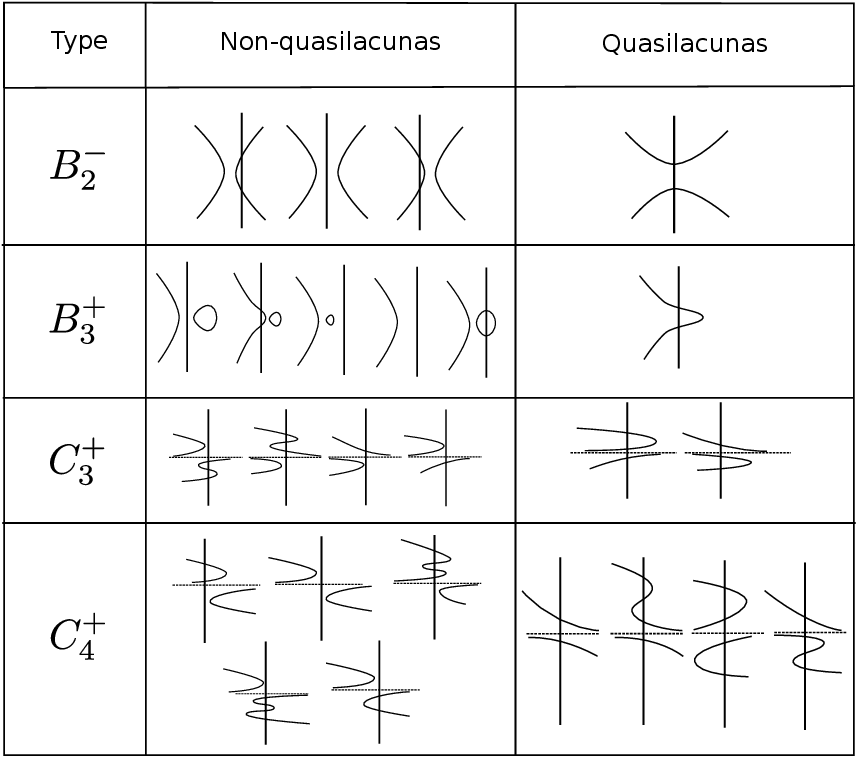}}
		\caption{Examples}
		\label{F18}	
	\end{figure}
	
	The pictures of generic fibers for $F_{4}$ can be found in the work of M.~A.~Gudiev~\cite{MAG}.
	
	The classification of purely boundary simple singularities reduces to the classification of ordinary simple singularities (see~\cite{Lyash}). The normal form of the germ $f$ in this case is written as $f=x_{1}+\varphi(x_{2},\dots,x_{N}),$ where $\varphi$ is the normal form (or its stabilization) of one of the following singularities: $A_{k}, D_{k}, E_{6}, E_{7}, E_{8}$. From Proposition~\ref{Pred7} on p.~\pageref{Pred7} and Lemma~\ref{simpl}, it follows that:
	
	\begin{PREDLOZHENIE}\label{dfneqo}
		For a miniversal deformation of the form \eqref{miniversal} of the germ $x\pm((y-\varepsilon x^2)^{2k}+Q)$ of a purely boundary singularity $A_{2k-1}$ with the boundary $x=0$, where $Q$ is a non-degenerate quadratic form of an odd number of variables with $s=0$, the unique component is not a quasilacuna (see Example~\ref{A2k}, Fig.~\ref{F0}). For all deformations of the form \eqref{obsch} of any other simple purely boundary singularity, all components are quasilacunas. For even $N$, an analogous component represents the unique non-trivial quasilacuna. $\Box$
	\end{PREDLOZHENIE}
	
	\begin{lemm}\label{ust1}
		The numbers of quasilacunas for linear deformations of stable germs of types $B_k, C_k, F_4$ coincide with the numbers given in Theorem~\ref{perech}. 
	\end{lemm}
	
	\textbf{Proof.} Let $N=\mu-1$, where $\mu$ is the Milnor number. Then the linear $\mu$-parameter deformation of a stable germ is its miniversal deformation in the form \eqref{miniversal}, and the statement of the lemma follows from Proposition~\ref{K-K} on p.~\pageref{K-K}. If $N>\mu-1$, then the stable wave front in $\mathbb{R}^{N+1}$ is diffeomorphic to the product of the front in $\mathbb{R}^{\mu}$ by $\mathbb{R}^{N-\mu+1}$, and the proof reduces to the case considered. $\Box$ 
	
	In conclusion, let us consider the generalized even Petrovsky cycle $\Pi$.
	
	\begin{lemm}\label{lala}
		If, for a real perturbation of a simple boundary singularity with $df(0)=0$ (except for $B_2^+$), the cycle $\Pi$ is defined and non-empty, then $\Pi\neq0\in H_{N-1}(A\cap X, \partial D)$.
	\end{lemm}
	
	\textbf{Proof.} It is sufficient to study certain morsifications of $B_k$ and the unique $C_{2k}$. In each case, for any $N>2$, there exists a unique stabilization (preserving the number of connected components) for which the cycle $\Pi$ is defined. If the cycle $\Pi$ is connected, then everything is already proved; if not, it is easy to specify a path in the component of the complement to $\Sigma_{\mathbb R}^f$ (running to a generic point of its boundary) along which two components of the cycle $\Pi$ collide. It remains to apply the corollary to Proposition~\ref{DB}. $\Box$
	
	\section{Ramification close to a tangent hyperplane}\label{kasal}
	Here $f, A, X, D, B, \text{\CYRSH}_j$ are the same as in Sec.~\ref{vetv}, except that there is no boundary and $df(0)=0$ is assumed to hold. The homomorphism $H_{N+1}(D, A\cup X\cup\partial D)\to H_{N-1}(A\cap X, \partial D)$ (cf.~\eqref{1}) is an epimorphism, and $H_{N+1}(D, A\cup X)\to H_{N-1}(A\cap X)$ is an isomorphism, see~\cite{VVA}, Prop. 3, p. 122. Let us denote by $\Omega_i$ any local integration contour (i.e., an oriented connected domain in $\mathbb R^{N+1}\cap D$ whose boundary lies in $\partial D\cup A\cup X$), by $\Pi_i$ its corresponding element of the group $H_{N-1}(A\cap X, \partial D)$, and by $V_i$ the volume function $\int_{\Omega_i(X)}\omega$.
	
	\begin{lemm}\label{VV} 
		Let $N$ be even. If $\Pi_i\neq0\in H_{N-1}(A\cap X, \partial D)$, then the function $V_i$ is infinitely ramified in $B$.
	\end{lemm}
	
	This statement is a special case of Vassiliev's theorem, see~\cite{VVA1}, Theorem 1. 
	
	\textbf{Proof.} By Poincaré duality, there exists a vanishing cycle $\Delta_j\in H_{N-1}(A\cap X)$ such that $\langle\Pi_i, \Delta_j\rangle\neq0$. Then, by the (generalized) Picard--Lefschetz formula, an $r$-fold bypass along a loop around the critical value for the cycle $\Delta_j$ adds the cycle $\text{\CYRSH}_j$ taken $(-1)^{\frac{N(N+1)}{2}} r\langle\Pi_i, \Delta_j\rangle$ times to the cycle $\Omega_i$. The statement now follows from Lemma~\ref{Vol1}. $\Box$  
	
	Now let $N$ be odd. In the case of simple singularities, the quadratic intersection form on the group $H_{N-1}(A\cap X)$ is sign-definite (\cite{Norm}). In this case, the monodromy group --- the image of the representation $\pi_1(B\setminus\Sigma)\rightarrow O(H_{N-1}(A\cap X), \langle\cdot,\cdot\rangle)$ --- is finite, and consequently, the orbit of the cycle $\Pi_i$ is also finite (cf. Lemma~\ref{simpl}).
	
	\begin{theorem}\label{Th}
		Let $N$ be odd, and let both inertia indices of the quadratic intersection form on the group $H_{N-1}(A\cap X)$ be non-zero. If $\Pi_i\neq0\in H_{N-1}(A\cap X, \partial D),$ then the function $V_i$ is infinitely ramified in $B$.
	\end{theorem}
	
	\textbf{Proof.} Let us quotient the group $H_{N+1}(D, A\cup X; \mathbb R)\simeq\mathbb{R}^{\mu(f)}$ formed by the classes of vanishing integration contours $\{\text{\CYRSH}_j\}$ by the subgroup of all elements yielding an identically zero integral. The monodromy group acts on this quotient group, preserving the descended intersection form (induced by the isomorphism $H_{N+1}(D, A\cup X)\simeq H_{N-1}(A\cap X)$). 
	
	This group is generated by reflections in the hyperplanes orthogonal to the cosets of the vanishing contours. Since it is irreducible (Corollary 6, p. 45 of~\cite{VVA}), according to the classification of groups generated by reflections (see~\cite{Burb}), it is finite only if the quadratic intersection form is sign-definite. It follows from the assumption that the analytic continuation of the function $\int_{\text{\CYRSH}_j(X)}\omega$ takes infinitely many values at the point $X_0$ for any vanishing cycle $\text{\CYRSH}_j$. 
	
	Suppose now that the function $V_i$ is finite-sheeted in $B$. By definition, this means that the integral of the form $\omega$ takes only a finite number of values on the set of variations of the cycle $\Omega_i$. Consequently, the integral also takes only a finite number of values on the set of differences between the elements of the orbit of the cycle $\Omega_i$.
	
	However, if $\langle\Pi_i, \Delta_j\rangle\neq0$, then by the generalized Picard--Lefschetz formula, the orbit of the cycle $(-1)^{\frac{N(N+1)}{2}}\langle\Pi_i, \Delta_j\rangle\text{\CYRSH}_j$ is contained in the specified set of differences. Since the integral takes infinitely many values on the elements of this orbit, we arrive at a contradiction. $\Box$
	
	The case of {\em parabolic} singularities, whose intersection form is semi-definite, remains uncovered. In this case if $\Pi_i\neq0\in H_{N-1}(A\cap X, \partial D)$ then the function $V_i$ is finite-sheeted $\iff$ the kernel of the intersection form on the group $H_{N+1}(D, A\cup X)$ coincides with the subgroup of elements yielding an identically zero integral.
	
	\begin{example}\label{Para}
		For the tabular versal deformation of a singularity of type $\pm X_9$, there exists a component of the complement to $\Sigma_{\mathbb R}$ which is a non-trivial formal lacuna (i.e., $\Pi\neq\varnothing$ and $\Pi=0\in H_{N-1}(A\cap X, \partial D)$), see Example~\ref{newex}. The cycle $\Pi$ realizes a non-zero element of the group $H_{N-1}(A\cap X)$ because it corresponds to a class $\Omega\in H_{N+1}(D, A\cup X)$ such that $V(X)=\int_{\Omega(X)}\omega\neq 0$. On the other hand, $\Pi$ lies in the kernel of the intersection form.
	\end{example}
	
	\section{Obstructions to the algebraic integrability of domains in $\mathbb{R}^{N+1}$}\label{integr}
	
	In this section, we formulate obstructions to the algebraicity of the volume function, which follow from the obtained results on ramification close to an asymptotic hyperplane. Let $\Xi\subset\mathbb R^{N+1}$ be an algebraic hypersurface, and let $\Upsilon$ be an arbitrary connected component of the set $\mathbb R^{N+1}\setminus\Xi$. Suppose that at least one of the two parts (possibly empty) into which a generic hyperplane $X_0$ divides the domain $\Upsilon$ does not extend to infinity. Then, near $X_0$, the volume function $V:(\mathbb{RP}^{N+1})^*\to\mathbb R$ is defined by the formula $V(X)=\int_{\Omega(X)}\omega$, where $\Omega(X)$ is the corresponding $(N+1)$-chain (for a compact $\Upsilon$ we choose either of the two) equipped with the positive orientation of $\mathbb{R}^{N+1}$, and $\omega$ is the Euclidean volume form. The domain $\Upsilon$ is called {\em locally algebraically integrable} near $X_0$ if the function $V$ is defined and is a branch of an algebraic function. A domain is called {\em algebraically integrable} if it is locally algebraically integrable for any $X$ for which the function $V(X)$ is defined, and if the germ of $V$ is non-zero at least at one point.
	
	\subsection{Integrability and quasilacunas}\label{intk}
	
	Let $a\in\overline{\partial\Upsilon}\cap S$ be a {\em simple} point (see Sec.~\ref{Pro}), let $f$ be the generating function for $(\overline\Xi, \overline\Xi\cap S)$ at this point, and let $X$ be a generic hyperplane close to the asymptotic one.
	
	The following proposition follows from the definitions, see Definition~\ref{slow}. 
	
	\begin{PREDLOZHENIE}
		If $\Upsilon$ is locally algebraically integrable near $X$, then the local component of the complement to $\Sigma_{\mathbb R}(\Lambda_f)$ containing $X$ is a strict weak quasilacuna.  
	\end{PREDLOZHENIE}
	
	\begin{theorem}\label{Vasvas}
		Let $\Upsilon$ be locally algebraically integrable near $X$, and let the boundary singularity of the generating function $f$ be stable and simple. Then this singularity can only belong to the class $\pm B_{2}^+(\pm(x^{2}+y^2+Q))$ or $\pm A_{2k-1}(x\pm (y^{2k}+Q))$, where the boundary is $x=0$, $Q$ is a non-degenerate quadratic form of an even number of additional variables with $i_-(Q)=0$. The numbers of strict weak quasilacunas are given below:
		\begin{table}[h]
			\centering
			\begin{tabular}{|l|l|}
				\hline
				$B_{2}^{+}$ & $2$ \\ \hline
				$A_{2k-1}$  & $1$ \\ \hline
			\end{tabular}
		\end{table}
	\end{theorem}
	
	\textbf{Proof.} The statement follows from Lemmas~\ref{ust1}, \ref{lala} and the analysis of the remaining cases, see Examples~\ref{b2}, \ref{A2k} and Fig.~\ref{F0}. $\Box$
	
	We call a real irreducible hypersurface $\Xi\subset\mathbb{C}^{N+1}$ \textit{algebraically integrable} if at least one connected component $\Upsilon$ of the set $\mathbb{R}^{N+1}\setminus\Xi$ is algebraically integrable.
	
	\begin{theorem}\label{Vas}
		If the hypersurface $\Xi$ is algebraically integrable, then: 
		a)~for even $N$, there are no non-degenerate parabolic simple points on $\overline{\Xi}$; 
		b)~for odd $N$, there are no non-degenerate simple points on $\overline{\Xi}$ at which the germ of the generating function for the pair $(\overline\Xi, \overline\Xi\cap S)$ is stable (or at least satisfies Proposition~\ref{parab}). 
	\end{theorem}
	
	\begin{remark}
		This theorem is an analogue of Theorem 1, p. 188 of~\cite{VVA}. For odd $N$, the statement follows immediately from Proposition~\ref{parab} and Theorem 9, p. 181 of~\cite{VVA} stating that the closure of an integrable hypersurface does not contain {\em trivial} points (of type $B_2$).
	\end{remark}
	
	\textbf{Proof.} Suppose the contrary. Then, for a generic hyperplane $X$ close to the tangent hyperplane to $\overline{\Xi}$ at a simple point $a$, let us take an arbitrary vanishing cycle $\Delta$ of the group $H_{N-1}(D(a)\cap\overline{\Xi}\cap X)\equiv H_{N-1}(A\cap X)$. By Lemma~\ref{baza}, $\Delta$ represents a basic element of the group $H_{N-1}(A\cap X\setminus S)$. It follows from Lemma~\ref{monodromia}, Proposition~\ref{Izol}, and the proof of Theorem~\ref{o} that the orbit of the cycle $\Delta$ under the action of the local monodromy is infinite, and the analytic continuation of the function $\int_{\text{\CYRSH}(X)}\omega$ is infinitely-valued for the corresponding element $\text{\CYRSH}\in H_{N+1}(D\setminus S, A\cap X\setminus S)$. Let us denote by $X'$ a hyperplane very close to the critical one for the cycle $\Delta.$
	
	Let $X_{c'}$ be a point from the domain of definition of the function $V(X)$ corresponding to a tangency at a non-singular non-parabolic point on $\partial\Upsilon$. Let us consider a line $L'$ in $(\mathbb{RP}^{N+1})^*$ consisting of hyperplanes in the space $\mathbb{R}^{N+1}$ parallel to $X_{c'}$. The volume function is defined for both hyperplanes $X_{c'\pm\varepsilon}.$ Let us denote by $x_{c'}\in\Xi$ the point of tangency of $X_{c'}$ and $\Xi$. Then, by Vassiliev's jump theorem, the classes of the two integration contours in the group $H_{N+1}(\mathbb{C}^{N+1}, \Xi\cup X_{c'-\varepsilon})$ --- one of which is equal to $\Omega(X_{c'-\varepsilon})$, and the other is obtained by the Gauss--Manin transport of $\Omega(X_{c'+\varepsilon})$ along the path $\mu(\varphi)\equiv\{c'+\varepsilon e^{i\varphi}, \varphi\in[0,\pi]\}$ on the complexified line $L'$ --- differ by the class of the generator of the group $H_{N+1}(O(x_{c'}), O(x_{c'})\cap(\Xi\cup X_{c'-\varepsilon}))$ realized by a vanishing contour $\text{\CYRSH}_{c'}$. Let us consider the function $W(X)=V(X)-V_{\mu}(X)$ near the point $X_{c'-\varepsilon}$, where the function $V_{\mu}(X)$ is obtained by the analytic continuation of the function $V$ defined near $X_{c'+\varepsilon}$ along the path $\mu.$ From the integrability condition, it follows that the function $W(X)$ is algebraic. The value of the function $W$ at the point $X_{c'-\varepsilon}$ is equal to the oriented volume of the vanishing contour $\text{\CYRSH}_{c'}$ (see Lemma~\ref{Vol1}).
	
	Since the set of non-singular non-parabolic points of an irreducible hypersurface $\Xi$ is path-connected, there exists a path $\gamma\subset\operatorname{Reg}(\Xi\cup S)$ running near $\operatorname{tang}(\Xi)$ (see~\cite{VVA}, Ch.~III,~\S2) that connects the hyperplanes $X_{c'-\varepsilon}$ and $X'$. Under this transport, the element $\text{\CYRSH}_{c'}$ transforms into the element $\pm\text{\CYRSH}$. Considering the analytic continuation of the function $W$ along the path $\gamma$ and then along the orbit specified above, we arrive at a contradiction. $\Box$

\end{document}